%% file: main.tex
\documentclass[3p,computermodern,10pt]{elsarticle}
\input{packages}
\input{commands}

\input{commands_cwst}
\begin{document}
\begin{frontmatter}

\title{Windowed least-squares model reduction for dynamical systems}

\author[a]{Eric J. Parish and Kevin T. Carlberg}

\address[a]{Sandia National Laboratories,  Livermore, CA}
\begin{abstract}
This work proposes a \methodNameLower\ (\methodAcronym) approach for model-reduction
	of dynamical systems. The proposed approach sequentially minimizes the
	time-continuous full-order-model residual within a low-dimensional space--time trial
	subspace over time windows. The approach comprises a generalization 
of existing model reduction approaches, as particular instances of
  the methodology recover Galerkin,
	least-squares Petrov--Galerkin (LSPG), and space–time LSPG projection. In
	addition, the approach
	addresses key deficiencies in existing model-reduction
	techniques, e.g., the dependence of LSPG and space--time LSPG projection on the
	time discretization and the exponential growth in time exhibited by \textit{a posteriori}
	error bounds for both Galerkin and LSPG projection.  We consider two types of
	space--time trial
	subspaces within the proposed approach: one that reduces only the spatial dimension of the full-order model, and one that reduces both the spatial and temporal dimensions of the full-order model. For each type of trial
	subspace, we consider two different solution techniques: direct (i.e.,
	discretize then optimize) and indirect (i.e., optimize then discretize).
	Numerical experiments conducted using trial subspaces characterized by spatial
	dimension reduction demonstrate that the \methodAcronym\
	approach can yield more accurate solutions with lower space--time residuals than
	Galerkin and LSPG projection. 

\end{abstract}
\end{frontmatter}

\section{Introduction}

Simulating parameterized dynamical systems arises in many applications across
science and engineering. In many contexts, executing a dynamical-system
simulation at a single parameter instance---which entails the numerical
integration of a system of ordinary differential equations (ODEs)---incurs an
extremely large computational cost.  This occurs, for example, when the
state-space dimension is large (e.g., due to fine spatial resolution when
discretizing a partial differential equation) and/or when the number of time
instances is large (e.g., due to time-step limitations incurred by stability
or accuracy considerations).  When the application is time critical or many
query in nature, analysts must replace such large-scale parameterized
dynamical-system models (which we refer to as the full-order model) with a low-cost approximation that makes the
application tractable.

Projection-based reduced-order models (ROMs) comprise one such approximation
strategy. First, these techniques execute a computationally expensive
\textit{offline} stage that computes a low-dimensional \textit{trial subspace} on
which the dynamical-system state can be well approximated (e.g., by computing
state ``snapshots'' at different time and parameter instances, by solving
Lyapunov equations). Second, these methods execute an inexpensive
\textit{online} stage during which they compute approximations to the
dynamical-system trajectory that reside on this trial subspace via, e.g., 
projection of the full-order model or residual minimization. 
 
Model reduction for linear-time-invariant systems (and other well structured
dynamical systems) is quite mature
\cite{wilcox_benner_rev,moore,roberts,GugercinIRKA}, as system-theoretic properties (e.g., controllability,
observability, asymptotic stability, $\mathcal H_2$-optimality) can be readily
quantified and accounted for; this often results in
reduced-order models that inherit such important properties.  The primary
challenge in
developing reduced-order models for general nonlinear dynamical systems is
that such properties are difficult to assess quantitatively. As a result, it
is challenging to develop reduced-order models that preserve important
dynamical-system properties, which often results in methods that yield
trajectories that are
inaccurate, unstable, or violate physical properties.  To address this, researchers have pursued several directions that
aim to imbue reduced-order models for nonlinear
dynamical systems with properties that can improve robustness and accuracy.
These
efforts include residual-minimization approaches that equip the ROM solution with a notion of optimality~\cite{carlberg_lspg,carlberg_gnat,legresley_1,legresley_2,legresley_3,bui_resmin_steady,bui_unsteady,rovas_thesis,carlberg_thesis,bui_thesis,l1};
space--time approaches that lead to error bounds that grow slowly in time~\cite{choi_stlspg,constantine_strom,URBAN2012203,Yano2014ASC,benner_st};
``energy-based'' inner
products that ensure non-increasing entropy in the ROM
solution~\cite{rowley_pod_energyproj,Kalashnikova_sand2014,chan2019entropy};
basis-adaptation methods that improve the ROM's accuracy \textit{a
posteriori}~\cite{carlberg_hadaptation,adeim_peherstorfer,etter2019online}, stabilizing subspace rotations that account for truncated modes \textit{a priori}~\cite{basis_rotation}, structure-preserving
methods that enforce conservation \cite{carlberg_conservative_rom} or
(port-)Hamiltonian/Lagrangian structure
\cite{LALL2003304,carlberg2012spd,structurePreserveBeattie,chaturantabut2016structure,farhat2014dimensional} in the ROM; and subgrid-scale
modeling methods that aim to improve accuracy by addressing the closure
problem~\cite{san_iliescu_geostrophic,iliescu_pod_eddyviscosity,iliescu_vms_pod_ns,Bergmann_pod_vms,iliescu_ciazzo_residual_rom,Wang_ROM_thesis,wentland_apg,Wang:269133,San2018}.
We note that these techniques are often not mutually exclusive.
Residual-minimizing and space--time approaches are the most relevant classes of
methods for the current work and comprise the focus of the following review.

Residual-minimization methods in model reduction compute the solution within a
low-dimensional trial subspace that minimizes the full-order-model residual.\footnote{While we focus our review on residual-minimization approaches in the context of model reduction, we note that these approaches are intimately related to least-squares finite element methods~\cite{bochev_leastsquares,bochev_lsfem_book}.}
Researchers have developed such residual-minimizing model-reduction methods
for both static systems (i.e., systems without time-dependence)
\cite{legresley_1,legresley_2,legresley_3,bui_resmin_steady,rovas_thesis,carlberg_thesis,bui_thesis}
and dynamical
systems~\cite{bui_thesis,bui_unsteady,carlberg_thesis,carlberg_gnat,carlberg_lspg,carlberg_lspg_v_galerkin}.
In the latter category,
Refs.~\cite{bui_thesis,bui_unsteady,carlberg_thesis,carlberg_gnat,carlberg_lspg}
formulated the residual minimization problem for dynamical systems by
sequentially minimizing the \textit{time-discrete} full-order-model residual
(i.e., the residual arising after applying time discretization) at each time
instance on the time-discretization grid. This formulation is often referred
to as the \textit{least--squares Petrov--Galerkin} (LSPG) method.
Ref.~\cite{carlberg_lspg_v_galerkin} performed detailed analyses of this
formulation and examined its connections with Galerkin projection. Critically,
this work demonstrated that (1) under certain conditions, LSPG can be cast as a 
Petrov--Galerkin projection applied to the time-continuous
full-order-model residual, and (2) LSPG and Galerkin projection are equivalent
in the limit as the time step goes to zero (i.e., Galerkin projection 
minimizes the time-instantaneous full-order-model residual). Numerous
numerical experiments have demonstrated that LSPG often yields more
accurate and stable solutions than Galerkin~\cite{bui_thesis,carlberg_lspg_v_galerkin,carlberg_gnat,carlberg_thesis,parish_apg}.
The common intuitive explanation for this improved performance is that, by
minimizing the full-order-model residual over a finite time window (rather
than time instantaneously), LSPG computes solutions that are more
accurate over a larger part of the trajectory as compared to Galerkin.

However, LSPG has several notable shortcomings. First, LSPG
exhibits a complex dependence on the time discretization. In
particular, changing
the time step ($\Delta t$) modifies both the
time window over which LSPG minimizes the residual as well as the
time-discretization error of the full-order model on which LSPG is based. 
As LSPG and Galerkin projection are equivalent in the
limit of $\Delta t \rightarrow 0$, the accuracy of LSPG approaches the (sometimes poor)
accuracy of Galerkin as the time step shrinks.
For too-large a time step the accuracy of LSPG also degrades. It is unclear if this is due to the time-discretization error associated with enlarging the time step, or rather if it is due to the size of the window the residual is being minimized over.
As a
consequence, LSPG often yields the smallest error for an intermediate value of the
time step (see, e.g., Ref.~\cite[Figure 9]{carlberg_lspg_v_galerkin}); there
is no known way to compute this optimal time step \textit{a priori}.
Second, as the LSPG approach performs sequential residual minimization in
time, its \textit{a posteriori} error bounds grow exponentially in time
\cite{carlberg_lspg_v_galerkin}, and it is not equipped with any notion of
optimality over the entire time domain of interest. As a result, LSPG is not
equipped with \textit{a priori} guarantees of accuracy or stability, even for
linear time-invariant systems~\cite{bui_thesis}.

Researchers have pursued the development of space--time
residual-minimization approaches~\cite{choi_stlspg,constantine_strom,URBAN2012203,Yano2014ASC} to
address the issues incurred by sequential residual minimization in time.
Existing space--time approaches differ from the classic LSPG and Galerkin approaches in (1) the definition of the space--time trial subspace and (2) the definition of the residual minimization problem. 
First, space--time approaches leverage a space--time trial basis that characterizes both the spatial \textit{and} temporal dependence of the state, classic ``spatial" model reduction approaches such as LSPG and Galerkin leverage only a spatial trial basis that characterizes the spatial dependence of the state. 
Second, space--time residual minimization approaches compute the entire space--time trajectory of the state (within the low-dimensional space--time trial subspace) that minimizes the full-order-model residual over the entire time domain; Galerkin and LSPG sequentially compute instances of the state that either minimize the full-order model instantaneously (Galerkin) or over a time step (LSPG). 
The result of these differences is that space--time approaches yield a system of algebraic equations defined
over all space and time, whose solution comprises a vector of
(space--time) generalized coordinates; on the other hand, spatial-projection-only approaches 
generally associate with systems of ODEs whose solutions
comprise \textit{time-dependent} vectors of (spatial) generalized
coordinates.


Space--time residual minimization approaches minimize the FOM residual over all of space and time and, as a result, yield models that are equipped with a notion of space--time optimality and \textit{a priori} error bounds that grow more slowly
in time. Further, space--time approaches reduce both the spatial and temporal dimensions of the full-order model, and thus promise cost savings over spatial-projection-only approaches.
However, space--time techniques also suffer from several
limitations. First, the computational cost of solving the algebraic system arising from space--time
approaches scales cubically with the number of space--time degrees of freedom;
in contrast, the computational cost incurred by standard
spatial-projection-based ROMs is linear in the number of temporal degrees of
freedom, as the attendant solvers can leverage the lower-block triangular
structure of the system arising from the sequential nature of time
evolution. As a result, solving the algebraic systems arising from space--time
projection is generally intractable without applying hyper-reduction in time
\cite{choi_stlspg,constantine_strom}. Second, space--time residual
minimization precludes future
state prediction, as these methods employ space--time basis vectors defined over the
entire time domain of interest, which must have been included in the training
simulations.

The objectives of this work are to overcome the shortcomings of existing
residual-minimizing methods, and to provide a unifying framework from which
existing methods can be assessed. In essence, the proposed \textit{\methodNameLower\
(\methodAcronym)} \approachKwd\ sequentially minimizes the FOM
residual over a sequence of arbitrarily defined
time windows. The method is characterized by three notable aspects.  First, the method
minimizes the \textit{time-continuous}  residual (i.e., that associated with
the full-order model ODE). By adopting a time-continuous viewpoint, the
formulation decouples the underlying temporal discretization scheme from the
residual-minimization problem, thus addressing a key deficiency of both LSPG
and space--time LSPG. Critically, time-continuous residual minimization also exposes
two different solution methods:  a \textit{discretize-then-optimize}
(i.e., direct) method, and an \textit{optimize-then-discretize} (i.e.,
indirect) method. Second, the method sequentially minimizes
the residual over arbitrarily defined \textit{time windows} rather than
sequentially minimizing the residual over time steps (as in LSPG)
or over the entire time domain (as in space--time residual-minimization
methods). This equips the method with additionally flexibility that enables it
to explore more fine-grained tradeoffs between computational cost and error.
Finally, \methodAcronym\ is formulated for two kinds of space--time trial
subspaces: one that associates with spatial dimension reduction (as employed by traditional spatial-projection-based methods), and one that associates with space--time dimension reduction (as employed by space--time methods).
The above attributes allow the \methodAcronym\ approach to be viewed as a
generalization of existing model-reduction methods, as 
Galerkin, LSPG, and space--time LSPG projection correspond to specific
instances of the formulation.
Figure~\ref{fig:flowchart} depicts how the proposed \methodAcronym\ method
provides a unifying framework from which these existing approaches can be derived.

The \methodAcronym\ approach can be viewed as a hybrid space--time method and displays 
commonalities with several related efforts. 
First, Ref.~\cite{bui_thesis} briefly formulated a space--time least-squares ROM
and connected this formulation with optimal control; specifically it mentioned the
optimize-then-discretize vs.\ discretize-then-optimize approaches. This work
did not fully develop this approach, eschewing it 
for 
sequential residual minimization in time (i.e., LSPG). The present work thus
formally develops and extends several of the concepts put forth
in Ref.~\cite{bui_thesis}. Next, Ref.~\cite{constantine_strom} developed a
space--time residual minimization formulation for model interpolation.  The
present work distinguishes itself from Ref.~\cite{constantine_strom} in that
(1) this work also considers trial subspaces characterized by spatial dimension reduction only, and (2) we
minimize the time-continuous FOM residual over arbitrary time
windows. We note that,
similar to the current work, Ref.~\cite{constantine_strom} employs
minimization of the time-continuous FOM residual as a starting point; as such,
this work shares some thematic similarities with
Ref.~\cite{constantine_strom}.  Lastly, Ref.~\cite{choi_stlspg} develops a
space--time extension of LSPG projection that minimizes the time-discrete
FOM residual over the entire time domain.  The present work distinguishes
itself from Ref.~\cite{choi_stlspg} in that (1) this work minimizes the
time-continuous FOM residual, (2) this work minimizes this residual over
arbitrary time windows, and (3) this work also considers trial subspaces
associated with spatial dimension reduction only.

\begin{figure} 
\begin{centering} 
\includegraphics[trim={0.1cm 0cm 1cm 0cm},clip,width=0.99\textwidth]{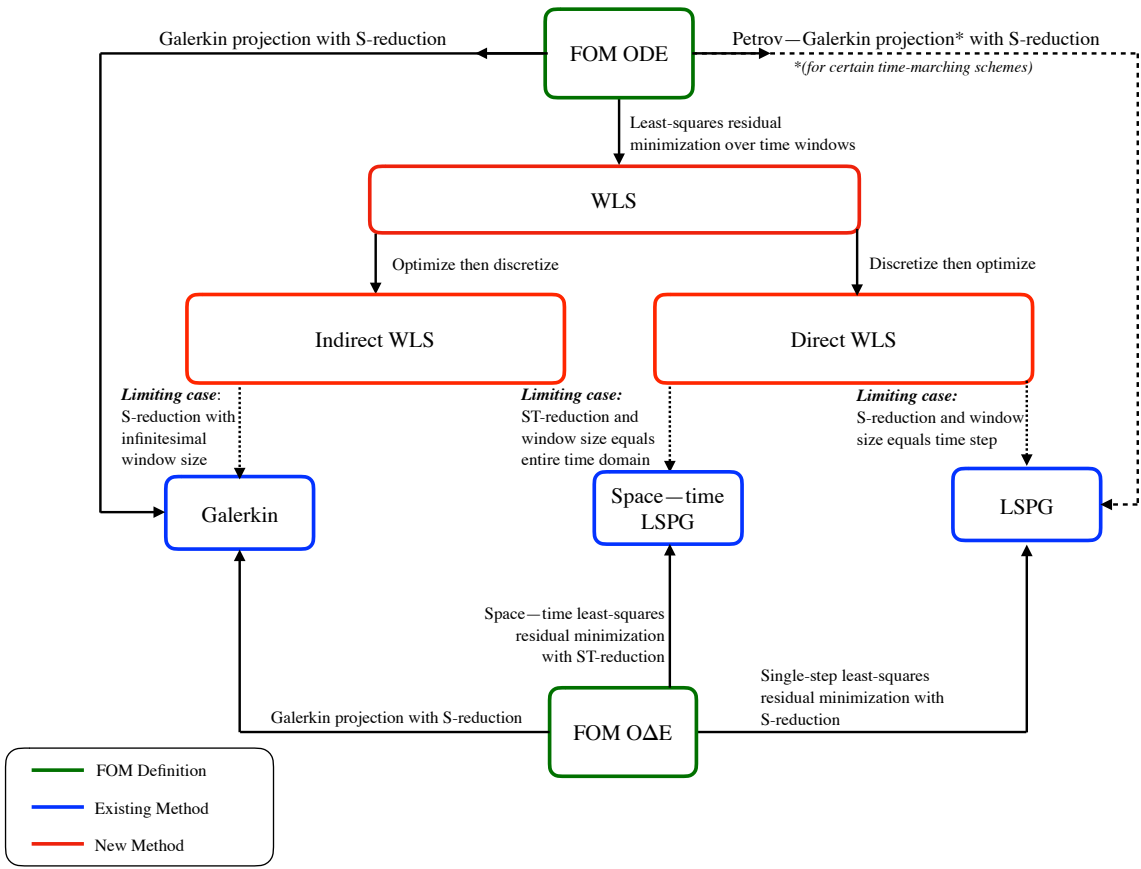} 
\caption{Relationship diagram for the \methodAcronym\ approach for model
	reduction.} 
\label{fig:flowchart} 
\end{centering} 
\end{figure}

 In summary, specific contributions of this work include:
\begin{enumerate}
\item The \methodNameLower\ (\methodAcronym) approach for dynamical-system
	model reduction. The approach sequentially minimizes the time-continuous
		full-order-model residual over arbitrary time windows.
\item Support of two space--time trial subspaces: one that associates with spatial dimension reduction and 
    one that associates with spatial and temporal dimension reduction.
The former case is of particular
		interest in the \methodAcronym\  context, as the stationary conditions are
		derived via the Euler--Lagrange equations and comprise a coupled two-point
		Hamiltonian boundary value problem containing a forward and backward
		system. The forward system, which is forced by an auxiliary costate,
		evolves the (spatial) generalized coordinates of the ROM in time. The
		backward system, which is forced by the time-continuous FOM residual
		evaluated about the ROM state, governs the dynamics of the costate. 
	\item Derivation of two solution techniques:
		\textit{discretize then optimize} (i.e., direct) and
		\textit{optimize then discretize} (i.e., indirect). 
	\item Remarks and derivations of conditions under which the
		\methodAcronym\ approach recovers Galerkin, LSPG, and space--time LSPG.
\item Error analysis of the \methodAcronym\ approach using
	trial subspaces associated with spatial dimension reduction.
	This analysis demonstrates that, over a given window, the \methodAcronymROMs\ error 
		is bounded \textit{a priori} by a combination of the
		error at the start of the window and the integrated FOM ODE residual
		evaluated at the FOM state projected onto the trial subspace. 
\item Numerical experiments for 
	trial subspaces associated with spatial dimension reduction, which demonstrate two key findings:
\begin{itemize}
\item Minimizing the residual over a larger time window leads to more stable solutions 
with lower space--time residuals norms.
\item Minimizing the residual over a larger time window does not necessarily
	lead to a more accurate trajectory (as measured in the space--time
		$\elltwo$-norm of the solution). Instead, minimizing the residual over an
		intermediate-sized time window leads to the smallest trajectory error.
\end{itemize}
\end{enumerate}

The paper proceeds as follows: Section~\ref{sec:math}
outlines the mathematical setting for the full-order model, along with Galerkin, LSPG, and
space--time LSPG projection. Section~\ref{sec:tclspg} outlines the proposed \methodAcronym\
\approachKwd. Section~\ref{sec:numerical_techniques}
outlines numerical techniques for solving \methodAcronymROMs, including both direct and
indirect methods. Section~\ref{sec:analysis} provides
equivalence conditions and error analysis for \methodAcronymROMs.
Section~\ref{sec:numerical_experiments} presents numerical experiments.
Section~\ref{sec:conclude} provides conclusions and perspectives.
We denote vector-valued functions with italicized bold symbols (e.g., $\boldsymbol
x$), vectors with standard bold symbols (e.g., $\mathbf{x}$), 
matrices with capital bold symbols (e.g., $\mathbf{X} \equiv \begin{bmatrix}
\mathbf{x}_1 & \cdots & \mathbf{x}_r\end{bmatrix}$), and spaces with
calligraphic symbols (e.g., $\mathcal{X}$). We additionally denote differentiation of a time-dependent 
function with respect to time with the $\dot\ $ operator.
\section{Mathematical formulation}\label{sec:math}
	We begin by providing the formulation for the full-order model,
	followed by a description of standard model-reduction methods
	classified according to the type of trial subspace they employ.
	\subsection{Full-order model}\label{sec;FOM}
We consider the full-order model to be a dynamical system expressed as a
	system of ordinary differential equations (ODEs)
\begin{equation}\label{eq:FOM}
 \stateFOMDotArg{t}  = \velocity(\stateFOMArg{}{t},t), \qquad \stateFOM(0) =
	\stateFOMIC,\qquad t \in [0,T],
\end{equation}
where  $\stateFOM: [0,T]
	\rightarrow  \RR{\fomdim}$ with $\stateFOM: \timeDummy \mapsto
	\stateFOM(\timeDummy) $ and $\stateFOMDot \equiv {d \stateFOM}/{d\tau}$
	denotes the state implicitly defined as the solution to initial value
	problem~\eqref{eq:FOM}, 
$T \in \RRplus$ denotes the final time, 
 $\stateFOMIC \in \mathbb{R}^{\fomdim}$ denotes the initial condition,
	and $\velocity: \mathbb{R}^{\fomdim} \times [0,T] \rightarrow
	\mathbb{R}^{\fomdim}$ with $(\stateyDiscrete,\timeDummy) \mapsto
	\velocity(\stateyDiscrete,\timeDummy)$ denotes the velocity, which is possibly
	nonlinear in its first argument. For subsequent exposition, we introduce
	$\timeSpace$ to denote the set of (sufficiently smooth) real-valued functions acting on the time
	domain (i.e., $\timeSpace = \{f\,|\,f:[0,T]\rightarrow\RR{}\}$); the
	state can be expressed equivalently as $\stateFOM \in \RR{\fomdim} \otimes
	\timeSpace $.  We refer to the initial value problem defined in
	Eq.~\eqref{eq:FOM} as the ``full-order model" (FOM) ODE. We note that
	although the problem of interest described in the introduction corresponds
	to a parameterized dynamical system, we suppress dependence of the FOM ODE
	\eqref{eq:FOM} on such parameters for notational convenience, as this work
	focuses on devising a model-reduction approach applicable to a specific parameter instance.

Directly solving the FOM ODE~\eqref{eq:FOM} is computationally expensive if
	either the state-space dimension $\fomdim$ is large, or if the time-interval
	length $T$ is large relative to the time step required to numerically
	integrate Eq.~\eqref{eq:FOM}. For time-critical or many-query applications,
	it is essential to replace the FOM ODE~\eqref{eq:FOM} with a strategy that
	enables an approximate trajectory to be computed at lower computational
	cost. Projection-based ROMs constitute one such promising approach. 

\subsection{Reduced-order models}
Projection-based ROMs generate approximate solutions to the FOM
	ODE~\eqref{eq:FOM} by approximating the state in a low-dimensional trial
	subspace. Two types of space--time trial subspaces are commonly used for
	this purpose:\footnote{For both spatial and space--time ROMs of dynamical systems, all trial subspaces are, strictly speaking, space--time subspaces.} 
\begin{enumerate} 
	\item \textit{Subspaces that reduce only the spatial dimension of the full-order
		model (\spatialAcronym)}. These trial subspaces are characterized by a spatial projection operator, associate with a basis that represents the spatial dependence of the state, and are employed in classic model reduction approaches, e.g., Galerkin and LSPG. 
	\item \textit{Subspaces that reduce both the spatial and temporal dimensions of the full-order
		model (\spaceTimeAcronym)}.
These trial subspaces are characterized by a space--time projection operator, associate with a basis that represents the spatial and temporal dependence of the state, and are employed in space--time 
model reduction approaches (e.g., space--time Galerkin~\cite{benner_st}, space--time LSPG~\cite{choi_stlspg}). 
\end{enumerate}
 We now describe these two types of space--time trial subspaces and their
	application to the Galerkin, LSPG, and space--time LSPG approaches. 

\subsection{\spatialAcronym\ trial subspaces}
At a given time instance 
$t\in[0,T]$,
\spatialAcronym\ trial subspaces approximate the FOM ODE solution
	as $\approxstate(t)\approx\stateFOM(t)$, which is enforced to reside in an
	affine spatial trial subspace of dimension $K\ll\fomdim$ such that
	$\approxstate(t)\in
	\stateIntercept+\trialspace
\subseteq\RR{\fomdim}$, where $\dim(\trialspace) = K$
and $\stateIntercept \in \mathbb{R}^{\fomdim}$ denotes the reference state, which
	is often taken to be the initial condition (i.e., $\stateIntercept = \stateFOMIC$).
Here, the trial subspace
$\trialspace$ 
is spanned by an orthogonal basis such that
$ \trialspace= \Range{\basismat}$
with 
$ \basismat \equiv \begin{bmatrix}  \basisvec_1  & \cdots &  \basisvec_K \end{bmatrix}
	\in \RRStar{\romdim}{\fomdim}$, where $\RRStar{\romdim}{\fomdim}$ denotes the compact Stiefel manifold (i.e.,  $
	\RRStar{\romdim}{\fomdim}\defeq
	\{ \mathbf{X} \in \RR{\fomdim
	\times \romdim}\, \big|\, \mathbf{X}^T \mathbf{X} = \mathbf{I} \}$).
The basis vectors $\basisvec_i$, $i=1,\ldots,K$ are typically constructed
using state snapshots, e.g., via
proper orthogonal decomposition (POD)~\cite{berkooz_turbulence_pod}, the reduced-basis method~\cite{rb_1,rb_2,rb_3,NgocCuong2005,Rozza2008}. 
Thus, at any time instance $t\in[0,T]$, ROMs that employ the  \spatialAcronym\
trial subspace approximate the FOM ODE solution as
\begin{equation}\label{eq:affine_trialspace}
\stateFOMArg{}{t} \approx \approxstate(t) = \basisspace \genstateArg{}{t} + \stateIntercept,
\end{equation}
where $\genstate \in \RR{\romdim} \otimes \timeSpace$ with
$\genstate:\timeDummy\mapsto\genstate(\timeDummy)$
denotes the generalized
coordinates. From the space--time perspective, this is equivalent to approximating the
	FOM ODE solution trajectory $\stateFOM\in\RR{N}\otimes\timeSpace$ with 
	$\approxstate\in \stspaceS$, where
\begin{equation}\label{eq:spatial_subspace}
\begin{split}
& \stspaceS \defeq \trialspace \otimes \timeSpace +
	\stateIntercept\otimes\onesFunction\subseteq\RR{N}\otimes\timeSpace,
\end{split}
\end{equation}
with $\onesFunction\in\timeSpace$ defined as
$\onesFunction:\timeDummy\mapsto 1$.
	 
Substituting the approximation~\eqref{eq:affine_trialspace} into the FOM ODE~\eqref{eq:FOM} and performing orthogonal
$\elltwo$-projection of the initial condition onto the trial subspace yields
the overdetermined system of ODEs
\begin{equation}\label{eq:g_truncation}
\basisspace \genstateDotArg{}{t} = \velocity(\basisspace
\genstateArg{}{t} + \stateIntercept,t ), \qquad \genstate(0) = \genstateIC,
	\qquad t \in [0,T],
\end{equation}
where $\genstateDot\equiv {d \genstate}/{d\tau}$.
Because Eq.~\eqref{eq:g_truncation} is overdetermined, a solution may not
exist. Typically, either \textit{Galerkin} or \textit{least-squares
Petrov--Galerkin} projection is employed to reduce the number of equations
such that a unique solution exists. We now describe these two methods.

\subsubsection{Galerkin projection}
The Galerkin approach reduces the number of equations in
Eq.~\eqref{eq:g_truncation} by enforcing orthogonality of the 
residual to the spatial trial subspace in the (semi-)inner product induced by the 
positive (semi-)definite $\fomdim\times\fomdim$ matrix $\stweightingMatArg{}\equiv \stweightingMatOneTArg{ }
\stweightingMatOneArg{ }$ (commonly set to
$\stweightingMatArg{}=\mathbf{I}$), i.e.,
\begin{equation}\label{eq:g_truncation2}
\genstateGalerkinDotArg{}{t} = \massArgnt{n}^{-1} \basisspace^T \stweightingMatArg{} \velocity(\basisspace
\genstateGalerkinArg{}{t} + \stateIntercept,t), \qquad \genstateGalerkin(0) = \genstateIC, \qquad t \in [0,T],
\end{equation}
where $\massArgnt{ } \equiv \basisspace^T \stweightingMatArg{ } \basisspace$
denotes the $K\times K$ positive definite mass matrix.
As demonstrated in Ref.~\cite{carlberg_lspg_v_galerkin}, the Galerkin approach
can be viewed alternatively as a residual-minimization
method, as the Galerkin ODE~\eqref{eq:g_truncation2} is equivalent
to
\begin{equation}\label{eq:GalOptimal}
\genstateGalerkinDotArg{}{t} = \underset{ \genstateyDiscrete \in \RR{\romdim}
	}{\text{arg\,min}} \| \basisspace \hat{\stateyDiscrete} -
	\velocity(\basisspace \genstateGalerkinArg{}{t} + \stateIntercept,t)
	\|_{\stweightingMatArg{}}^2, \qquad \genstate(0) = \genstateIC,
	\qquad t \in [0,T],
\end{equation}
where $\|\mathbf{x}\|_{\stweightingMatArg{}}\equiv
\sqrt{\mathbf{x}^T\stweightingMatArg{}\mathbf{x}}$.
Thus, the computed velocity $\genstateGalerkinDotArg{}{t}$ minimizes the
FOM ODE residual evaluated at the state $\basisspace
\genstateGalerkin(t) + \stateIntercept$ and time instance $t$ over the spatial trial
subspace $\trialspace$.
\subsubsection{Least-squares Petrov--Galerkin projection}
Despite its time-instantaneous residual-minimization optimality property 
\eqref{eq:GalOptimal}, the Galerkin approach can yield inaccurate solutions,
particularly if the velocity is not self-adjoint or
is nonlinear. 
Least-squares Petrov--Galerkin (LSPG)
~\cite{carlberg_lspg_v_galerkin,carlberg_thesis,carlberg_gnat,bui_unsteady,bui_thesis}
was developed
as an alternative method that exhibits several advantages over the Galerkin approach.
Rather than minimize the (time-continuous) FOM ODE residual at a time instance (as in Galerkin), LSPG minimizes
the (time-discrete) FOM O$\Delta$E residual 
(i.e., the residual arising after applying time discretization to the FOM ODE)
over a time step.
We now describe the LSPG approach in the case of linear multistep methods;
Ref.~\cite{carlberg_lspg_v_galerkin} also presents LSPG for Runge--Kutta
schemes. 

Without loss of generality, we introduce a uniform time
grid characterized by time step $\Delta t$ and time instances
$t^n = n\Delta
t$, $n=0,\ldots,N_t$.
Applying a linear multistep method to discretize the FOM ODE \eqref{eq:FOM}
with this time grid
yields the FOM O$\Delta$E, which computes the sequence of discrete
solutions
$\stateFOMDiscreteArg{n}( \approx \stateFOM(t^n))$, $n=1,\ldots,N_t$
as the implicit solution to the system of algebraic equations
\begin{equation}\label{eq:lms}
\residLMSArg{n}
	(\stateFOMDiscreteArg{n};\stateFOMDiscreteArg{n-1},\ldots,\stateFOMDiscreteArg{n-k^n})
	= \bz,\qquad n=1,\ldots,N_t,
\end{equation}
with the initial condition $\stateFOMDiscreteArg{0} = \stateFOMIC$. In the above, $k^n$ denotes the number of steps employed by the scheme at the $n$th
time instance and 
$\residLMSArg{n}$ denotes the FOM O$\Delta$E residual defined as
\begin{align*}
\residLMSArg{n} &: (\stateyDiscreteArgnt{n};\stateyDiscreteArgnt{n-1},\ldots,\stateyDiscreteArgnt{n-k^n}) \mapsto  \frac{1}{\Delta t} \sum_{j=0}^{k^n} \alpha^n_j \stateyDiscreteArgnt{n-j} -  \sum_{j=0}^{k^n} \beta^n_j \velocity(\stateyDiscreteArgnt{n-j},t^{n-j}),
\\
&: \RR{\fomdim} \otimes \RR{k^n + 1} \rightarrow \RR{\fomdim}.
\end{align*} 
Here, $\alpha^n_j,\beta^n_j\in\RR{}$, $j=0,\ldots,k^n$ are coefficients
that define the linear multistep method at the $n$th time instance.

At each time instance on the time grid, LSPG substitutes the \spatialAcronym\ trial subspace approximation
\eqref{eq:affine_trialspace} into the FOM O$\Delta$E~\eqref{eq:lms} and
minimizes the residual, i.e., LSPG sequentially computes the solutions
$\approxstateLSPG^n \approx \stateFOMDiscreteArg{n}$, $n=1,\ldots,
\ntimeSteps$ that satisfy
\begin{equation*}
\approxstateLSPG^n = \underset{\stateyDiscrete \in \trialspace + \stateIntercept }{\text{arg\,min}}|| \lspgWeightingArg{\stateyDiscrete}\residLMSArg{n} (\stateyDiscrete; \approxstateLSPG^{n-1},\ldots,\approxstateLSPG^{n-k^n}) ||_2^2, \qquad n = 1,\ldots,\ntimeSteps,
\end{equation*}
where 
$\lspgWeightingArg{\cdot} \in \RR{\nsamples \times \fomdim}$, with $\romdim
\le \nsamples \le \fomdim$, is a weighting matrix that can be used, e.g., to
enable hyper-reduction by requiring it to have a small number of nonzero
columns. 

As described in the introduction, although numerical experiments have
demonstrated that LSPG often yields more accurate and stable
solutions than Galerkin~\cite{bui_thesis,carlberg_lspg_v_galerkin,carlberg_gnat,carlberg_thesis,parish_apg},
LSPG still suffers from several shortcomings. In particular, LSPG suffers from its complex
dependence on the time discretization, exponentially growing error bounds, and
lack of optimality for the trajectory defined over the entire time domain.

\subsection{\spaceTimeAcronym\ trial spaces and space--time ROMs}

Space--time projection methods that employ \spaceTimeAcronym\ trial
spaces~\cite{choi_stlspg,constantine_strom,URBAN2012203,Yano2014ASC,benner_st,bui_thesis}
aim to overcome the latter two shortcomings of LSPG. Because these methods employ \spaceTimeAcronym\ trial
spaces, they reduce both the spatial and temporal dimensions of the full-order
model; further, they yield error bounds that grow more slowly in time and
their trajectories exhibit an optimality property over the entire time domain. 

\spaceTimeAcronym\ trial subspaces approximate the FOM ODE solution
trajectory
	$\stateFOM\in\RR{N}\otimes \timeSpace$ with an approximation that resides in an
	affine space--time trial subspace of dimension $\stdim\ll\fomdim$, i.e., 
	$\approxstate\in \stspaceST$ with $\dim(\stspaceST) =\stdim $, where
\begin{equation}\label{eq:sttrialspace_def}
 \stspaceST \defeq 
	\Range{\stbasis} + 
	\stateFOMIC\otimes\onesFunction
	\subseteq \RR{N} \otimes \timeSpace.
\end{equation}
Here $\stbasis \in \RR{\fomdim \times \stdim} \otimes \timeSpace$, with $\stbasis:\timeDummy\mapsto\stbasis(\timeDummy)$ and $\stbasis(0) =
\boldsymbol 0$
denotes the space--time trial basis. 
Thus, at any time instance $t\in[0,T]$, ROMs that employ the
\spaceTimeAcronym\ trial subspace approximate the FOM ODE solution as
\begin{equation}\label{eq:stapprox1}
 \stateFOMArg{}{t} \approx \approxstateArg{}{t}  = \stbasisArg{}{t} \stgenstate + \stateFOMIC,
\end{equation}
where $\stgenstateArg{} \in \RR{\stdim}$ denotes the space--time generalized coordinates. 
Critically, comparing the approximations arising from \spatialAcronym\ and
\spaceTimeAcronym\ trial subspaces in Eqs.~\eqref{eq:affine_trialspace} and \eqref{eq:stapprox1}, respectively,
highlights that the former approximation associates with time-dependent
generalized coordinates, while the latter approximation associates with a
time-dependent basis matrix.

Substituting the approximation~\eqref{eq:stapprox1} into the FOM
ODE~\eqref{eq:FOM} yields
\begin{equation}\label{eq:st_fom}
\stbasisDotArg{}{t} \stgenstate =  \velocity(\stbasisArg{}{t} \stgenstate +
	\stateFOMIC,t) , \qquad t \in [0,T],
\end{equation}
where 
$\stbasisDot\equiv d\stbasis/d\timeDummy$. We note that the initial conditions
are automatically satisfied from the definition of the \spaceTimeAcronym\
trial subspace.

Space--time methods reduce the number of equations in
\eqref{eq:st_fom} to ensure a unique solution exists.
We now outline one such method: space--time least-squares Petrov--Galerkin (ST-LSPG)
~\cite{choi_stlspg}. 
While the space--time Galerkin method~\cite{benner_st} is another alternative, it does not
associate with any residual-minimization principle, and thus we do not discuss
it further.
\subsubsection{Space--time LSPG projection} 
Analogously to LSPG, space--time LSPG~\cite{choi_stlspg}
minimizes the (time-discrete) FOM O$\Delta$E residual, but does so
using the \spaceTimeAcronym\ subspace and simultaneously minimizes this
residual over all $N_t$ time instances.
We first introduce the full space--time FOM O$\Delta$E residual for linear
multistep methods as
\begin{align*}
\residST_{} & \vcentcolon (\stateyDiscreteArg{1}, \ldots
	,\stateyDiscreteArg{N_t};\stateFOMIC) \mapsto \begin{bmatrix}
\residLMSArg{1}(\stateyDiscreteArg{1}; \stateFOMIC) 
\\ 
\vdots \\
\residLMSArg{N_t}(\stateyDiscreteArg{N_t};\stateyDiscreteArg{N_t-1},\ldots,\stateyDiscreteArg{N_t - k^{N_t}}) \end{bmatrix} , \\
& \vcentcolon \RR{\fomdim} \otimes \RR{N_t+1} \rightarrow \RR{\fomdim N_t},
\end{align*}
and define the counterpart function acting on space--time generalized
coordinates:
\begin{align*}
\genresidST_{\text{}} & \vcentcolon (\stgenstatey ;\stateFOMIC) \mapsto \begin{bmatrix}
\residLMSArg{1}\bigg( \stbasisArg{}{t^1}\stgenstatey + \stateIntercept; \stateFOMIC \bigg)
\\ 
\vdots \\
\residLMSArg{N_t}\bigg( \stbasisArg{}{t^{N_t}}\stgenstatey + \stateIntercept; \stbasisArg{}{t^{N_t-1}}\stgenstatey + \stateIntercept ,\ldots,  \stbasisArg{}{t^{N_t-k^{N_t}}}\stgenstatey + \stateIntercept \bigg)  \\ 
\end{bmatrix} , \\
& \vcentcolon \RR{\stdim} \times \RR{N} \rightarrow \RR{\fomdim N_t}. 
\end{align*}
ST-LSPG computes the space--time generalized coordinates that
minimize the space--time FOM O$\Delta$E residual:
\begin{equation}\label{eq:stlspg}
	\stgenstate_\text{ST-LSPG} = \underset{ \stgenstatey\in\RR{\stdim} }{\text{arg\,min}}|| \lspgWeightingSTArg{\stgenstatey}  \genresidST_{\text{}}(\stgenstatey;\stateFOMIC) ||_2^2, 
\end{equation}
where $\lspgWeightingSTArg{\cdot} \in \RR{n_{st} \times N N_t}$, with $\stdim
\le n_{st} \le N N_t$, is a space--time weighting matrix that can be chosen,
e.g., to enable hyper-reduction.

ST-LSPG overcomes two of the primary shortcomings of LSPG. In particular, it
leads to error bounds that grow sub-quadratically in time rather than
exponentially in time, and it generates entire trajectories that associate
with an optimality property over the entire time domain \cite{choi_stlspg}.
However, it is subject to several challenges. First, the computational cost of
solving Eq.~\eqref{eq:stlspg} scales cubically with the number of space--time
degrees of freedom $\stdim$. This cost is due to the fact that ST-LSPG yields 
dense systems that do not expose any natural mechanism for exploiting
the sequential nature of time evolution. Second, it is unclear how these
methods can be employed for future state prediction, as the space--time trial
basis $\stbasis$ must be defined over the entire time interval of interest
$[0,T]$. Third, ST-LSPG is still strongly tied to the time discretization
employed for the full-order model, as it minimizes the (time-discrete) FOM
O$\Delta$E residual over all time instances.
  
\subsection{Outstanding challenges}
This work seeks to overcome the limitations of existing residual-minimizing
model-reduction methods, and to provide a unifying framework from which
existing methods can be assessed. Specifically, we look to overcome the
complex dependence of LSPG on the time discretization,
exponential time growth of the error bounds for Galerkin and LSPG,
the cubic dependence of the computational cost of ST-LSPG on the number of
degrees of freedom, and the lack of ability for ST-LSPG to perform prediction
in time. We now describe the proposed \methodNameLower\ \approachKwd\ for this
purpose. 

\input{section3/tclsrm_formulation}
\input{section4/numerical_methods}

\input{analysis}
\input{numerical_experiments}

\section{Conclusions}\label{sec:conclude}
This paper proposed the windowed least-squares (\methodAcronym) \approachKwd\ for model reduction  of dynamical systems. The approach sequentially minimizes the
	time-continuous full-order-model residual within a low-dimensional space--time trial
	subspace over time windows. The approach was formulated for two types of trial subspaces: one that reduces only the spatial dimension of the full-order model, and one that reduces both the spatial and temporal dimensions of the full-order model. For each type of trial
	subspace, we outlined two different solution techniques: direct (i.e.,
	discretize then optimize) and indirect (i.e., optimize then discretize). We showed that particular instances of the approach recover Galerkin,
	least-squares Petrov--Galerkin (LSPG), and space–time LSPG projection.  
The case of \spatialAcronym\ trial subspaces is of particular interest in the \methodAcronym\ context: it was shown that indirect methods comprise solving a coupled two-point 
Hamiltonian boundary value problem.
The forward system, which is forced by an auxiliary costate,
		evolves the (spatial) generalized coordinates of the ROM in time. The
		backward system, which is forced by the time-continuous FOM residual
		evaluated about the ROM state, governs the dynamics of the costate. 

Numerical experiments of the compressible Euler and Navier--Stokes equations demonstrated the utility in the proposed approach. The first numerical experiment, in where the Sod shock tube was examined, demonstrated that \methodAcronymROMs\ minimizing the residual over larger time windows yielded solutions with lower space--time residuals. Increasing the window size over which the residual was minimized, however, did not necessarily decrease the solution error in the $\elltwo$-norm; we observed this to occur over an intermediary window size. We additionally observed that the  
\methodAcronym\ \approachKwd\ overcomes the time-discretization sensitivity that LSPG is subject to. The second numerical experiment, which examined collocated ROMs of a compressible cavity flow, demonstrated the utility of the \methodAcronym\ formulation on a more complex flow. In this experiment, \methodAcronymROMs\ yielded predictions with relative errors less than $5-10\%$, while the Galerkin and LSPG ROMs failed to converge/blew up. Increasing the window size over which the residual was minimized again led to a lower space--time residual, but not necessarily a lower error in the $\elltwo$-norm.

 
The principal challenge encountered in the \methodAcronym\ formulation is the computational cost: increasing the window size over which the residual is minimized leads to a higher computational cost. In the context of the direct solution approach, this increased cost is due to the increased expense of forming and solving the least-squares problem associated with larger window sizes. In the context of indirect methods, this increased cost is a result of the increased number of iterations of the forward backward sweep method. While numerical experiments demonstrated that the increase in computational cost is mild, future work will target the development of new solution techniques tailored for the \methodAcronym\ approach.



\section{Acknowledgments}
The authors thank Yukiko Shimizu and Patrick Blonigan for numerous conversations from which this work benefited. 
E.\ Parish acknowledges an appointment to the Sandia National Laboratories'
John von Neumann Postdoctoral Research Fellowship in Computational Science. 
This work was partially sponsored by Sandia's Advanced Simulation and
Computing (ASC) Verification and Validation (V\&V) Project/Task
\#103723/05.30.02.  This paper describes objective technical results and
analysis. Any subjective views or opinions that might be expressed in the
paper do not necessarily represent the views of the U.S.\ Department of Energy
or the United States Government.  Sandia National Laboratories is a
multimission laboratory managed and operated by National Technology and
Engineering Solutions of Sandia, LLC., a wholly owned subsidiary of Honeywell
International, Inc., for the U.S.\  Department of Energy's National Nuclear
Security Administration under contract DE-NA-0003525.

\begin{appendices}
\section{Proper orthogonal decomposition}
Algorithm~\ref{alg:pod} presents the algorithm for computing the trial basis via proper orthogonal decomposition.
\begin{algorithm}
\caption{Algorithm for generating POD Basis.}
\label{alg:pod}
\textbf{Input:} Number of time-steps between snapshots $N_{\text{skip}}$; intercept $\stateInterceptArg{}$; basis dimension $K$ \; 
\textbf{Output:} POD Basis $\basisspace \in \mathbb{V}_{K}(\RR{N})$ \;
\textbf{Steps:}
\begin{enumerate}
    \item Solve FOM O$\Delta$E and collect solutions into snapshot matrix
$$\mathbf{S}(N_{\text{skip}}) \defeq \begin{bmatrix} \stateFOMDiscreteArg{0} - \stateInterceptArg{n} & \stateFOMDiscreteArg{N_{\text{skip}}} - \stateInterceptArg{n} & \cdots & \stateFOMDiscreteArg{\text{floor}(N_t/N_{\text{skip}})N_{\text{skip}}} - \stateInterceptArg{n} \end{bmatrix}$$
    \item Compute the thin singular value decomposition, $$\mathbf{S} (N_{\text{skip}}) = \mathbf{U \Sigma Z}^T,$$
    where $\mathbf{U} \equiv \begin{bmatrix} \mathbf{u}_1 & \cdots & \mathbf{u}_{\text{floor}(N_t/N_{\text{skip}})}\end{bmatrix}$
    \item Truncate left singular vectors and to form the basis, $\basisspace \equiv \begin{bmatrix} \mathbf{u}_1 & \cdots & \mathbf{u}_K \end{bmatrix}$
\end{enumerate}

\end{algorithm}

\section{Selection of sampling points}
To construct the sampling point matrix used for hyper-reduction in the second numerical experiment, we employ q-sampling~\cite{qdeim_drmac} and the sample 
mesh concept~\cite{carlberg_gnat}. Algorithm~\ref{alg:qdeim} outlines the steps used in the second numerical experiment to compute the sampling points.

\begin{algorithm}
\caption{Algorithm for generating the sampling matrix through q-sampling.}
\label{alg:qdeim}
\textbf{Input:} Number of time-steps between snapshots $N_{\text{skip}}$, number of primal sampling points, $n_s$ \; 
\textbf{Output:} Weighting matrix $\stweightingMatArg{} \equiv \stweightingMatOneTArg{} \stweightingMatOneArg{} \in \{0,1\}^{N \times N}$ \;
\textbf{Steps:}
\begin{enumerate}
    \item Solve FOM O$\Delta$E and collect velocity snapshots 
$$\mathbf{F}(N_{\text{skip}}) \defeq \begin{bmatrix} \velocity (\stateFOMDiscreteArg{0})  & \velocity( \stateFOMDiscreteArg{N_{\text{skip}}}) & \cdots & \velocity( \stateFOMDiscreteArg{\text{floor}(N_t/N_{\text{skip}})N_{\text{skip}}} ) \end{bmatrix}$$
    \item Compute the thin singular value decomposition, $$\mathbf{F} (N_{\text{skip}}) = \mathbf{U \Sigma Z}^T,$$
    where $\mathbf{U} \equiv \begin{bmatrix} \mathbf{u}_1 & \cdots & \mathbf{u}_{\text{floor}(N_t/N_{\text{skip}})}\end{bmatrix}$
    \item Compute the QR factorization of  $\mathbf{U}^T$ with column pivoting,
    \begin{equation*}
        \mathbf{U}^T \mathbf{P}^* = \mathbf{QR}
    \end{equation*}
    with $\mathbf{P}^* \equiv \begin{bmatrix} \mathbf{p}_1 & \cdots & \mathbf{p}_{\text{floor}(N_t/N_{\text{skip}})} \end{bmatrix}$, $\mathbf{p}_i \in \{0,1\}^N$. 
    \item Select the first $n_s$ columns of $\mathbf{P}^*$ to form the sampling point matrix $\stweightingMatOneTArg{} \in \{0,1\}^{N \times n_s}$. 
    \item Augment the sampling point matrix, $\stweightingMatOneArg{}$, with additional columns such that all unknowns are computed at the mesh cells selected by Step 3. For the second numerical experiment, these additional unknowns correspond to each conserved variable and quadrature point in the selected cells. 
\end{enumerate}

\end{algorithm}

\section{Derivation of the Euler--Lagrange equations}\label{appendix:eulerlagrange}
This section details the derivation of the Euler--Lagrange equations. To this end, we consider the generic functional of the form,
\begin{equation}\label{eq:el_gen}
\mathcal{J}: (\statey,\stateyDot) \mapsto  \int_a^b \integrand(\statey(t),\stateyDot(t),t)  dt,
\end{equation}
where $\integrand: \RR{M}  \times \RR{M}  \times[a,b] \rightarrow  \RRplus$
	(for arbitrary $M$) with $\integrand:
	(\stateyDiscrete,\stateyDotDiscrete,\timeDummy) \mapsto
	\integrand(\stateyDiscrete,\stateyDotDiscrete,\timeDummy)$. We now introduce
	the function $\statez: [a,b] \rightarrow \RR{M}$ with $\statez: \timeDummy \mapsto \statez(\timeDummy)$ along with $\statezDot \equiv d \statez /d\timeDummy$. We define this function to be a stationary point of~\eqref{eq:el_gen} (with $\statez$ being the first argument and $\statezDot$ being the second argument) subject to the boundary condition
	$\statez(a) = \statez_a$. 
We additionally introduce an arbitrary function $\variation : [a,b] \rightarrow \RR{M}$
with the boundary condition $\variationArg{a} = \bz$ and define a variation
from the stationary point by
\begin{align*}
\statezBar  &: (\tau,\delta) \mapsto \statez(\tau) + \delta \variation(\tau),\\
&: [a,b] \times \RR{} \rightarrow \RR{M}.
\end{align*}
We note that $\statezBar$ satisfies the same boundary condition $\statez$ since $\variationArg{a} = \bz$.
We define a new function that is equivalent to the function~\eqref{eq:el_gen} evaluated at $\statezBar$ in the first argument and 
$\statezDotBar \equiv d\statezBar / d\timeDummy$ in the second argument,
$$
\mathcal{J}_{\delta}: \delta \mapsto \int_a^b \integrand(\statezBar(t,\delta),\statezDotBar(t,\delta) ,t)  dt.
$$
The objective is to now find $\overline{\statez}$ that makes
$\mathcal{J}_{\delta}$ stationary. This can be done by differentiating with
respect to $\delta$ and setting the result to zero, i.e., 
\begin{equation}\label{eq:stationaryOne}
\frac{d}{d\delta} \big( \mathcal{J}_{\delta} \big)=  0.
\end{equation}
Using the chain rule,
$$ 
\frac{d}{d\delta} \big( \mathcal{J}_{\delta} \big)(\epsilon)= 
\int_a^b \bigg[ \frac{\partial \integrand  }{\partial \stateyDiscrete } \big(\statezBar(t,\epsilon),\statezDotBar(t,\epsilon) ,t\big) \frac{\partial \statezBar }{\partial \delta}(t,\epsilon)  + \frac{\partial \integrand}{\partial \stateyDotDiscrete } \big(\statezBar(t,\epsilon),\statezDotBar(t,\epsilon) ,t \big) \frac{\partial \statezDotBar}{\partial \delta} (t,\epsilon)\bigg]dt . $$
Noting that
$$\frac{\partial \statezBar}{\partial \delta} (t,\cdot)= \variation(t), \qquad \frac{\partial \statezDotBar }{\partial \delta}(t,\cdot)= \variationDot(t),$$
where $\variationDot \equiv d \variation / d\timeDummy$, 
we have
$$
\frac{d}{d\delta} \big( \mathcal{J}_{\delta} \big)(\epsilon)
= \int_a^b \bigg[ \frac{\partial \integrand  }{\partial \stateyDiscrete } \big(\statezBar(t,\epsilon),\statezDotBar(t,\epsilon),t \big) \variation(t)  + \frac{\partial \integrand}{\partial \stateyDotDiscrete } \big(\statezBar(t,\epsilon),\statezDotBar(t,\epsilon) ,t \big) \variationDot(t) \bigg]dt. $$
We integrate the second term by parts,
\begin{multline}\label{eq:gradient}
\frac{d}{d\delta} \big( \mathcal{J}_{\delta} \big)(\epsilon)
	=  \int_a^b  \bigg[\frac{\partial \integrand  }{\partial \stateyDiscrete }
	\big(\statezBar(t,\epsilon),\statezDotBar(t,\epsilon) , t \big)
	\variation(t) - \frac{d}{dt}\bigg( \frac{\partial \integrand}{\partial
	\stateyDotDiscrete }  \big(\statezBar(t,\epsilon),\statezDotBar(t,\epsilon)
	,t \big) \bigg) \variation(t) \bigg]dt + \\ 
\frac{\partial \integrand}{\partial \stateyDotDiscrete} \big(\statezBar(b,\epsilon),\statezDotBar(b,\epsilon) ,b \big) \variation(b) , 
\end{multline}
where we have used $\variationArg{a} = \bz$.
Substituting Eq.~\eqref{eq:gradient} in the stationarity condition~\eqref{eq:stationaryOne} yields
\begin{equation}\label{eq:statOne}
  \int_a^b  \bigg[\frac{\partial \integrand  }{\partial \stateyDiscrete } \big(\statezBar(t,\epsilon),\statezDotBar(t,\epsilon),t \big) \variation(t) - \frac{d}{dt}\bigg( \frac{\partial \integrand}{\partial \stateyDotDiscrete }  \big(\statezBar(t,\epsilon),\statezDotBar(t,\epsilon),t \big) \bigg) \variation(t) \bigg]dt + \frac{\partial \integrand}{\partial \stateyDotDiscrete} \big(\statezBar(b,\epsilon),\statezDotBar(b,\epsilon),b \big) \variation(b)  = 0. 
\end{equation}
By construction, $\statez(\cdot) \equiv \statezBar(\cdot;0)$ and $\statezDot(\cdot) \equiv \statezDotBar(\cdot;0)$, comprise a stationary
point; thus, setting $\epsilon = 0$ in Eq.~\eqref{eq:statOne} yields
\begin{equation}\label{eq:euler_lagrange_analysis}
  \int_a^b  \bigg[\frac{\partial \integrand  }{\partial \stateyDiscrete } \big(\statez(t),\statezDot(t),t \big) \variation(t) - \frac{d}{dt}\bigg( \frac{\partial \integrand}{\partial \stateyDotDiscrete }  \big(\statez(t),\statezDot(t),t \big) \bigg) \variation(t) \bigg]dt + \frac{\partial \integrand}{\partial \stateyDotDiscrete} \big(\statez(b),\statezDot(b) ,b\big) \variation(b)  = 0. 
\end{equation}
As $\variation$ is an arbitrary function, this equality requires
\begin{equation}\label{eq:euler_lagrange_appendix}
\begin{split}
\bigg[ &\frac{\partial \integrand  }{\partial \stateyDiscrete }
	\big(\statez(t),\statezDot(t) ,t\big) \bigg]^T - \bigg[ \frac{d}{dt}\bigg(
	\frac{\partial \integrand}{\partial \stateyDotDiscrete }
	\big(\statez(t),\statezDot(t),t \big) \bigg) \bigg]^T = \bz , \\
&\statez(a) = \statez_a, \hspace{0.65 in} \bigg[ \frac{\partial \integrand}{\partial \stateyDotDiscrete} \big(\statez(b),\statezDot(b),b \big) \bigg]^T = \bz.
\end{split}
\end{equation}
Equation~\eqref{eq:euler_lagrange_appendix} is known as the Euler--Lagrange
equation. It states that, for ($\statez$,$\statezDot$) to define a stationary point of $\mathcal{J}$, then they must 
satisfy~\eqref{eq:euler_lagrange_appendix}. It is emphasized
that~\eqref{eq:euler_lagrange_appendix} is a necessary condition on ($\statez,\statezDot$)
to make $\mathcal{J}$ stationary, but it is not a sufficient condition. It is
additionally noted that~\eqref{eq:euler_lagrange_appendix} provides a
stationary point of $\mathcal{J}$, but the resulting stationary point could be a local minima, local maxima, or saddle point.

\section{Evaluation of gradients in the Euler--Lagrange equations for \methodAcronym\ with \spatialAcronym\ trial subspaces}\label{appendix:vector_calc}
We now derive the specific form of the Euler--Lagrange equations for the \methodAcronym\ formulation with \spatialAcronym\ trial subspaces. Without loss of generality, we present the derivation for a single window $t \in [0,T]$ with a constant basis $\basisspace$ and weighting matrix $\stweightingMat$.  
To obtain the specific form of the Euler--Lagrange equations for the \methodAcronym\ formulation, we need to evaluate the gradients in~\eqref{eq:euler_lagrange_appendix} for the integrand
\begin{align*}\label{eq:integrand_apx}
 \minintegrand & \vcentcolon
(\genstateyDiscreteArgnt{}, \genstateyDiscreteDotArgnt{} ,\timeDummy) \mapsto \frac{1}{2} \big[
\basisspace \genstateyDiscreteDotArgnt{} - \velocity(\basisspace \genstateyDiscreteArgnt{}
+ \stateIntercept,\timeDummy) \big]^T \stweightingMatArg{} \big[
\basisspace \genstateyDiscreteDotArgnt{}  - \velocity(\basisspace \genstateyDiscreteArgnt{} +
\stateIntercept,\timeDummy) \big], \\ & 
\vcentcolon \RR{\romdim} \times \RR{\romdim} \times [0,T]
 \rightarrow \RR{} .  
\end{align*}
To evaluate the gradients, we first expand $\minintegrand$:
\begin{align*}
 \minintegrand(\genstateyDiscrete,\genstateyDiscreteDot,\timeDummy)  &= \frac{1}{2} \big[ \basisspace \genstateyDiscreteDotArgnt{}  - \velocity(\basisspace \genstateyDiscreteArgnt{} + \stateIntercept,\timeDummy) \big]^T \stweightingMatArg{} \big[ \basisspace \genstateyDiscreteDotArgnt{} - \velocity(\veloargsromy) \big] \\ 
 &= \frac{1}{2}\big[\basisspace \genstateyDiscreteDotArgnt{}  \big]^T  \stweightingMatArg{}  \big[\basisspace \genstateyDiscreteDotArgnt{} \big]  - \frac{1}{2}\big[ \velocity(\veloargsromy) \big]^T \stweightingMatArg{}  \big[\basisspace \genstateyDiscreteDotArgnt{}\big] - \frac{1}{2} \big[ \basisspace \genstateyDiscreteDotArgnt{} \big]^T \stweightingMatArg{} \big[ \velocity(\veloargsromy) \big]  \\ &+ \frac{1}{2}\big[\velocity(\veloargsromy) \big]^T \stweightingMatArg{} \big[ \velocity(\veloargsromy) \big].
\end{align*}
Since $\stweightingMat$ is symmetric,
\begin{equation}\label{eq:int_expand}
 \minintegrand(\genstateyDiscrete,\genstateyDiscreteDot,\timeDummy)  = \frac{1}{2}\big[\basisspace \genstateyDiscreteDotArgnt{}  \big]^T  \stweightingMatArg{}  \big[ \basisspace \genstateyDiscreteDotArgnt{} \big]  -  \big[ \basisspace \genstateyDiscreteDotArgnt{} \big]^T \stweightingMatArg{} \big[ \velocity(\veloargsromy) \big]  + \frac{1}{2}\big[\velocity(\veloargsromy) \big]^T \stweightingMatArg{} \big[ \velocity(\veloargsromy) \big].
\end{equation}
For notational purposes, we write the above as,
$$  \minintegrand(\genstateyDiscrete,\genstateyDiscreteDot,\timeDummy) =  \minintegrand_1(\genstateyDiscrete,\genstateyDiscreteDot,\timeDummy) +  \minintegrand_2(\genstateyDiscrete,\genstateyDiscreteDot,\timeDummy) +  \minintegrand_3(\genstateyDiscrete,\genstateyDiscreteDot,\timeDummy),$$
where
\begin{align*}
&\minintegrand_1(\genstateyDiscrete,\genstateyDiscreteDot,\timeDummy) =  \frac{1}{2}\big[\basisspace \genstateyDiscreteDotArgnt{}  \big]^T  \stweightingMatArg{}  \big[ \basisspace \genstateyDiscreteDotArgnt{} \big], \\
&   \minintegrand_2(\genstateyDiscrete,\genstateyDiscreteDot,\timeDummy) = -  \big[ \basisspace \genstateyDiscreteDotArgnt{} \big]^T \stweightingMatArg{} \big[ \velocity(\veloargsromy) \big], \\
&  \minintegrand_3(\genstateyDiscrete,\genstateyDiscreteDot,\timeDummy) = \frac{1}{2}\big[\velocity(\veloargsromy) \big]^T \stweightingMatArg{} \big[ \velocity(\veloargsromy) \big] . 
\end{align*}
Constructing the Euler--Lagrange equations for this functional $\minintegrand$ requires evaluating the derivatives $\frac{\partial \minintegrand}{\partial \genstateyDiscrete}$ and $\frac{\partial \minintegrand}{\partial \genstateyDiscreteDot}$. We start by evaluating $\frac{\partial \minintegrand}{\partial \genstateyDiscrete}$ and go term by term.

Starting with $\minintegrand_1(\genstateyDiscrete,\genstateyDiscreteDot)$, we see,
$$\frac{\partial \minintegrand_1 }{\partial \genstateyDiscrete} = \boldsymbol 0,$$
where it is noted that $\mathcal{I}_1$ only depends on $\genstateyDiscreteDot$. Working with the second term:
\begin{align*}
\frac{\partial \minintegrand_2}{\partial \genstateyDiscrete}  &= -\frac{\partial}{\partial \genstateyDiscrete} \bigg( \big[ \basisspace \genstateyDiscreteDot \big]^T \stweightingMatArg{} \big[ \velocity(\veloargsromy) \big] \bigg) \\ 
&= - \big[ \basisspace \genstateyDiscreteDotArgnt{} \big]^T \stweightingMatArg{} \frac{\partial}{\partial \genstateyDiscreteArgnt{}} \bigg( \big[ \velocity(\veloargsromy) \big]  \bigg)\\
 &= -\big[ \basisspace \genstateyDiscreteDotArgnt{} \big]^T \stweightingMatArg{}   \frac{\partial \velocity}{\partial \stateyDiscrete}\frac{\partial \stateyDiscrete}{\partial \genstateyDiscrete} \\
& = - \big[ \basisspace \genstateyDiscreteDotArgnt{} \big]^T \stweightingMatArg{}  \frac{\partial \velocity}{\partial \stateyDiscrete}\basisspace \\
&= -[\genstateyDiscreteDotArgnt{}]^T \basisspace^T \stweightingMatArg{} \frac{\partial \velocity}{\partial \stateyDiscrete}  \basisspace,
\end{align*}
where we have suppressed the arguments of the Jacobian for simplicity; e.g., formally 
$$\frac{\partial \velocity}{\partial \stateyDiscrete} : (\statewDiscrete,\timeDummy) \mapsto \frac{ \partial \velocity }{ \partial \stateyDiscrete} (\statewDiscrete, \timeDummy).$$ 
For $\minintegrand_3$,
\begin{align}\label{eq:i3sol}
\frac{\partial  \minintegrand_3}{\partial \genstateyDiscrete}  &= \frac{1}{2} \frac{\partial }{\partial \genstateyDiscrete} \bigg( \big[\velocity(\veloargsromy) \big]^T \stweightingMatArg{}\big[ \velocity(\veloargsromy) \big] \bigg) \\  
&=  [\velocity(\veloargsromy) ]^T \stweightingMatArg{} \frac{\partial \velocity}{\partial \stateyDiscrete} \frac{\partial \stateyDiscrete}{\partial \genstateyDiscrete} \\
 &=   [\velocity(\veloargsromy) ]^T  \stweightingMatArg{} \frac{\partial \velocity}{\partial \stateyDiscrete} \basisspace.
\end{align}
This gives the final expression,
$$
 \frac{\partial \minintegrand}{\partial \boldsymbol \genstateyDiscrete} = - \big[ \basisspace \genstateyDiscreteDotArgnt{} \big]^T \stweightingMatArg{}  \frac{\partial \velocity}{\partial \stateyDiscrete}\basisspace +  [\velocity(\veloargsromy) ]^T \stweightingMatArg{} \frac{\partial \velocity}{\partial \stateyDiscrete} \basisspace.
$$%
We now evaluate $\frac{\partial \minintegrand}{\partial \genstateyDiscreteDotArgnt{}}$ and again go term by term. Starting with $\minintegrand_1$,
\begin{align*}
\frac{\partial \minintegrand_1}{\partial \genstateyDiscreteDotArgnt{}} &=
\frac{1}{2} \frac{\partial}{\partial \genstateyDiscreteDotArgnt{} }\bigg( \big[ \basisspace \genstateyDiscreteDotArgnt{} \big]  \stweightingMatArg{}  \big[ \basisspace \genstateyDiscreteDotArgnt{} \big]  \bigg) \\
&= [\genstateyDiscreteDotArgnt{}]^T \basisspace^T \stweightingMatArg{} \basisspace.
\end{align*}
Now working with the second term:
\begin{align*}
\frac{\partial \minintegrand_2}{\partial \genstateyDiscreteDotArgnt{} } &=
\frac{\partial}{\partial \genstateyDiscreteDotArgnt{}  } \bigg( \big[ \basisspace \genstateyDiscreteDotArgnt{} \big]^T \stweightingMatArg{} \big[ \velocity(\veloargsromy) \big]  \bigg)\\
 &= \frac{\partial}{\partial \genstateyDiscreteDotArgnt{} } \bigg( [\genstateyDiscreteDotArgnt{}]^T \bigg) \basisspace^T  \stweightingMatArg{} \big[ \velocity(\veloargsromy) \big] \\ 
 &= \bigg[  \basisspace^T  \stweightingMatArg{}  \velocity(\veloargsromy )   \bigg]^T \\
 &= [\velocity(\veloargsromy) ]^T \stweightingMatArg{} \basisspace.
\end{align*}
Finally, for the last term,
\begin{align*}
\frac{\partial \minintegrand_3}{\partial \genstateyDiscreteDotArgnt{} } &= \boldsymbol 0,
\end{align*}
where it is noted that $\mathcal{I}_3$ only depends on $\genstateyDiscrete$. We thus have,
$$\frac{\partial \minintegrand}{\partial \genstateyDiscreteDot } =   [\genstateyDiscreteDotArgnt{}]^T \basisspace^T \stweightingMatArg{} \basisspace -  [\velocity(\veloargsromy ) ]^T \stweightingMatArg{} \basisspace.$$
Combining all terms and evaluating at $(\genstateArg{}{t},\genstateDotArg{}{t},t)$,
\begin{multline*}
 \frac{\partial \minintegrand}{\partial \genstateyDiscrete}(\genstateArg{}{t},\genstateDotArg{}{t},t)  - \frac{d}{dt} \bigg[ \frac{\partial \minintegrand}{\partial \genstateyDiscreteDotArgnt{}} (\genstateArg{}{t},\genstateDotArg{}{t},t) \bigg] =  - \big[ \basisspace \genstateDotArg{}{t}  \big]^T \stweightingMatArg{} \big[  \frac{\partial \velocity}{\partial \stateyDiscrete} (\veloargsrom) \big] \basisspace + \\  [\velocity(\veloargsrom) ]^T \stweightingMatArg{} \big[ \frac{\partial \velocity}{\partial \stateyDiscrete} (\veloargsrom) \big] \basisspace - \frac{d}{dt} \bigg[  [\genstateDotArg{}{t} ]^T  \basisspace^T \stweightingMatArg{} \basisspace - [\velocity(\veloargsrom) ]^T \stweightingMatArg{} \basisspace \bigg]  = \boldsymbol 0.
\end{multline*}
To put this in a more recognizable form, we can pull out the common factor in the first two terms,
\begin{multline*}
- \bigg( \big[ \basisspace \genstateDotArg{}{t} \big]^T  -  [\velocity(\veloargsrom)] ^T \bigg) \stweightingMatArg{}
 \frac{\partial \velocity}{\partial \stateyDiscrete}(\veloargsrom) \basisspace -  \\
\frac{d}{dt} \bigg[  [\genstateDotArg{}{t}]^T \basisspace^T \stweightingMatArg{} \basisspace -  
[\velocity(\veloargsrom) ]^T \stweightingMatArg{} \basisspace  \bigg] = \boldsymbol 0.
\end{multline*}
Taking the transpose to put into the common column major format,
$$ -  \basisspace^T \big[ \frac{\partial \velocity}{\partial \stateyDiscrete}(\veloargsrom)  \big]^T \stweightingMatArg{} \bigg( \big[ \basisspace \genstateDotArg{}{t}\big]  -  \velocity(\veloargsrom)  \bigg) -  \frac{d}{dt} \bigg[  \basisspace^T \stweightingMatArg{} \basisspace \genstateDotArg{}{t}  - \basisspace^T \stweightingMatArg{} \velocity(\veloargsrom)   \bigg] = \bz.
 $$
We now factor the second term,
$$ -  \basisspace^T \big[ \frac{\partial \velocity}{\partial \stateyDiscrete}(\veloargsrom)  \big]^T \stweightingMatArg{ } \bigg(  \basisspace \genstateDotArg{ }{t}  -  \velocity(\veloargsrom) \bigg) -  \basisspace^T \stweightingMatArg{ } \frac{d}{dt} \bigg[   \basisspace \genstateDotArg{}{t}  - \velocity(\veloargsrom) \bigg] = \bz. $$
Gathering terms and multiplying by negative one, the final form of the Euler--Lagrange equations are obtained,
\begin{equation}\label{eq:el_secondorder}
 \bigg[\basisspace^T \big[\frac{\partial \velocity}{\partial \stateyDiscrete}(\veloargsrom) \big]^T \stweightingMatArg{} + \basisspace^T \stweightingMatArg{} \frac{d}{dt} \bigg] \bigg(  \basisspace \genstateDotArg{}{t} -  \velocity(\veloargsrom) \bigg) = \bz.
\end{equation} 
This is the \methodAcronym-ROM. Note that this is a second order equation and can be written as two separate first order equations. Defining the ``costate" as,
\begin{equation*}
\adjointArgnt{}  : \timeDummy \mapsto \genstateDotArg{}{\timeDummy}  -  \mass^{-1} \basisspace^T \stweightingMatArg{} \velocity(\basisspace \genstateArg{}{t} + \stateIntercept, \timeDummy) ,
\end{equation*}
we can manipulate Eq.~\eqref{eq:el_secondorder} as follows: 
First, we add and subtract the first term multiplied by $\basisspace [\massArgnt{n}]^{-1}\basisspace^T \stweightingMat$, 
\begin{multline*} 
\bigg[\basisspace^T \bigg[\frac{\partial \velocity}{\partial \stateyDiscrete}(\veloargsrom) \bigg]^T \stweightingMatArg{} \bigg( \mathbf{I} - \basisspace [\massArgnt{}]^{-1} \basisspace^T\stweightingMatArg{} + \basisspace [\massArgnt{}]^{-1} \basisspace^T \stweightingMatArg{} \bigg)  + \basisspace^T \stweightingMatArg{}  \frac{d}{dt} \bigg] \\ 
\bigg(  \basisspace \genstateDotArg{}{t}   -  \velocity(\veloargsrom) \bigg) = \bz.
\end{multline*}
Pulling out the term multiplied by the positive portion of $\basisspace [\massArgnt{}]^{-1} \basisspace^T \stweightingMatArg{}$,
\begin{multline*} 
\bigg[\basisspace^T \bigg[\frac{\partial \velocity}{\partial \stateyDiscrete} ( \veloargsrom) \bigg]^T \stweightingMatArg{}\bigg( \mathbf{I} - \basisspace [\massArgnt{}]^{-1}  \basisspace^T  \stweightingMatArg{} \bigg)  + \\ \basisspace^T \bigg[\frac{\partial \velocity}{\partial \stateyDiscrete} (\veloargsrom) \bigg]^T \stweightingMatArg{} \basisspace [\massArgnt{}]^{-1}  \basisspace^T \stweightingMatArg{} +   \basisspace^T \stweightingMatArg{} \frac{d}{dt} \bigg] \bigg(  \basisspace \genstateDotArg{}{t}   -  \velocity(\veloargsrom) \bigg) = \bz.
\end{multline*}
Splitting into two separate terms, 
\begin{multline*}
\basisspace^T \bigg[\frac{\partial \velocity}{\partial \stateyDiscrete}(\veloargsrom) \bigg]^T \stweightingMatArg{}\bigg( \mathbf{I} - \basisspace [\massArgnt{}]^{-1}  \basisspace^T \stweightingMatArg{} \bigg)  \bigg(  \basisspace \genstateDotArg{}{t}   -  \velocity(\veloargsrom) \bigg)  + \\  
\bigg[ \basisspace^T \bigg[\frac{\partial \velocity}{\partial \stateyDiscrete}(\veloargsrom) \bigg]^T \stweightingMatArg{} \basisspace [\massArgnt{}]^{-1} \basisspace^T \stweightingMatArg{} +   \basisspace^T \stweightingMatArg{} \frac{d}{dt} \bigg] \bigg(  \basisspace \genstateDotArg{}{t}  -  \velocity(\veloargsrom) \bigg) = \bz.
\end{multline*}
Pulling $\mass^{-1} \basisspace^T \stweightingMatArg{}$ inside the parenthesis on the second term,
\begin{multline*}
\basisspace^T \bigg[\frac{\partial \velocity}{\partial \state^n}(\veloargsrom) \bigg]^T \stweightingMatArg{}\bigg( \mathbf{I} - \basisspace[\massArgnt{}]^{-1} \basisspace^T \stweightingMatArg{} \bigg)  \bigg(  \basisspace \genstateDotArg{}{t}   -  \velocity(\veloargsrom) \bigg)  + \\  
\bigg[ \basisspace^T \bigg[\frac{\partial \velocity}{\partial \stateyDiscrete}(\veloargsrom) \bigg]^T \stweightingMatArg{} \basisspace  +   \mass \frac{d}{dt} \bigg] \bigg( \genstateDotArg{}{t}  - \mass^{-1} \basisspace^T \stweightingMatArg{}  \velocity(\veloargsrom) \bigg) = \bz.
\end{multline*}
By definition, the term inside the parenthesis of the second term is $\adjointArg{}{t}$,
\begin{multline*}
\basisspace^T \bigg[\frac{\partial \velocity}{\partial \stateyDiscrete}(\veloargsrom) \bigg]^T \stweightingMatArg{}\bigg( \mathbf{I} - \basisspace [\massArgnt{}]^{-1} \basisspace^T \stweightingMatArg{} \bigg)  \bigg(  \basisspace \genstateDotArg{}{t}    -  \velocity(\veloargsrom) \bigg)  +   \\ \basisspace^T \bigg[\frac{\partial \velocity}{\partial \stateyDiscrete}(\veloargsrom) \bigg]^T \stweightingMatArg{} \basisspace \adjointArg{}{t}  +  \massArgnt{} \frac{d}{dt} \adjointArg{}{t}= \bz.
\end{multline*}
Re-arranging,
\begin{multline*}
\mass \frac{d }{dt}\adjointArg{}{t} + \basisspace^T \bigg[\frac{\partial \velocity}{\partial \stateyDiscrete} ( \veloargsrom ) \bigg]^T \stweightingMatArg{} \basisspace  \adjointArg{}{t}  \\
= - \basisspace^T \bigg[\frac{\partial \velocity}{\partial \stateyDiscrete}(\veloargsrom)\bigg]^T \stweightingMatArg{}\bigg( \mathbf{I} - \basisspace [\massArgnt{}]^{-1} \basisspace^T \stweightingMatArg{} \bigg)  \bigg(  \basisspace \genstateDotArg{}{t}   -  \velocity(\veloargsrom) \bigg).
\end{multline*}
We thus get the splitting
\begin{align*}\label{eq:lspg_continuous_appendix}
& \mass \frac{d}{dt}\genstateArg{}{t}  -  \basisspace^T \stweightingMatArg{} \velocity(\veloargsrom) =  \mass \adjointArg{}{t},\\
&\mass \frac{d}{dt} \adjointArg{}{t}  + \basisspace^T \bigg[\frac{\partial \velocity}{\partial \stateyDiscrete} (\veloargsrom) \bigg]^T \stweightingMatArg{} \basisspace \adjointArg{}{t} = \\
& -\basisspace^T \big[ \frac{\partial \velocity}{\partial \stateyDiscrete}(\veloargsrom) \big] ^T \stweightingMatArg{}  \bigg( \mathbf{I} -   \basisspace [\massArgnt{}]^{-1} \basisspace^T   \stweightingMatArg{} \bigg)\bigg( \basisspace \genstateDotArg{}{t}   -   \velocity(\veloargsrom) \bigg) . 
\end{align*}
\section{Evaluation of gradients for optimal control formulation}\label{appendix:optimal_control}
When formulated as an optimal control problem of Lagrange type, the gradients of the Hamiltonian with respect to the state, controller, and costate 
need to be evaluated. This section details this evaluation. The case derivation is presented for the case with one window, for notational simplicity.

The Pontryagin Maximum Principle leverages the following Hamiltonian,
\begin{align*}
\hamiltonian \; &: \;  (\genstateyDiscreteArgnt{},\adjointDiscreteDumArgnt{},\controllerDiscreteDumArgnt{},\timeDummy) \mapsto 
 \adjointDiscreteDumArgnt{T} \bigg[  [\massArgnt{}]^{-1}\basisspace^T \stweightingMatArg{}\velocity(\veloargsromy) + [\massArgnt{}]^{-1}\controllerDiscreteDumArgnt{} \bigg] +  \objectiveControlArg{}(\genstateyDiscreteArgnt{},\controllerDiscreteDumArgnt{},\timeDummy) \\
&: \; \RR{\romdim} \times \RR{\romdim} \times \RR{\romdim} \times [0,T] \rightarrow \RR{},
\end{align*} 
where,
\begin{align*}
 \objectiveControlArg{} &:  (\genstateyDiscreteArgnt{},\controllerDiscreteDumArgnt{},\timeDummy)
\mapsto \frac{1}{2} \bigg[ \basisspace \bigg(  [\massArgnt{}]^{-1}\basisspace^T
\stweightingMatArg{}  \velocity(\veloargsromy) + [\massArgnt{}]^{-1}\controllerDiscreteDumArgnt{} \bigg) -
\velocity(\veloargsromy) \bigg]^T
\stweightingMatArg{}  \\ & \hspace{1.5 in}\bigg[ \basisspace \bigg(
[\massArgnt{}]^{-1}\basisspace^T \stweightingMatArg{}\velocity(\veloargsromy) + [\massArgnt{}]^{-1}\controllerDiscreteDumArgnt{}
\bigg) - \velocity( \veloargsromy ) \bigg]
, \nonumber \\ & : \RR{\romdim} \times \RR{\romdim} \times [0,T] \rightarrow \RR{}.
\end{align*} 
As described in Section~\ref{sec:optimal_control}, to derive the stationary conditions of the \methodAcronym\ objective function, we require evaluating the following gradients,
$ \frac{\partial \hamiltonianArg{}}{\partial \adjointDiscreteDumArgnt{}{}}, \;   \frac{\partial \hamiltonianArg{}}{\partial \genstateyDiscreteArgnt{}{}}$, and $\frac{\partial \hamiltonianArg{}}{\partial \controllerDiscreteDumArgnt{}{}}.$  Starting with $\frac{\partial \hamiltonianArg{}}{\partial \adjointDiscreteDumArgnt{}{}}$, we have,
$$
\bigg[\frac{\partial \hamiltonianArg{}}{\partial \adjointDiscreteDumArgnt{}{}}\bigg]^T = [\massArgnt{}]^{-1}\basisspace^T \stweightingMatArg{}\velocity(\veloargsromy) + [\massArgnt{}]^{-1}\controllerDiscreteDumArgnt{} .
$$
Next, we address $ \frac{\partial \hamiltonianArg{}}{\partial \genstateyArgnt{}{}}$. 
\begin{align*}
  \frac{\partial \hamiltonianArg{}}{\partial \genstateyDiscreteArgnt{}{}} &= \frac{\partial}{\partial \genstateyDiscreteArgnt{}{}} \bigg[  \adjointDiscreteDumArgnt{T} \big[  [\massArgnt{}]^{-1}\basisspace^T \stweightingMatArg{}\velocity(\veloargsromy) + [\massArgnt{}]^{-1}\controllerDiscreteDumArgnt{} \big] +  \objectiveControlArg{}(\genstateyDiscreteArgnt{},\controllerDiscreteDumArgnt{},\timeDummy) \bigg] \\
 &= \frac{\partial}{\partial \genstateyDiscreteArgnt{}{}} \bigg[  \adjointDiscreteDumArgnt{T} \big[  [\massArgnt{}]^{-1}\basisspace^T \stweightingMatArg{}\velocity(\veloargsromy
) + [\massArgnt{}]^{-1}\controllerDiscreteDumArgnt{} \big] \bigg]+  \frac{\partial}{\partial \genstateyDiscreteArgnt{}{}}  \bigg[ \objectiveControlArg{}(\genstateyDiscreteArgnt{},\controllerDiscreteDumArgnt{},\timeDummy) \bigg] \\
 &=   \adjointDiscreteDumArgnt{T} \big[  [\massArgnt{}]^{-1}\basisspace^T \stweightingMatArg{}  \frac{\partial}{\partial \genstateyDiscreteArgnt{}{}} \bigg( \velocity(\veloargsromy)\bigg) \big] +  \frac{\partial}{\partial \genstateyDiscreteArgnt{}{}}  \bigg[ \objectiveControlArg{}(\genstateyDiscreteArgnt{},\controllerDiscreteDumArgnt{},\timeDummy) \bigg] \\
 &=   \adjointDiscreteDumArgnt{T} \big[  [\massArgnt{}]^{-1}\basisspace^T \stweightingMatArg{}  \frac{\partial \velocity}{\partial \stateyDiscrete} \frac{\partial \stateyDiscreteArgnt{}{}}{\partial  \genstateyDiscreteArgnt{}{} } \big] +  \frac{\partial}{\partial \genstateyDiscreteArgnt{}{}}  \bigg[ \objectiveControlArg{}(\genstateyDiscreteArgnt{},\controllerDiscreteDumArgnt{},\timeDummy) \bigg] \\
 &=   \adjointDiscreteDumArgnt{T} \big[  [\massArgnt{}]^{-1}\basisspace^T \stweightingMatArg{}  \frac{\partial \velocity}{\partial \stateyDiscrete}\basisspace  \big] +  \frac{\partial}{\partial \genstateyDiscreteArgnt{}{}}  \bigg[ \objectiveControlArg{}(\genstateyDiscreteArgnt{},\controllerDiscreteDumArgnt{},\timeDummy) \bigg]. 
\end{align*}
To evaluate $\frac{\partial}{\partial \genstateyDiscreteArgnt{}{}}  \bigg[ \objectiveControlArg{}(\genstateyDiscreteArgnt{},\controllerDiscreteDumArgnt{},\timeDummy) \bigg]$, we leverage the previous result~\eqref{eq:int_expand} and insert $\genstateyDiscreteDotArgnt{} =  [\massArgnt{}]^{-1}\basisspace^T
\stweightingMatArg{}  \velocity(\veloargsromy) + [\massArgnt{}]^{-1}\controllerDiscreteDumArg{}{t} .$ This leads to the expression for the expanded Lagrangian,
\begin{multline*}
 \objectiveControlArg{}(\genstateyDiscreteArgnt{},\controllerDiscreteDumArgnt{},\timeDummy) = \frac{1}{2}\big[\basisspace  [\massArgnt{}]^{-1}\basisspace^T
\stweightingMatArg{}  \velocity(\veloargsromy)   + \basisspace [\massArgnt{}]^{-1}\controllerDiscreteDumArg{}{t} \big]^T  \stweightingMatArg{}  \big[ \basisspace  [\massArgnt{}]^{-1}\basisspace^T
\stweightingMatArg{}  \velocity(\veloargsromy)   + \basisspace [\massArgnt{}]^{-1}\controllerDiscreteDumArg{}{t}\big] \\  -  \big[ \basisspace  [\massArgnt{}]^{-1}\basisspace^T
\stweightingMatArg{}  \velocity(\veloargsromy)  + \basisspace [\massArgnt{}]^{-1}\controllerDiscreteDumArg{}{t} \big]^T \stweightingMatArg{} \big[ \velocity(\veloargsromy) \big]  + \frac{1}{2}\big[\velocity(\veloargsromy) \big]^T \stweightingMatArg{} \big[ \velocity(\veloargsromy) \big].
\end{multline*}
Again for notational purposes, we split this into three terms,
$$
\objectiveControlArg{}(\genstateyDiscreteArgnt{},\controllerDiscreteDumArgnt{},\timeDummy)  = 
\objectiveControlArg{}_1(\genstateyDiscreteArgnt{},\controllerDiscreteDumArgnt{},\timeDummy)  + 
\objectiveControlArg{}_2(\genstateyDiscreteArgnt{},\controllerDiscreteDumArgnt{},\timeDummy)  + 
\objectiveControlArg{}_3(\genstateyDiscreteArgnt{},\controllerDiscreteDumArgnt{},\timeDummy) ,$$
where
\begin{align*}
& \objectiveControlArg{}_1(\genstateyDiscreteArgnt{},\controllerDiscreteDumArgnt{},\timeDummy) =  \frac{1}{2}\big[\basisspace  [\massArgnt{}]^{-1}\basisspace^T
\stweightingMatArg{}  \velocity(\veloargsromy)  + \basisspace [\massArgnt{}]^{-1}\controllerDiscreteDumArgnt{}{}  \big]^T  \stweightingMatArg{}  \big[ \basisspace  [\massArgnt{}]^{-1}\basisspace^T
\stweightingMatArg{}  \velocity(\veloargsromy)   +  \basisspace [\massArgnt{}]^{-1}\controllerDiscreteDumArgnt{}{} \big],\\
& \objectiveControlArg{}_2(\genstateyDiscreteArgnt{},\controllerDiscreteDumArgnt{},\timeDummy) = -  \big[ \basisspace  [\massArgnt{}]^{-1}\basisspace^T
\stweightingMatArg{}  \velocity(\veloargsromy)  + \basisspace [\massArgnt{}]^{-1} \controllerDiscreteDumArgnt{}\big]^T \stweightingMatArg{} \big[ \velocity(\veloargsromy) \big] ,\\
& \objectiveControlArg{}_3(\genstateyDiscreteArgnt{},\controllerDiscreteDumArgnt{},\timeDummy) = \frac{1}{2}\big[\velocity(\veloargsromy) \big]^T \stweightingMatArg{} \big[ \velocity(\veloargsromy) \big] .
\end{align*}
Evaluating the first term,
\begin{align*}
\frac{\partial  \objectiveControlArg{}_1 }{\partial \genstateyDiscreteArgnt{}{}} &= 
\frac{1}{2} \frac{\partial }{\partial \genstateyDiscreteArgnt{}{}} \bigg( [\velocity(\veloargsromy)]^T \stweightingMatArg{} \basisspace [\massArgnt{}]^{-1} \basisspace^T \stweightingMatArg{}  \big[ \basisspace  [\massArgnt{}]^{-1}\basisspace^T
\stweightingMatArg{}  \velocity(\veloargsromy) \big] \bigg) +  \\
& \qquad \frac{\partial }{\partial \genstateyDiscreteArgnt{}{}} \bigg( [\controllerDiscreteDumArgnt{}]^T [\massArgnt{}]^{-1} \basisspace^T \stweightingMatArg{}  \big[ \basisspace  [\massArgnt{}]^{-1}\basisspace^T
\stweightingMatArg{}  \velocity(\veloargsromy) \big] \bigg) + \\
& \qquad \frac{1}{2} \frac{\partial }{\partial \genstateyDiscreteArgnt{}{}} \bigg( [\controllerDiscreteDumArgnt{}]^T [\massArgnt{}]^{-1} \basisspace^T \stweightingMatArg{} \basisspace [\massArgnt{}]^{-1} \controllerDiscreteDumArgnt{}\bigg) 
,\\ 
&= 
[\velocity(\veloargsromy)]^T \stweightingMatArg{} \basisspace [\massArgnt{}]^{-1} \basisspace^T \stweightingMatArg{}  \basisspace  [\massArgnt{}]^{-1}\basisspace^T
\stweightingMatArg{}  \frac{\partial \velocity}{\partial \stateyDiscrete} \basisspace + 
[\controllerDumArgnt{}]^T [\massArgnt{}]^{-1} \basisspace^T \stweightingMatArg{}   \basisspace  [\massArgnt{}]^{-1}\basisspace^T
\stweightingMatArg{} \frac{\partial \velocity}{\partial \statey} \basisspace , \\ 
&= 
\bigg( [\velocity(\veloargsromy)]^T \stweightingMatArg{} \basisspace [\massArgnt{}]^{-1}  + [\controllerDiscreteDumArgnt{}]^T [\massArgnt{}]^{-1} \bigg)\basisspace^T  \stweightingMatArg{}   \basisspace  [\massArgnt{}]^{-1}\basisspace^T
\stweightingMatArg{} \frac{\partial \velocity}{\partial \stateyDiscrete} \basisspace ,\\
&= 
[\genstateyDiscreteDotArgnt{} ]^T \basisspace^T  \stweightingMatArg{}   \basisspace  [\massArgnt{}]^{-1}\basisspace^T
\stweightingMatArg{} \frac{\partial \velocity}{\partial \stateyDiscrete} \basisspace , \\
&= 
[\basisspace \genstateyDiscreteDotArgnt{} ]^T   \stweightingMatArg{}   \basisspace  [\massArgnt{}]^{-1}\basisspace^T
\stweightingMatArg{} \frac{\partial \velocity}{\partial \stateyDiscrete} \basisspace .
\end{align*}
Now evaluating the second term,
\begin{align*}
\frac{\partial  \objectiveControlArg{}_2 }{\partial \genstateyDiscreteArgnt{}{}} &= 
-\frac{\partial  }{\partial \genstateyDiscreteArgnt{}{}}   \big[ \basisspace  [\massArgnt{}]^{-1}\basisspace^T
\stweightingMatArg{}  \velocity(\veloargsromy) \big]^T \stweightingMatArg{} \big[ \velocity(\veloargsromy) \big]  - \frac{\partial  }{\partial \genstateyDiscreteArgnt{}{}} \bigg( [\controllerDiscreteDumArgnt{}]^T[\massArgnt{}]^{-1} \basisspace^T \stweightingMatArg{} \big[ \velocity(\veloargsromy) \big]  \bigg)  \\&= 
-\frac{\partial  }{\partial \genstateyDiscreteArgnt{}{}}    [\velocity(\veloargsromy)]^T \stweightingMatArg{} \basisspace [\massArgnt{}]^{-1} \basisspace^T 
 \stweightingMatArg{} \big[ \velocity(\veloargsromy) \big]  - 
  [\controllerDiscreteDumArgnt{}]^T[\massArgnt{}]^{-1} \basisspace^T \stweightingMatArg{} \frac{\partial \velocity}{\partial \stateyDiscrete} \basisspace 
\\
&= 
- 2[\velocity(\veloargsromy)]^T \stweightingMatArg{} \basisspace [\massArgnt{}]^{-1} \basisspace^T \stweightingMatArg{}  
 \frac{\partial \velocity}{\partial \stateyDiscrete} \basisspace - 
 [\controllerDiscreteDumArgnt{}]^T[\massArgnt{}]^{-1} \basisspace^T \stweightingMatArg{} \frac{\partial \velocity}{\partial \stateyDiscrete} \basisspace, \\
&= - \bigg[ 2[\velocity(\veloargsromy)]^T  \stweightingMatArg{} \basisspace [\massArgnt{}]^{-1}  + [\controllerDiscreteDumArgnt{}]^T  [\massArgnt{}]^{-1} \bigg] \basisspace^T \stweightingMatArg{} \frac{\partial \velocity}{\partial \stateyDiscrete} \basisspace ,\\
&= -\bigg( \big[ \basisspace \genstateyDiscreteDotArgnt{} ]^T  + \velocity(\veloargsromy)^T  \stweightingMatArg{} \basisspace [\massArgnt{}]^{-1} \basisspace^T \bigg) \stweightingMatArg{} \frac{\partial \velocity}{\partial \stateyDiscrete} \basisspace. 
\end{align*}
Next, we can use the result~\eqref{eq:i3sol} to have,
$$ \frac{\partial  \objectiveControlArg{}_3 }{\partial \genstateyDiscreteArgnt{}{}} =
  [\velocity(\veloargsromy) ]^T  \stweightingMatArg{} \frac{\partial \velocity}{\partial \stateyDiscrete} \basisspace.
$$
Thus we have
\begin{align*}
\frac{\partial \objectiveControlArg{}}{\partial \genstateyDiscreteArgnt{}} &= 
[\basisspace \genstateyDiscreteDotArgnt{} ]^T   \stweightingMatArg{}   \basisspace  [\massArgnt{}]^{-1}\basisspace^T
\stweightingMatArg{} \frac{\partial \velocity}{\partial \stateyDiscrete} \basisspace -
\bigg( \big[ \basisspace \genstateyDiscreteDotArgnt{} ]^T  + \velocity(\veloargsromy)]^T  \stweightingMatArg{} \basisspace [\massArgnt{}]^{-1} \basisspace^T  \bigg) \stweightingMatArg{} \frac{\partial \velocity}{\partial \stateyDiscrete} \basisspace + 
 [\velocity(\veloargsromy) ]^T  \stweightingMatArg{} \frac{\partial \velocity}{\partial \stateyDiscrete} \basisspace, \\
&= 
\bigg( [\basisspace \genstateyDiscreteDotArgnt{} ]^T - \velocity(\veloargsromy)^T \bigg)  \stweightingMatArg{}   \basisspace  [\massArgnt{}]^{-1}\basisspace^T
\stweightingMatArg{} \frac{\partial \velocity}{\partial \stateyDiscrete} \basisspace -
\bigg( [\basisspace \genstateyDiscreteDotArgnt{} ]^T - \velocity(\veloargsromy)^T \bigg)  \stweightingMatArg{}
\frac{\partial \velocity}{\partial \stateyDiscrete} \basisspace, \\ 
&= 
\bigg( [\basisspace \genstateyDiscreteDotArgnt{} ]^T - \velocity(\veloargsromy)^T \bigg)   \bigg( \stweightingMatArg{} \basisspace  [\massArgnt{}]^{-1}\basisspace^T - \mathbf{I} \bigg)
\stweightingMatArg{} \frac{\partial \velocity}{\partial \stateyDiscrete} \basisspace, 
\end{align*}
such that,
\begin{equation*}
\frac{\partial \hamiltonianArg{}}{\partial \genstateyDiscreteArgnt{}{}} = 
  \adjointDumArgnt{T} \big[  [\massArgnt{}]^{-1}\basisspace^T \stweightingMatArg{}  \frac{\partial \velocity}{\partial \stateyDiscrete}\basisspace  \big]  + 
\bigg( [\basisspace \genstateyDiscreteDotArgnt{} ]^T - \velocity(\veloargsromy)^T \bigg)   \bigg( \stweightingMatArg{} \basisspace  [\massArgnt{}]^{-1}\basisspace^T - \mathbf{I} \bigg)
\stweightingMatArg{} \frac{\partial \velocity}{\partial \stateyDiscrete} \basisspace.
\end{equation*}
Equivalently we write
\begin{equation*}
\bigg[ \frac{\partial \hamiltonianArg{}}{\partial \genstateyDiscreteArgnt{}{}} \bigg]^T =  \basisspace^T \big[  \frac{\partial \velocity}{\partial \stateyDiscrete} \big]^T \stweightingMatArg{}  \basisspace [\massArgnt{}]^{-1}  \adjointDiscreteDumArgnt{} + 
 \basisspace^T \big[  \frac{\partial \velocity}{\partial \stateyDiscrete} \big]^T 
  \stweightingMatArg{} \bigg(  \basisspace  [\massArgnt{}]^{-1}\basisspace^T  \stweightingMatArg{}- \mathbf{I} \bigg)
 \bigg( \basisspace \genstateyDiscreteDotArgnt{} - \velocity(\veloargsromy) \bigg). 
\end{equation*}
Next, we evaluate $\frac{\partial \hamiltonianArg{}}{\partial \controllerDumArgnt{}}$:
\begin{align*}
  \frac{\partial \hamiltonianArg{}}{\partial \controllerDumArgnt{}{}} &= \frac{\partial}{\partial \controllerDumArgnt{}{}} \bigg[  \adjointDumArgnt{T} \big[  [\massArgnt{}]^{-1}\basisspace^T \stweightingMatArg{}\velocity(\veloargsromy) + [\massArgnt{}]^{-1}\controllerDumArgnt{} \big] +  \objectiveControlArg{}(\genstateyDiscreteArgnt{},\controllerDiscreteDumArgnt{},\timeDummy) \bigg] \\
 &= \frac{\partial}{\partial \controllerDiscreteDumArgnt{}{}} \bigg[  \adjointDiscreteDumArgnt{T} \big[  [\massArgnt{}]^{-1}\basisspace^T \stweightingMatArg{}\velocity(\veloargsromy) + [\massArgnt{}]^{-1}\controllerDiscreteDumArgnt{} \big] \bigg]+  \frac{\partial}{\partial \controllerDiscreteDumArgnt{}{}}  \bigg[ \objectiveControlArg{}(\genstateyDiscreteArgnt{},\controllerDiscreteDumArgnt{},\timeDummy) \bigg] \\
 &= [\adjointDiscreteDumArgnt{}]^T[\massArgnt{}]^{-1} + \frac{\partial}{\partial \controllerDiscreteDumArgnt{}{}}  \bigg[ \objectiveControlArg{}(\genstateyDiscreteArgnt{},\controllerDiscreteDumArgnt{},\timeDummy) \bigg]. 
\end{align*}
To evaluate $\frac{\partial}{\partial \controllerDiscreteDumArgnt{}{}}  \bigg[ \objectiveControlArg{}(\genstateyDiscreteArgnt{},\controllerDiscreteDumArgnt{},\timeDummy) \bigg]$, we again 
go term by term. Starting with the first term,
\begin{align*}
\frac{\partial  \objectiveControlArg{}_1 }{\partial \genstateyDiscreteArgnt{}{}} &= 
\frac{1}{2} \frac{\partial }{\partial \controllerDiscreteDumArgnt{}} \bigg( [\velocity(\veloargsromy)]^T \stweightingMatArg{} \basisspace [\massArgnt{}]^{-1} \basisspace^T \stweightingMatArg{}  \big[ \basisspace  [\massArgnt{}]^{-1}\basisspace^T
\stweightingMatArg{}  \velocity(\veloargsromy) \big] \bigg) + \\
& \qquad \frac{\partial }{\partial \controllerDiscreteDumArgnt{}{}} \bigg( [\controllerDiscreteDumArgnt{}]^T [\massArgnt{}]^{-1} \basisspace^T \stweightingMatArg{}  \big[ \basisspace  [\massArgnt{}]^{-1}\basisspace^T
\stweightingMatArg{}  \velocity(\veloargsromy) \big] \bigg) + \\
& \qquad \frac{1}{2} \frac{\partial }{\partial \controllerDiscreteDumArgnt{}{}} \bigg( [\controllerDiscreteDumArgnt{}]^T [\massArgnt{}]^{-1} \basisspace^T \stweightingMatArg{} \basisspace [\massArgnt{}]^{-1} \controllerDiscreteDumArgnt{}\bigg) \\
&= \bigg(  [\massArgnt{}]^{-1} \basisspace^T \stweightingMatArg{}  \big[ \basisspace  [\massArgnt{}]^{-1}\basisspace^T                      
\stweightingMatArg{}  \velocity(\veloargsromy) \big] \bigg)^T +   [\controllerDiscreteDumArgnt{}]^T [\massArgnt{}]^{-1} \basisspace^T \stweightingMatArg{} \basisspace [\massArgnt{}]^{-1} \\
&= \big[ \velocity(\veloargsromy) \big]^T \stweightingMatArg{} \basisspace [\massArgnt{}]^{-1}\basisspace^T  \stweightingMatArg{} \basisspace [\massArgnt{}]^{-1} +   [\controllerDiscreteDumArgnt{}]^T [\massArgnt{}]^{-1} \basisspace^T \stweightingMatArg{} \basisspace [\massArgnt{}]^{-1}. 
\end{align*}
Moving on to the second term,
\begin{align*}
\frac{\partial  \objectiveControlArg{}_2 }{\partial \controllerDiscreteDumArgnt{}{}} &= 
-\frac{\partial  }{\partial \controllerDiscreteDumArgnt{}{}}   \big[ \basisspace  [\massArgnt{}]^{-1}\basisspace^T
\stweightingMatArg{}  \velocity(\veloargsromy) \big]^T \stweightingMatArg{} \big[ \velocity(\veloargsromy) \big]  - \frac{\partial  }{\partial \controllerDiscreteDumArgnt{}{}} \bigg( [\controllerDiscreteDumArgnt{}]^T[\massArgnt{}]^{-1} \basisspace^T \stweightingMatArg{} \big[ \velocity(\veloargsromy) \big]  \bigg)  \\
&=- \bigg[ [\massArgnt{}]^{-1} \basisspace^T \stweightingMatArg{} \velocity(\basisspace \genstateyDiscrete + \stateIntercept) \bigg]^T \\
&= - [\velocity(\veloargsromy) ]^T \stweightingMatArg{} \basisspace [\massArgnt{}]^{-1}. 
\end{align*}
For the third term we have simply,
$$ \frac{\partial  \objectiveControlArg{}_3 }{\partial \controllerDiscreteDumArgnt{}{}} = \bz .$$
Thus,
\begin{align*}
\frac{\partial \hamiltonianArg{}}{\partial \controllerDiscreteDumArgnt{}} &= 
[\adjointDiscreteDumArgnt{}]^T[\massArgnt{}]^{-1}  +  \big[ \velocity(\veloargsromy) \big]^T \stweightingMatArg{} \basisspace [\massArgnt{}]^{-1}\basisspace^T  \stweightingMatArg{} \basisspace [\massArgnt{}]^{-1} + \\
& \qquad   [\controllerDiscreteDumArgnt{}]^T [\massArgnt{}]^{-1} \basisspace^T \stweightingMatArg{} \basisspace [\massArgnt{}]^{-1} - [\velocity(\veloargsromy) ]^T \stweightingMatArg{} \basisspace [\massArgnt{}]^{-1}, \\
& = \bigg( [\adjointDiscreteDumArgnt{}]^T - [\velocity(\veloargsromy) ]^T \stweightingMatArg{} \basisspace \bigg) [\massArgnt{}]^{-1}  + \bigg(  \big[ \velocity(\veloargsromy) \big]^T \stweightingMatArg{} \basisspace + [\controllerDiscreteDumArgnt{}]^T   \bigg) [\massArgnt{}]^{-1}\basisspace^T  \stweightingMatArg{} \basisspace [\massArgnt{}]^{-1}  \\
& = \bigg( [\adjointDiscreteDumArgnt{}]^T - [\velocity(\veloargsromy) ]^T \stweightingMatArg{} \basisspace \bigg) [\massArgnt{}]^{-1}  + \bigg(  \big[ \velocity(\veloargsromy) \big]^T \stweightingMatArg{} \basisspace + [\controllerDiscreteDumArgnt{}]^T   \bigg) [\massArgnt{}]^{-1}  \\
& =  [\adjointDiscreteDumArgnt{}]^T  [\massArgnt{}]^{-1}  +  [\controllerDiscreteDumArgnt{}]^T   [\massArgnt{}]^{-1}  
\end{align*}
Or, equivalently,
\begin{equation*}
\frac{\partial \hamiltonianArg{}}{\partial \controllerDiscreteDumArgnt{}} = 
  [\massArgnt{}]^{-1} [\adjointDiscreteDumArgnt{} +  \controllerDiscreteDumArgnt{}]. 
\end{equation*}
Evaluating at $(\genstateArg{}{t},\genstateDotArg{}{t},t$), the gradients in the Pontryagin Maximum Principle yield
\begin{align*}
&\frac{d}{dt}\genstateArg{}{t} =  [\massArgnt{}]^{-1}\basisspace^T \stweightingMatArg{}\velocity(\veloargsrom) + [\massArgnt{}]^{-1}\controllerArg{}{t}\\ 
&\frac{d }{dt} \adjointArg{}{t}  + \basisspace^T \big[  \frac{\partial \velocity}{\partial \statey} \big]^T \stweightingMatArg{}  \basisspace [\massArgnt{}]^{-1}  \adjointArg{}{t} = 
 - \basisspace^T \big[  \frac{\partial \velocity}{\partial \statey} \big]^T \stweightingMatArg{} 
\bigg(  \basisspace  [\massArgnt{}]^{-1}\basisspace^T \stweightingMatArg{} - \mathbf{I} \bigg)
 \bigg( \basisspace \genstateDotArg{}{t} - \velocity(\basisspace \genstateArg{}{t} + \stateIntercept) \bigg) \\ 
& \adjointArg{}{t}= -\controllerArg{}{t}.
\end{align*}
This can be written equivalently as
\begin{align*}
&\frac{d}{dt}\genstateArg{}{t} =  [\massArgnt{}]^{-1}\basisspace^T \stweightingMatArg{}\velocity(\basisspace \genstateArg{}{t} + \stateIntercept) + [\massArgnt{}]^{-1}\controllerArg{}{t} \\ 
&\frac{d}{dt} \controllerArg{}{t}  + \basisspace^T \big[  \frac{\partial \velocity}{\partial \statey} \big]^T \stweightingMatArg{}  \basisspace [\massArgnt{}]^{-1}  \controllerArg{}{t} = 
 - \basisspace^T \big[  \frac{\partial \velocity}{\partial \statey} \big]^T \stweightingMatArg{} 
\bigg( \mathbf{I} -   \basisspace  [\massArgnt{}]^{-1}\basisspace^T \bigg)
\stweightingMatArg{} \bigg( \basisspace \genstateDotArg{}{t} - \velocity(\basisspace \genstateArg{}{t} + \stateIntercept) \bigg). 
\end{align*}

\end{appendices}

\bibliographystyle{siam}
\bibliography{refs}
\end{document}

%% file: packages.tex
\usepackage{fullpage}
\usepackage[ruled]{algorithm2e}
\usepackage{algpseudocode}
\usepackage[utf8]{inputenc}
\usepackage[english]{babel}
\usepackage{amsthm}
\newtheorem{theorem}{Theorem}[section]
\newtheorem{corollary}{Corollary}[theorem]
\newtheorem{remark}{Remark}[section] 
\usepackage{color}
\usepackage{amsmath}
\usepackage{amsfonts}
\usepackage{mathtools}
\usepackage[titletoc,title]{appendix}
\usepackage{float}
\usepackage{graphicx} 
\usepackage{epstopdf} 
\usepackage{caption}
\usepackage{subcaption}
\usepackage{pdfpages}
\usepackage{comment}

\usepackage{hyperref}
\hypersetup{
    colorlinks=true,
    linkcolor=blue,
    filecolor=magenta,      
    urlcolor=cyan,
}
\usepackage{tikz}
\usetikzlibrary{positioning, fit, arrows.meta, shapes}

\usetikzlibrary{
  arrows.meta, 
  fit, 
  positioning, 
}
\tikzset{
  neuron/.style={ 
    circle,draw,thick, 
    inner sep=0pt, 
    minimum size=3.5em, 
    node distance=1ex and 2em, 
  },
  group/.style={ 
    rectangle,draw,thick, 
    inner sep=0pt, 
  },
  io/.style={ 
    neuron, 
    fill=gray!15, 
  },
  conn/.style={ 
    -{Straight Barb[angle=60:2pt 3]}, 
    thick, 
  },
}
\numberwithin{equation}{section}

%% file: commands.tex
\newcommand{\unroll}{\boldsymbol v}

\definecolor{Blue}{rgb}{0,0,1}
\definecolor{Red}{rgb}{1,0,0}

\newcommand{\param}{\boldsymbol\mu}

\newcommand{\state}{\boldsymbol x}
\newcommand{\error}{\boldsymbol e}

\newcommand{\genstate}{\hat{ \boldsymbol x}}

\newcommand{\resid}{\boldsymbol r}

\newcommand{\decoder}{\boldsymbol g}

\newcommand{\norm}[1]{ \|  #1 \|_2}

\newcommand{\velocity}{\boldsymbol f}

\newcommand{\basisst}{\boldsymbol S}


%% file: commands_cwst.tex
\newcommand{\approachKwd}{approach}

\newcommand{\methodAcronym}{WLS}
\newcommand{\methodAcronymROM}{WLS ROM}
\newcommand{\methodAcronymROMs}{WLS ROMs}

\newcommand{\RRplus}{\RR{}_{+}}

\newcommand{\methodNameLower}{windowed least-squares}

\newcommand{\EP}[1]{{\color{red}EP: #1}}
\newcommand{\matshapea}{
\tikz[baseline=0.0ex]\draw (0,0) rectangle (15pt,-25pt);}

\newcommand{\stdim}{K_{\text{ST}}}
\newcommand{\stdimArg}[1]{K_{\text{ST}}^{#1}}

\newcommand{\quadWeightsLMSScalarArg}[2]{\gamma^{#1,#2}}

\newcommand{\timeDummy}{\tau}
\newcommand{\fbsmGrowth}{\psi_1}
\newcommand{\fbsmDecay}{\psi_2}
\newcommand{\projector}{\mathbb{P}}
\newcommand{\projectorArg}[1]{\projector^{#1}}

\newcommand{\veloargsromy}{\basisspace \genstateyDiscreteArgnt{} + \stateIntercept,\timeDummy}

\newcommand{\statew}{\boldsymbol w}

\newcommand{\stweightingMatCollocArgt}[2]{\stweightingMat}
\newcommand{\stweightingMatCollocArg}[2]{\overline{\stweightingMat}}
\newcommand{\stweightingMat}{ \mathbf{A}}

\newcommand{\stweightingMatOne}{ \mathbf{W}}
\newcommand{\stweightingMatOneArg}[1]{\stweightingMatOne^{#1}}
\newcommand{\stweightingMatOneTArg}[1]{[\stweightingMatOne^{#1}]^T}

\newcommand{\stweightingMatArg}[1]{\stweightingMat^{#1}}
\newcommand{\stweightingMatArgt}[2]{\stweightingMat^{#1}}
\newcommand{\DeltaSlabArg}[1]{\Delta T^{#1}}

\newcommand{\hamiltonian}{\mathcal{H}}
\newcommand{\hamiltonianArg}[1]{\hamiltonian^{#1}}
\newcommand{\adjoint}{\hat{\boldsymbol \lambda}}

\newcommand{\adjointDot}{\dot{\adjoint}}
\newcommand{\adjointDotArg}[2]{\adjointDot^{#1}(#2)}
\newcommand{\adjointDotArgnt}[1]{\adjointDot^{#1}}

\newcommand{\adjointArg}[2]{\adjoint^{#1}(#2)}
\newcommand{\adjointArgnt}[1]{\adjoint^{#1}}

\newcommand{\adjointOCDot}{\dot{\hat{\boldsymbol \nu}}}
\newcommand{\adjointOCDotArg}[2]{\adjointOCDot^{#1}(#2)}

\newcommand{\adjointOC}{\hat{\boldsymbol \nu}}
\newcommand{\adjointOCArg}[2]{\adjointOC^{#1}(#2)}
\newcommand{\adjointOCArgnt}[1]{\adjointOC^{#1}}

\newcommand{\adjointDum}{\hat{\boldsymbol \mu}}

\newcommand{\adjointDumArgnt}[1]{\adjointDum^{#1}}
\newcommand{\adjointDiscreteDum}{\hat{\pmb \mu}}

\newcommand{\adjointDiscreteDumArgnt}[1]{\adjointDiscreteDum^{#1}}

\newcommand{\Range}[1]{\mathrm{Ran}(#1)}

\newcommand{\lipshitzf}{\Gamma}
\newcommand{\lipshitz}{\kappa}
\newcommand{\lipshitzi}{\alpha}
\newcommand{\lipshitziArg}[1]{\alpha^{#1}}

\newcommand{\stateFOMDot}{\dot{\stateFOM}}
\newcommand{\stateFOMDotArg}[1]{\dot{\stateFOM}(#1)}

\newcommand{\stateFOMSolArg}[1]{\stateFOM^{#1}}
\newcommand{\stateFOMSolArgt}[2]{\stateFOM^{#1}(#2)}
\newcommand{\stateFOMProjSol}{\tilde{\state}_{\ell^2}}
\newcommand{\stateFOMProjSolArg}[1]{\stateFOMProjSol^{#1}}
\newcommand{\stateFOMProjSolArgt}[2]{\stateFOMProjSol^{#1}(#2)} 

\newcommand{\intSlabArg}[1]{\int_{\timeStartArg{#1}}^{\timeEndArg{#1}} }
\newcommand{\genstateROMSol}{\genstate}
\newcommand{\genstateROMSolArg}[1]{\genstateROMSol^{#1}}
\newcommand{\genstateROMSolArgt}[2]{\genstateROMSol^{#1}(#2)}
\newcommand{\genstateROMStarSol}{\genstate_{*}}
\newcommand{\genstateROMStarSolArg}[1]{\genstateROMStarSol^{#1}}
\newcommand{\genstateROMStarSolArgt}[2]{\genstateROMStarSol^{#1}(#2)}

\newcommand{\approxstateIC}{\basisspace \basisspace^T (\stateFOMIC - \stateIntercept) + \stateIntercept}

\newcommand{\errorStar}{\hat{\error}_{*}}
\newcommand{\errorStarArg}[1]{\errorStar^{#1}}
\newcommand{\errorStarArgt}[2]{\errorStar^{#1}(#2)}
\newcommand{\errorStarDot}{\dot{\hat{\error}}_{*}}
\newcommand{\errorStarDotArg}[1]{\errorStarDot^{#1}}
\newcommand{\errorStarDotArgt}[2]{\errorStarDot^{#1}(#2)}

\newcommand{\stateROMSol}{\tilde{\state}}
\newcommand{\stateROMSolArg}[1]{\stateROMSol^{#1}}
\newcommand{\stateROMSolArgt}[2]{\stateROMSol^{#1}(#2)}
\newcommand{\stateROMStarSol}{\tilde{\state}_{*}}
\newcommand{\stateROMStarSolArg}[1]{\stateROMStarSol^{#1}}
\newcommand{\stateROMStarSolArgt}[2]{\stateROMStarSol^{#1}(#2)}

\newcommand{\adjointROMStarSol}{\adjoint_*}
\newcommand{\adjointROMStarSolArg}[1]{\adjointROMStarSol^{#1}}

\newcommand{\adjointROMSol}{\adjoint}
\newcommand{\adjointROMSolArg}[1]{\adjoint^{#1}}
\newcommand{\adjointROMStarSolArgt}[2]{\adjointROMStarSol^{#1}(#2)}
\newcommand{\adjointROMSolArgt}[2]{\adjointROMSol^{#1}(#2)}

\newcommand{\errorArg}[1]{\error^{#1}}
\newcommand{\errorArgt}[2]{\error^{#1}(#2)}
\newcommand{\objectiveArg}[1]{\objective^{#1}}
\newcommand{\objectiveDiscrete}{J}
\newcommand{\objectiveArgLMS}[1]{\objectiveDiscrete^{#1}}

\newcommand{\objectiveDiscreteST}{J_{\text{ST}}}
\newcommand{\objectiveDiscreteSTArg}[1]{\objectiveDiscreteST^{#1}}

\newcommand{\objective}{\mathcal{J}}
\newcommand{\controller}{\hat{\boldsymbol u}}
\newcommand{\controllerArg}[2]{\controller^{#1}(#2)}
\newcommand{\controllerArgnt}[1]{{\controller^{#1}}}

\newcommand{\objectiveControl}{\mathcal{L}}
\newcommand{\objectiveControlArg}[1]{\objectiveControl^{#1}}
\newcommand{\controllerDum}{\hat{\boldsymbol v}}
\newcommand{\controllerDumArg}[2]{\controllerDum^{#1}(#2)}
\newcommand{\controllerDumArgnt}[1]{\controllerDum^{#1}}
\newcommand{\controllerDiscreteDum}{\hat{\mathbf{v}}}
\newcommand{\controllerDiscreteDumArg}[2]{\controllerDiscreteDum^{#1}(#2)}
\newcommand{\controllerDiscreteDumArgnt}[1]{\controllerDiscreteDum^{#1}}

\newcommand{\spatialIC}{\basisspaceTArg{n}(\basisspaceArg{n-1} \genstate^{n-1}(\timeEndArg{n-1} )  + \stateInterceptArg{n-1} - \stateInterceptArg{n} )}

\newcommand{\defeq}{\vcentcolon=}
\newcommand{\mass}{\mathbf{M}}
\newcommand{\massArg}[1]{\mass^{#1}}

\newcommand{\massArgnt}[1]{\mass}
\newcommand{\massArgt}[2]{\mass}
\newcommand{\genstateyDot}{\hat{\boldsymbol v}}
\newcommand{\genstateyDotArg}[2]{\genstateyDot^{#1}(#2)}

\newcommand{\genstateyDiscreteDot}{\hat{\mathbf{v}}}
\newcommand{\genstateyDiscreteDotArgnt}[1]{\genstateyDiscreteDot^{#1}}

\newcommand{\genstateDot}{\dot{\genstate}}

\newcommand{\stateyDot}{{\boldsymbol v}}

\newcommand{\statezDot}{\dot{\statez}}
\newcommand{\statez}{\boldsymbol z}
\newcommand{\statezDotBar}{\dot{\statezBar}}
\newcommand{\statezBar}{\overline{\statez}}

\newcommand{\variation}{\boldsymbol \eta}
\newcommand{\variationArgn}[1]{\boldsymbol \eta^{#1}}
\newcommand{\variationArgntt}[2]{\boldsymbol \eta^{#1}(#2)}

\newcommand{\variationDot}{\dot{\boldsymbol \eta}}

\newcommand{\variationArg}[1]{\boldsymbol \eta(#1)}

\newcommand{\genstateDotArg}[2]{\dot{\genstate}^{#1}(#2)}
\newcommand{\genstateDotArgnt}[1]{\dot{\genstate}^{#1}}

\newcommand{\genstateDiscrete}{\hat{\mathbf{x}}}
\newcommand{\genstateDiscreteArg}[1]{\genstateDiscrete^{#1}}
\newcommand{\genstateGuessDiscrete}{\hat{\mathbf{x}}}
\newcommand{\genstateGuessDiscreteArg}[2]{\genstateGuessDiscrete^{#1}_{#2}}

\newcommand{\genstateArg}[2]{\genstate^{#1}(#2)}
\newcommand{\genstateArgnt}[1]{\genstate^{#1}}

\newcommand{\genstateGalerkinDot}{\dot{\genstate}_{\text{G}}}
\newcommand{\genstateGalerkinDotArg}[2]{\genstateGalerkinDot^{#1}(#2)}

\newcommand{\genstateGalerkin}{\genstate_{\text{G}}}
\newcommand{\genstateGalerkinArg}[2]{\genstateGalerkin^{#1}(#2)}

\newcommand{\genstateyArg}[2]{\genstatey^{#1}(#2)}
\newcommand{\genstateyArgnt}[1]{\genstatey^{#1}}

\newcommand{\genstateyDiscreteArg}[1]{\genstateyDiscrete^{#1}}
\newcommand{\genstateyDiscreteArgnt}[1]{\genstateyDiscrete^{#1}}

\newcommand{\stateyDiscreteArgnt}[1]{\stateyDiscrete^{#1}}
\newcommand{\stateyDiscreteArg}[1]{\stateyDiscrete^{#1}}
\newcommand{\statewDiscrete}{\mathbf{w}}

\newcommand{\ncollocST}{N_{\text{ST}}}
\newcommand{\ncollocSTArg}[1]{\ncollocST^{#1}}
\newcommand{\collocPointST}{\chi_{\text{ST}}}
\newcommand{\collocPointSTArg}[2]{\collocPointST^{#1,#2}}

\newcommand{\minintegrandArg}[1]{\minintegrand^{#1}}

\newcommand{\dIdV}{\minintegrand_{ \genstateyDiscreteDot}}
\newcommand{\dIdVDot}{\dot{\minintegrand}_{ \genstateyDiscreteDot}}
\newcommand{\dIdY}{\minintegrand_{ \genstateyDiscrete}}

\newcommand{\dIdVArg}[1]{\dIdV^{#1}}
\newcommand{\dIdYArg}[1]{\dIdY^{#1}}
\newcommand{\dIdVDotArg}[1]{\dIdVDot^{#1}}

\newcommand{\nsamples}{n_s}

\newcommand{\timeWindow}{\tau}
\newcommand{\timeWindowArg}[2]{\tau^{#1,#2}}

\newcommand{\bz}{\boldsymbol 0}
\newcommand{\nsteps}{N_{\tau}}
\newcommand{\nstepsArg}[1]{\nsteps^{#1}}

\newcommand{\lspgWeightingArg}[1]{\stweightingMatOne}
\newcommand{\lspgWeightingST}{\stweightingMatOne_{\text{ST}}}
\newcommand{\lspgWeightingSTArg}[1]{\stweightingMatOne_{\text{ST}}}

\newcommand{\approxstate}{\tilde{\boldsymbol x}} 
\newcommand{\approxstateDiscrete}{\tilde{\mathbf{x}}} 
\newcommand{\approxstateDiscreteArg}[1]{\approxstateDiscrete^{#1}} 

\newcommand{\approxstateLSPG}{\approxstateDiscrete_{\text{L}}} 

\newcommand{\approxstateArg}[2]{\approxstate^{#1}(#2)} 
\newcommand{\approxstateArgnt}[1]{\approxstate^{#1}}

\newcommand{\stateMat}{\mathbf {x}}
\newcommand{\stateMaty}{\mathbf{y}}

\newcommand{\genstatecollocMatSlab}{\hat{\overline{\overline{\stateMat}}}}
\newcommand{\genstatecollocMatSlabArg}[1]{\genstatecollocMatSlab^{#1}}
\newcommand{\genstatecollocMatySlab}{\hat{\overline{\overline{\stateMaty}}}}
\newcommand{\genstatecollocMatySlabArg}[1]{\genstatecollocMatySlab^{#1}}

\newcommand{\timeSpace}{\mathcal{T}}
\newcommand{\onesFunction}{\mathcal{O}}
\newcommand{\onesFunctionArg}[1]{\mathcal{O}^{#1}}

\newcommand{\timeSpaceArg}[1]{\timeSpace^{#1}}

\newcommand{\stateFOM}{{\state}}
\newcommand{\stateFOMIC}{{\stateFOMDiscrete_0}}

\newcommand{\stateFOMArg}[2]{\stateFOM^{#1}(#2)}
\newcommand{\stateFOMArgnt}[1]{\stateFOM^{#1}}

\newcommand{\stateFOMDiscrete}{{\mathbf{x}}}
\newcommand{\stateFOMDiscreteArg}[1]{\stateFOMDiscrete^{#1}}

\newcommand{\stateFOMDiscreteLMSArg}[2]{\stateFOMDiscrete^{\indexMapper(#1,#2)}}

\newcommand{\indexMapper}{\theta}
\newcommand{\indexMappern}{\theta^*}

\newcommand{\stateyArg}[2]{\statey^{#1}(#2)}

\newcommand{\stgenstate}{\hat{\mathbf{\overline{x}}}}
\newcommand{\stgenstateArg}[1]{\stgenstate^{#1}}
\newcommand{\stgenstateStar}{\hat{\mathbf{\overline{x}}}_*}
\newcommand{\stgenstateStarArg}[1]{\stgenstateStar^{#1}}

\newcommand{\stgenstatey}{\hat{\mathbf{\overline{y}}}}
\newcommand{\stgenstateyArg}[1]{\stgenstatey^{#1}}

\newcommand{\stateyDiscrete}{\mathbf{y}}
\newcommand{\stateyDotDiscrete}{\mathbf{{v}}}

\newcommand{\statey}{\boldsymbol y}
\newcommand{\stateIntercept}{\mathbf{x}_\text{ref}}
\newcommand{\stateInterceptArg}[1]{\stateIntercept^{#1}}
\newcommand{\stateInterceptST}{\mathbf{x}_\text{st}}
\newcommand{\stateInterceptSTArg}[1]{\stateInterceptST^{#1}}

\newcommand{\genstatey}{\hat{\boldsymbol y}}
\newcommand{\genstateyDiscrete}{\hat{\mathbf{y}}}

\newcommand{\residST}{\overline{\boldsymbol r}}
\newcommand{\genresidST}{\hat{\overline{\boldsymbol r}}}

\newcommand{\residLMSSlabArg}[1]{\overline{\overline{\boldsymbol r}}^{#1}_{\text{}}}
\newcommand{\spatialAcronym}{S-reduction}
\newcommand{\spaceTimeAcronym}{ST-reduction}

\newcommand{\residLMSArg}[1]{\resid_{}^{#1}}
\newcommand{\residLMSWArg}[2]{\resid_{}^{#1,#2}}

\newcommand{\lmsWidthArg}[2]{k^{#1,#2}}
\newcommand{\ntimeSteps}{N_t}
\newcommand{\nslabs}{N_w}
\newcommand{\RR}[1]{\mathbb{R}^{#1}}
\newcommand{\RRStar}[2]{\mathbb{V}_{#1}(\RR{#2})}
\newcommand{\veloargsrom}{\basisspaceArg{} \genstateArg{}{t} + \stateInterceptArg{},t}
\newcommand{\veloargsromn}{\basisspaceArg{n} \genstateArg{n}{t} + \stateInterceptArg{n},t}

\newcommand{\veloargsromArg}[1]{\basisspace \genstateArgnt{n}_{#1}(t) + \stateInterceptArg{n},t}

\newcommand{\timeStartArg}[1]{t_s^{#1}}
\newcommand{\timeEndArg}[1]{t_f^{#1}}
\newcommand{\trialspace}{\mathcal{V } } 
\newcommand{\trialspaceArg}[1]{\trialspace^{#1}} 

\newcommand{\stspace}{\mathcal{ ST } } 
\newcommand{\stspaceS}{\mathcal{ ST }_{\text{S}} } 
\newcommand{\stspaceST}{\mathcal{ ST }_{\text{ST}} }

\newcommand{\stspaceBoundStartArg}[1]{\stspaceBoundStart^{#1}} 
\newcommand{\stspaceBoundStart}{\mathcal{B}} 

\newcommand{\stspaceArg}[1]{\stspace^{#1}} 
\newcommand{\stspaceSTArg}[1]{\stspaceST^{#1}} 
\newcommand{\stspaceSArg}[1]{\stspaceS^{#1}} 
\newcommand{\stbasis}{\boldsymbol \Pi}

\newcommand{\stbasisArg}[2]{\stbasis^{#1}(#2)}
\newcommand{\stbasisArgnt}[1]{\stbasis^{#1}}
\newcommand{\stbasisDot}{\dot{\boldsymbol \Pi}}
\newcommand{\stbasisDotArg}[2]{\stbasisDot^{#1}(#2)}

\newcommand{\genstateIC}{ \basisspaceTArg{}( \stateFOMIC - \stateInterceptArg{}) }
\newcommand{\genstateICOne}{ [\basisspaceArg{1}]^T( \stateFOMIC -
\stateInterceptArg{1}) }

\newcommand{\elltwo}{\ell^2}
\newcommand{\romdim}{K}
\newcommand{\romdimArg}[1]{K^{#1}}
\newcommand{\fomdim}{N}

\newcommand{\basisspace}{\basismat}
\newcommand{\basisspaceArg}[1]{\basismat^{#1}}
\newcommand{\basisspaceTArg}[1]{[\basismat^{#1}]^T}

\newcommand{\basismat}{\mathbf {V}}
\newcommand{\basisvec}{\mathbf{v}}
\newcommand{\basismatArg}[1]{\basismat^{#1}}
\newcommand{\basisvecArg}[1]{\basisvec^{#1}}

\newcommand{\minintegrand}{\mathcal{I}}

\newcommand{\integrand}{\mathcal{I}}

%% file: section3/tclsrm_formulation.tex
\input{section3/general_form}
\input{section3/spatial_projection}

\input{section3/space_time}

\subsection{\methodAcronym\ summary}
This section outlined the \methodAcronym\ \approachKwd\ for model reduction. In
summary, \methodAcronym\ sequentially minimizes the time continuous FOM ODE
residual within the range of a space--time trial subspace over time
windows; i.e.,  
\methodAcronym\ sequentially computes approximate solutions
$\approxstateArgnt{n}\in\stspace^n$, $n=1,\ldots,\nslabs$ that comprise
solutions to the optimization problem \eqref{eq:tclsrm}.
For \spatialAcronym\ trial subspaces, the stationary conditions for the
residual minimization problem can be derived via the Euler--Lagrange equations
and yield the system \eqref{eq:lspg_continuous}--\eqref{eq:lspg_bcs} over the $n$th window for $t \in
[\timeStartArg{n},\timeEndArg{n}]$.
\methodAcronym\ with \spatialAcronym\ trial subspaces can be alternatively interpreted as a Lagrange problem from optimal control: the Galerkin method is forced by a 
controller that enforces the residual minimization property. 
Section~\ref{sec:wls_spacetime} additionally considered \spaceTimeAcronym\
trial subspaces. For \spaceTimeAcronym\ trial subspaces, in where the
generalized coordinates are no longer functions, the stationary conditions
correspond to the system of algebraic equations \eqref{eq:st_stationary}.

%% file: section3/general_form.tex
\section{Windowed least-squares approach}\label{sec:tclspg} 
This section
outlines the proposed \methodNameLower\ 
(\methodAcronym) \approachKwd. In contrast to (1) Galerkin projection, 
which minimizes the (time-continuous) FOM ODE residual at a time instance, 
(2) LSPG projection, 
which minimizes
the (time-discrete) FOM O$\Delta$E residual 
over a time step, and (3) ST-LSPG projection, 
which minimizes
the  FOM O$\Delta$E residual 
over the entire time interval, the proposed \methodAcronym\ \approachKwd\ sequentially minimizes the 
FOM ODE residual over arbitrarily defined
\textit{time windows}. The formulation is compatible with both \spatialAcronym\ and \spaceTimeAcronym\ trial subspaces. 
In this section, we start by outlining the \methodAcronym\ formulation  
for a general space--time trial subspace. We then examine the \spatialAcronym\ 
trial subspace in this context, followed by the \spaceTimeAcronym\ trial subspace. In each case, we derive the stationary conditions 
associated with the residual-minimization problems.
\subsection{Windowed least-squares for general space--time trial subspaces} 
We begin by introducing a (potentially nonuniform) partition of the time domain $[0,T]$
into $\nslabs$ non-overlapping windows $[\timeStartArg{n} ,
\timeEndArg{n}]\subseteq[0,T]$ of length $\DeltaSlabArg{n}\defeq\timeEndArg{n} -
\timeStartArg{n}$, $n=1,\ldots,\nslabs$ such that 
$\timeStartArg{1} = 0$, $\timeEndArg{\nslabs} = T$, and
$\timeStartArg{n+1} = \timeEndArg{n}$,
$n=1,\ldots,\nslabs-1$; Fig.~\ref{fig:slab_fig} depicts such a partitioning.
\begin{figure} 
\begin{centering} 
\includegraphics[trim={0.0cm 5cm 0cm 3cm},clip,width=1.0\textwidth]{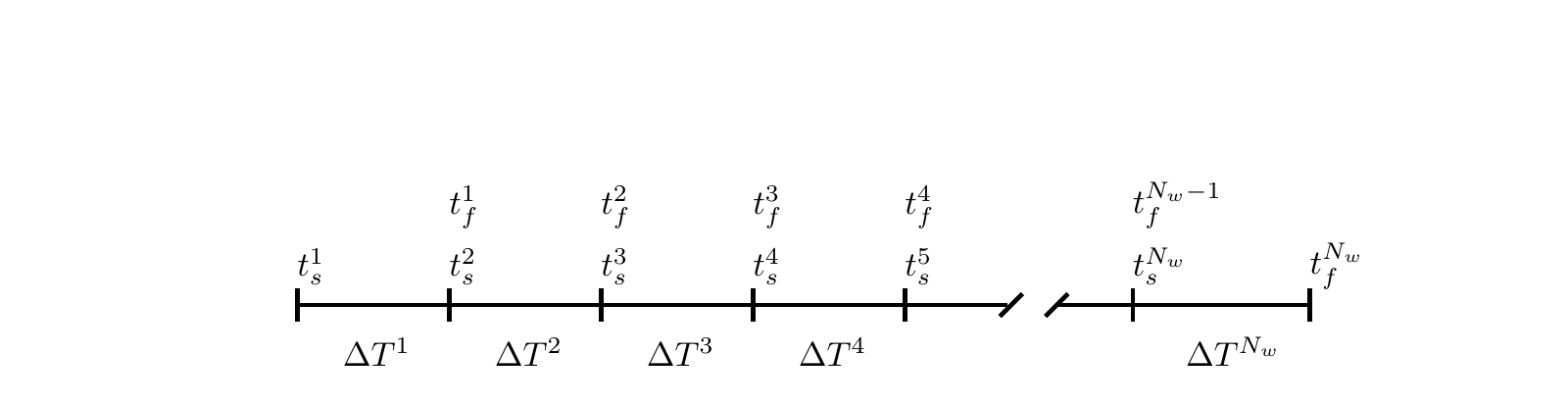} 
\caption{Partitioning of the time
	domain into time windows.} 
\label{fig:slab_fig} 
\end{centering} 
\end{figure}
Over the $n$th time window, we approximate the FOM ODE solution as 
$\approxstate^n(t)\approx \stateFOM(t)$,
$t\in[\timeStartArg{n},\timeEndArg{n}]$, which is enforced to reside in 
the $n$th space--time trial subspace 
$\stspaceArg{n}$
such that
\begin{equation}
\approxstate^n \in \stspace^n \subseteq \RR{N} \otimes \timeSpaceArg{n} , \qquad  n = 1,\ldots,\nslabs,
\end{equation}
where  $\timeSpaceArg{n}$ denotes the set of (sufficiently smooth) real-valued functions acting on
$[\timeStartArg{n},\timeEndArg{n}]$ (i.e., $\timeSpaceArg{n} =
\{f\,|\,f:[\timeStartArg{n},\timeEndArg{n}]\rightarrow\RR{}\}$).
For notational purposes, we additionally define the spatial trial subspaces at the start of each window as
\begin{equation}
	\stspaceBoundStartArg{n}\defeq\mathrm{span}(\{\statey(\timeStartArg{n})\,|\, \statey\in
	\stspaceArg{n}\})\subseteq\RR{\fomdim}, \qquad n=1,\ldots,\nslabs, 
\end{equation}
such that $\approxstateArg{n}{\timeStartArg{n}} \in \stspaceBoundStartArg{n}$.
To outline \methodAcronym, we now define the objective functional over the
$n$th window as
\begin{equation}\label{eq:obj}
\begin{split} \mathcal{J}^n &\vcentcolon \statey \mapsto
\frac{1}{2} \int_{\timeStartArg{n}}^{\timeEndArg{n}} \big[ \dot{\statey}(t)
- \velocity(\stateyArg{}{t},t) \big]^T \stweightingMatArgt{n}{t} \big[
\dot{\statey}(t) - \velocity(\stateyArg{}{t},t) \big] d t, \\
&\vcentcolon \RR{\fomdim}\otimes \timeSpaceArg{n}  \rightarrow
\RRplus, 
\end{split}
\end{equation}
where $\stweightingMatArg{n} \equiv \stweightingMatOneTArg{n}
\stweightingMatOneArg{n} \in \RR{\fomdim \times \fomdim}$ denotes a
symmetric positive semi-definite matrix that can enable hyper-reduction, for
example. 

The \methodAcronym\ \approachKwd\ sequentially computes approximate solutions
$\approxstateArgnt{n}\in\stspace^n$, $n=1,\ldots,\nslabs$, where 
$\approxstateArgnt{n}$ is the solution to the
minimization problem
\begin{equation}\label{eq:tclsrm}
\begin{split}
      &\underset{\statey \in \stspace^n}{\text{minimize}} \; \mathcal{J}^n(\statey), \\
			&\text{subject to } \;  \statey(\timeStartArg{n}) =
\begin{cases}\mathbb{P}^n (\approxstateArgnt{n-1}(\timeEndArg{n-1})) & n = 2,\ldots,\nslabs, \\
\mathbb{P}^n( \stateFOMIC)& n=1, \end{cases} 
\end{split}
\end{equation}
where $\projectorArg{n} : \RR{\fomdim} \rightarrow \stspaceBoundStartArg{n}$ 
is a (spatial) projection operator onto the start of the $n$th trial space
(e.g., $\elltwo$-orthogonal projection operator).

We now define the trial subspaces $\stspaceArg{n}$ considered in this work. In
	particular, we introduce \spatialAcronym\ trial subspaces and
	\spaceTimeAcronym\ trial subspaces tailored for this context. 
We leave further investigation into other approaches, such as nonlinear
	trial manifolds~\cite{leeCarlberg}, as a subject for future work.

%% file: section3/spatial_projection.tex
\subsection{\spatialAcronym\ trial subspaces}
In this context, the \spatialAcronym\ trial subspace over the $n$th time window approximates
the FOM ODE solution trajectory $\stateFOM(t)$,
$t\in[\timeStartArg{n},\timeEndArg{n}]$ 
with $\approxstateArgnt{n} \in \stspaceSArg{n}$, where
\begin{equation}\label{eq:sttrialspace}
 \stspaceSArg{n} \defeq 
	\trialspaceArg{n} \otimes \timeSpaceArg{n} +
	\stateInterceptArg{n}\otimes\onesFunction^n\subseteq\RR{N} \otimes \timeSpaceArg{n}.
\end{equation}
Here, the spatial trial subspaces $\trialspaceArg{n}\subseteq\RR{\fomdim}$,
$n=1,\ldots,\nslabs$ satisfy 
$\trialspaceArg{n}\defeq\Range{\basismatArg{n}}$ with 
$\basismatArg{n}\equiv[\basisvecArg{n}_1\ \cdots\
\basisvecArg{n}_{\romdimArg{n}}]\in
\RRStar{\romdimArg{n}}{\fomdim}
$, the reference states $\stateInterceptArg{n} \in \RR{\fomdim}$, $n=1,\ldots,\nslabs$ provide the affine transformation, 
and  
$\onesFunction^n\in\timeSpaceArg{n}$ is defined as
$\onesFunction^n:\timeDummy\mapsto 1.$
Thus, at any time instance $t\in[\timeStartArg{n},\timeEndArg{n}]$, the
\spatialAcronym\ trial subspace 
approximates the FOM ODE solution as
\begin{equation}\label{eq:affine_trialspace_tclsrm}
	\stateFOM(t)\approx \approxstateArgnt{n}(t) = \basismatArg{n}
	\genstateArg{n}{t} + \stateInterceptArg{n},
\end{equation}
where $\genstateArgnt{n} \in \RR{\romdimArg{n}} \otimes \timeSpaceArg{n}$ with
$\genstateArgnt{n}: \timeDummy\mapsto \genstateArgnt{n}(\timeDummy)
$ denotes the generalized coordinates over the $n$th time window. 

Setting $\stspace^n \leftarrow \stspaceSArg{n}$ in the
\methodAcronym\ minimization problem~\eqref{eq:tclsrm} and setting 
$\mathbb{P}^n$ to the 
$\elltwo$-orthogonal projection operator 
implies that 
\methodAcronym\ with \spatialAcronym\ space--time trial subspaces
sequentially computes solutions
$\genstateArgnt{n}$, $n = 1,\ldots,\nslabs$ that satisfy
\begin{equation}\label{eq:obj_gen_slab}
\begin{split}
	& \underset{\genstateyArgnt{} \in \RR{\romdimArg{n}} \otimes \timeSpaceArg{n}
			}{\text{minimize}}\; \mathcal{J}^n(\basismatArg{n} \genstateyArgnt{} +
			\stateInterceptArg{n} \otimes \onesFunction^n), \\ 
      & \text{subject to }\; \genstateyArg{}{\timeStartArg{n}} =
	\begin{cases}
\spatialIC & n = 2,\ldots,\nslabs, \\
\genstateICOne & n=1. \end{cases} 
\end{split}
\end{equation}

\subsubsection{Stationary conditions and the Euler--Lagrange equations}
We derive stationary conditions for optimization problem~\eqref{eq:obj_gen_slab} via 
	the Euler--Lagrange equations from the
calculus of variations. To begin, we define the
integrand appearing in the objective function $\mathcal{J}^n$ defined in Eq.~\eqref{eq:obj} in terms of the generalized
coordinates induced by the \spatialAcronym\ subspace as 
\begin{equation}\label{eq:integrand}
\begin{split}
 \minintegrandArg{n} & \vcentcolon
(\genstateyDiscreteArgnt{}, \genstateyDiscreteDotArgnt{},\timeDummy) \mapsto \frac{1}{2} \big[
\basisspaceArg{n} \genstateyDiscreteDotArgnt{} - \velocity(\basisspaceArg{n} \genstateyDiscreteArgnt{}
+ \stateInterceptArg{n},\timeDummy ) \big]^T \stweightingMatArgt{n}{t} \big[
\basisspaceArg{n} \genstateyDiscreteDotArgnt{}  - \velocity(\basisspaceArg{n} \genstateyDiscreteArgnt{} +
\stateInterceptArg{n},\timeDummy) \big], \\ & \vcentcolon \RR{\romdimArg{n}} \times \RR{\romdimArg{n}} \times [\timeStartArg{n},\timeEndArg{n}]
 \rightarrow \RRplus .  
\end{split}
\end{equation}
We also define the quantities\footnote{We use numerator layout for the
scalar-by-vector gradients.}
\begin{equation}
\begin{split}
\dIdVArg{n}  & \vcentcolon
(\genstatey , \genstateyDot, \timeDummy) \mapsto \bigg[ \frac{\partial \minintegrandArg{n}}{\partial \genstateyDiscreteDot} (\genstateyArg{}{\timeDummy}, \genstateyDotArg{}{\timeDummy},\timeDummy) \bigg]^T, \\ 
& \vcentcolon  \RR{\romdimArg{n}} \otimes \timeSpaceArg{n} \times \RR{\romdimArg{n}} \otimes \timeSpaceArg{n} \times [\timeStartArg{n},\timeEndArg{n}]
 \rightarrow \RR{\romdimArg{n}} ,
\end{split}
\end{equation}
\begin{equation}
\begin{split}
\dIdYArg{n}  & \vcentcolon
(\genstatey , \genstateyDot, \timeDummy) \mapsto \bigg[ \frac{\partial \minintegrandArg{n}}{\partial \genstateyDiscrete} (\genstateyArg{}{\timeDummy}, \genstateyDotArg{}{\timeDummy},\timeDummy) \bigg]^T, \\ 
& \vcentcolon  \RR{\romdimArg{n}} \otimes \timeSpaceArg{n} \times \RR{\romdimArg{n}} \otimes \timeSpaceArg{n} \times [\timeStartArg{n},\timeEndArg{n}]
 \rightarrow \RR{\romdimArg{n}} .
\end{split}
\end{equation}
Using this notation, the Euler--Lagrange equations (see
Appendix~\ref{appendix:eulerlagrange} for the derivation) over the $n$th
time window for $t \in [\timeStartArg{n},\timeEndArg{n}]$ are given by
\begin{equation}\label{eq:el1} 
\begin{split}
& \dIdYArg{n}(\genstateArgnt{n},\genstateDotArgnt{n},t) - \dIdVDotArg{n}(\genstateArgnt{n},\genstateDotArgnt{n}, t )  = \bz, \\ 
&\genstate^n(\timeStartArg{n})  = \begin{cases} 
\spatialIC &
n=2,\ldots,\nslabs, \\ 
\genstateICOne & n=1,
\end{cases}\\ 
&\dIdVArg{n}(\genstateArgnt{n},\genstateDotArgnt{n},\timeEndArg{n})  = \bz.
\end{split} 
\end{equation}
Appendix~\ref{appendix:vector_calc} provides the steps required to evaluate
the terms in system~\eqref{eq:el1}; the resulting system can be written as the
following coupled forward--backward system for $t \in
[\timeStartArg{n},\timeEndArg{n} ]$:
\begin{align}\label{eq:lspg_continuous} 
&\massArg{n}   \genstateDotArg{n}{t}  -  [\basisspaceArg{n}]^T
\stweightingMatArgt{n}{t} \velocity(\veloargsromn) =  \massArg{n} \adjointArg{n}{t} , \\
\begin{split}
 &\massArg{n}  \adjointDotArg{n}{t}  + [\basisspaceArg{n}]^T \bigg[\frac{\partial
\velocity}{\partial \stateyDiscrete} (\veloargsromn) \bigg]^T \stweightingMatArgt{n}{t} \basisspaceArg{n}
 \adjointArg{n}{t} = -[\basisspaceArg{n}]^T \big[
\frac{\partial \velocity}{\partial \stateyDiscrete}(\veloargsromn) \big]^T \\& \hspace{1.7 in} \stweightingMatArgt{n}{t} \big( \mathbf{I} -
\basisspaceArg{n} [\massArg{n}]^{-1} [\basisspaceArg{n}]^T \stweightingMatArgt{n}{t} \big)
 \big( \basisspaceArg{n} \genstateDotArg{n}{t} -
\velocity(\veloargsromn) \big), \label{eq:lspg_adjoint}
\end{split}
 \\ &
\genstateArg{n}{\timeStartArg{n}} = \begin{cases}
\spatialIC & n=2,\ldots,\nslabs, \label{eq:lspg_bcs1}\\
\genstateICOne & n=1, \end{cases}\\ &
\adjointArg{n}{\timeEndArg{n}} = \boldsymbol 0, \label{eq:lspg_bcs} 
\end{align}
where 
\begin{equation}\label{eq:costate_def}
\begin{split}
\adjointArgnt{n} &: \timeDummy \mapsto \genstateDotArg{n}{\timeDummy}  -  [\massArg{n}]^{-1} [\basisspaceArg{n}]^T \stweightingMatArgt{n}{\timeDummy}\velocity(\basisspaceArg{n} \genstateArg{n}{\timeDummy} + \stateInterceptArg{n} , \timeDummy ) ,\\
&: [\timeStartArg{n},\timeEndArg{n}] \rightarrow \RR{\romdimArg{n}} ,
\end{split}
\end{equation}
is the adjoint or ``costate'' variable with $\adjointDotArgnt{n} \equiv d \adjointArgnt{n}/ d\tau$, and $\massArg{n} \equiv \basisspaceTArg{n} \stweightingMatArg{n} \basisspaceArg{n}$ is a mass matrix. 
Eq.~\eqref{eq:lspg_continuous} is equivalent to a Galerkin reduced-order
model forced by the costate variable $\adjointArgnt{n}$.
Eq.~\eqref{eq:lspg_adjoint} is typically referred to as the adjoint
equation, which is linear in the costate and is forced by the residual.
We note that both ODEs~\eqref{eq:lspg_continuous}
and~\eqref{eq:lspg_adjoint} can be equipped with hyper-reduction, e.g., via
collocation, (discrete) empirical interpolation, Gappy
POD~\cite{everson_sirovich_gappy,eim,qdeim_drmac}. The state--costate
coupled system~\eqref{eq:lspg_continuous}--\eqref{eq:lspg_bcs} can be interpreted as an ``optimally controlled"
ROM, wherein the adjoint equation controls the forward model by enforcing the
minimum residual condition over the time window.

We note that in the case $\stweightingMatArg{n} =
\mathbf{I}$, the system simplifies to 
\begin{align*}\label{eq:lspg_continuous_simle} & \genstateDotArg{n}{t}  -
\basisspaceTArg{n}  \velocity(\veloargsromn) =  \adjointArg{n}{t} , \\
 &\adjointDotArg{n}{t}   + \basisspaceTArg{n} \bigg[\frac{\partial
\velocity}{\partial \stateyDiscrete}(\veloargsromn)\bigg]^T \basisspaceArg{n} \adjointArg{n}{t} = \\
&\hspace{2 in} \basisspaceTArg{n} \bigg[
\frac{\partial \velocity}{\partial \stateyDiscrete} (\veloargsromn) \bigg]^T \big( \mathbf{I} -   \basisspaceArg{n} \basisspaceTArg{n}
\big)    \velocity(\veloargsromn) , \\ & \genstateArg{n}{\timeStartArg{n}} =
\begin{cases} \basisspaceTArg{n}(\basisspaceArg{n-1}\genstateArg{n-1}{\timeEndArg{n-1}} - \stateInterceptArg{n})& n=2,\ldots,\nslabs, \\
\genstateICOne& n=1, \end{cases}\\
&\adjointArg{n}{\timeEndArg{n}} = \boldsymbol 0 .  \end{align*}

\subsubsection{Formulation as an optimal-control problem of Lagrange type}\label{sec:optimal_control} 
The stationary conditions for \methodAcronym\ with \spatialAcronym\ trial subspaces~\eqref{eq:lspg_continuous}--\eqref{eq:lspg_bcs} can be alternatively formulated as a Lagrange
problem from optimal control. To this end, recall the dynamics of the Galerkin ROM over $t \in [\timeStartArg{n},\timeEndArg{n}]$, 
$$ \basisspaceTArg{} \stweightingMatArgt{}{t} \basisspaceArg{}
 \genstateGalerkinDotArg{}{t} - \basisspaceTArg{} \stweightingMatArgt{}{t}
\velocity(\basisspace \genstateGalerkin + \stateIntercept , t) = \bz.$$
We introduce now a controller $\controllerArgnt{n} \in \RR{\romdimArg{n}} \otimes \timeSpaceArg{n}$ 
and pose the problem of finding a controller that minimizes the residual
over the time window and forces the dynamics as
\begin{equation}\label{eq:controlled_rom}
 \basisspaceTArg{n}
\stweightingMatArgt{n}{t} \basisspaceArg{n} \genstateDotArg{n}{t}  - \basisspaceTArg{n}
\stweightingMatArgt{n}{t}\velocity(\veloargsromn) = \controllerArg{n}{t}. 
 \end{equation}
We now demonstrate how to compute this controller.
Before doing so, we note that~\eqref{eq:controlled_rom} displays commonalities with \textit{subgrid-scale}
methods~\cite{iliescu_pod_eddyviscosity,iliescu_vms_pod_ns,iliescu_ciazzo_residual_rom,parish_apg,wentland_apg,Wang:269133,San2018},
which add an additional term to the reduced-order model in order to account
for truncated states. 

We begin by defining a Lagrangian
\begin{equation}\label{eq:obj_controller}
\begin{split}
 \objectiveControlArg{n} &:  (\genstateyDiscreteArgnt{},\controllerDiscreteDumArgnt{},\timeDummy)
\mapsto \frac{1}{2} \bigg[ \basisspaceArg{n} \bigg(  [\massArg{n}]^{-1}\basisspaceTArg{n}
\stweightingMatArgt{n}{t}  \velocity(\basisspaceArg{n} \genstateyDiscreteArgnt{} +
\stateInterceptArg{n},\timeDummy) + [\massArg{n}]^{-1}\controllerDiscreteDumArgnt{} \bigg) -
\velocity(\basisspaceArg{n} \genstateyDiscreteArgnt{} + \stateInterceptArg{n},\timeDummy) \bigg]^T
\stweightingMatArgt{n}{t}  \\ & \hspace{1.5 in}\bigg[ \basisspaceArg{n} \bigg(
[\massArg{n}]^{-1}\basisspaceTArg{n} \stweightingMatArgt{n}{t}\velocity(\basisspaceArg{n}
\genstateyDiscreteArgnt{} + \stateInterceptArg{n},\timeDummy) + [\massArg{n}]^{-1}\controllerDiscreteDumArgnt{}
\bigg) - \velocity( \basisspaceArg{n}  \genstateyDiscreteArgnt{} + \stateInterceptArg{n},\timeDummy ) \bigg]
, \nonumber \\ & : \RR{\romdimArg{n}} \times \RR{\romdimArg{n}} \times
	[\timeStartArg{n},\timeEndArg{n}] \rightarrow \RRplus{},
\end{split}
\end{equation}
where we have used $ \genstateDotArg{n}{t} = [\massArg{n}]^{-1}\basisspaceTArg{n}
\stweightingMatArg{n} \velocity(\basisspaceArg{n} \genstateArg{n}{t} +
\stateInterceptArg{n} ,t) + [\massArg{n}]^{-1}\controllerArg{n}{t} 
$ from Eq.~\eqref{eq:controlled_rom}. Note that this Lagrangian measures the same residual as Eq.~\eqref{eq:integrand}.
The \methodAcronym\ approach with \spatialAcronym\ trial subspaces can be formulated as an optimal-control
method that sequentially computes the controllers $\controllerArgnt{n}
\in \RR{\romdimArg{n}} \otimes \timeSpaceArg{n}
$, $n=1,\ldots,\nslabs$ that satisfy
\begin{equation}\label{eq:tclspg_oc1a} 
\begin{split}
&\underset{\controllerDumArgnt{} \in \RR{\romdimArg{n}}\otimes \timeSpaceArg{n} }{\text{minimize } } 
\int_{\timeStartArg{n}}^{\timeEndArg{n}}
\objectiveControlArg{n}(\genstateArg{n}{t},\controllerDumArg{}{t},t)dt, 
 \\
&\text{subject to } \;  \left\{\begin{array}{l} 
 \basisspaceTArg{n} \stweightingMatArgt{n}{t}
\basisspaceArg{n}  \genstateDotArg{n}{t}  - \basisspaceTArg{n}
\stweightingMatArgt{n}{t} \velocity(\veloargsromn)
	=\controllerDumArg{}{t},\quad t\in[\timeStartArg{n},\timeEndArg{n}] \\
 \genstateArg{n}{\timeStartArg{n}} =
\begin{cases} \spatialIC  & n = 2,\ldots,\nslabs,
 \\ \genstateICOne & n=1. \end{cases} \end{array} \right.
\end{split}
\end{equation}
The solution to the system~\eqref{eq:tclspg_oc1a} is equivalent of that defined by~\eqref{eq:obj_gen_slab}. 
This can be demonstrated via the \textit{Pontryagin Maximum Principle} (PMP)~\cite{optimal_control_book}. To this end, we introduce the Lagrange multiplier (costate) 
	$\adjointOCArgnt{n} \in \RR{\romdimArg{n}} \otimes \timeSpaceArg{n}$ and define the Hamiltonian
\begin{equation}\label{eq:hamiltonian}
\begin{split}
\hamiltonianArg{n} \; &: \;  (\genstateyDiscreteArgnt{},\adjointDiscreteDumArgnt{},\controllerDiscreteDumArgnt{},\timeDummy) \mapsto 
 \adjointDiscreteDumArgnt{T} \bigg[  [\massArg{n}]^{-1}\basisspaceTArg{n} \stweightingMatArgt{n}{t}\velocity(\basisspaceArg{n} \genstateyDiscreteArgnt{} + \stateInterceptArg{n},\timeDummy) + [\massArg{n}]^{-1}\controllerDiscreteDumArgnt{} \bigg] +  \objectiveControlArg{n}(\genstateyDiscreteArgnt{},\controllerDiscreteDumArgnt{},\timeDummy), \\
&: \; \RR{\romdimArg{n}} \times \RR{\romdimArg{n}} \times \RR{\romdimArg{n}} \times [\timeStartArg{n},\timeEndArg{n}] \rightarrow \RR{}.
\end{split}
\end{equation} 
The Pontryagin Maximum Principle states that solutions of the optimization problem~\eqref{eq:tclspg_oc1a} must satisfy the 
following conditions over the $n$th window,
\begin{align*}
&\genstateDotArg{n}{t} = \frac{\partial \hamiltonianArg{n}}{\partial \adjointDiscreteDumArgnt{}{}}(\genstateArg{n}{t},\adjointOCArg{n}{t},\controllerArg{n}{t},t),\\
&\adjointOCDotArg{n}{t}= - \frac{\partial \hamiltonianArg{n}}{\partial \genstateyDiscreteArgnt{}{}}(\genstateArg{n}{t},\adjointOCArg{n}{t},\controllerArg{n}{t},t),\\
&\frac{\partial \hamiltonianArg{n}}{\partial \controllerDiscreteDumArgnt{}{}} (\genstateArg{n}{t},\adjointOCArg{n}{t},\controllerArg{n}{t},t) = \boldsymbol 0, \\
& \genstateArg{n}{\timeStartArg{n}} =
\begin{cases} \spatialIC & n = 2,\ldots,\nslabs,
 \\ \genstateICOne& n=1,\end{cases} \\
&\adjointOCArg{n}{\timeEndArg{n}}= \bz.
\end{align*}
Evaluation of the required gradients (Appendix~\ref{appendix:optimal_control}) yields the system to be solved over the $n$th window for $t \in [\timeStartArg{n},\timeEndArg{n}]$,
\begin{equation}\label{eq:sys_oc1}
\begin{split}
&\massArg{n}  \genstateDotArg{n}{t}  -  \basisspaceTArg{n} \stweightingMatArgt{n}{t} \velocity(\veloargsromn) =  \controllerArg{n}{t} , \\
 & \adjointOCDotArg{n}{t}  + \basisspaceTArg{n} \bigg[\frac{\partial \velocity}{\partial \stateyDiscrete} (\veloargsromn) \bigg]^T \stweightingMatArgt{n}{t} \basisspaceArg{n} [\massArg{n}]^{-1} \adjointOCArg{n}{t} =  \basisspaceTArg{n} \big[ \frac{\partial \velocity}{\partial \stateyDiscrete}(\veloargsromn) \big]^T \\
& \hspace{1.8in} \stweightingMatArgt{n}{t} \big( \mathbf{I} -   \basisspaceArg{n} [\massArg{n}]^{-1} \basisspaceTArg{n}  \stweightingMatArgt{n}{t} \big)  \big( \basisspaceArg{n}\genstateDotArg{n}{t}  -   \velocity(\veloargsromn) \big), \\
&\controllerArg{n}{t} = -\adjointOCArg{n}{t} ,\\
& \genstateArg{n}{\timeStartArg{n}} = 
\begin{cases}
\spatialIC& n=2,\ldots,\nslabs, \\
\genstateICOne & n=1,  
\end{cases} \\
& \adjointOCArg{n}{\timeEndArg{n}} = \boldsymbol 0. \\
\end{split}
\end{equation}
Setting 
$\controllerArgnt{n} = \massArg{n} \adjointArgnt{n}$ and
$\adjointOCArgnt{n} = -\massArg{n} \adjointArgnt{n}$ results in equivalence
between the system~\eqref{eq:sys_oc1} and the
system~\eqref{eq:lspg_continuous}--\eqref{eq:lspg_bcs}.
Thus, \methodAcronym\ with \spatialAcronym\ trial subspaces can be formulated
as an optimal control problem: \methodAcronym\ computes a controller that
modifies the Galerkin ROM to minimize the residual over the time window. 
The \methodAcronym\ method can additionally be
interpreted as a subgrid-scale modeling technique that constructs a
residual-minimizing subgrid-scale model.

\begin{remark}
The Euler--Lagrange equations comprise a Hamiltonian system. This imbues \methodAcronym\ with \spatialAcronym\ trial subspaces with certain properties; e.g., for autonomous systems the Hamiltonian~\eqref{eq:hamiltonian} is conserved. 
\end{remark}

%% file: section3/space_time.tex
\subsection{\spaceTimeAcronym\ trial spaces}\label{sec:wls_spacetime}
The \spaceTimeAcronym\ trial subspace over the $n$th time window approximates
the FOM ODE solution trajectory $\stateFOM\in\RR{N}\otimes\timeSpace$
with $\approxstateArgnt{n} \in \stspaceSTArg{n}$, where
\begin{equation}\label{eq:st_sttrialspace}
 \stspaceSTArg{n} \defeq
 \Range{\stbasisArgnt{n}} + \stateInterceptSTArg{n}\otimes \onesFunctionArg{n} \subseteq \RR{\fomdim} \otimes \timeSpaceArg{n}.
\end{equation}
Here $\stbasisArgnt{n} \in \RR{\fomdim \times
\stdimArg{n}}\otimes\timeSpaceArg{n}$, $n=1,\ldots,\nslabs$, with $\stbasisArgnt{n}: \timeDummy \mapsto \stbasisArgnt{n}(\timeDummy)$ and $\stbasisArg{n}{\timeStartArg{n}} = \bz$ is the space--time trial basis matrix {function} and $\stateInterceptSTArg{n} \in \RR{\fomdim}$, $n=1,\ldots,\nslabs$ provides the affine transformation. 
To enforce the initial condition and ensure solution continuity across time windows, we set $\stateInterceptSTArg{1} = \stateFOMIC$ and
$\stateInterceptSTArg{n} = \approxstateArg{n-1}{\timeEndArg{n-1}}$ for
$n=2,\ldots,\nslabs$.
At any time instance $t \in [\timeStartArg{n},\timeEndArg{n}]$, the \spaceTimeAcronym\ trial subspace approximates the FOM ODE solution as
 \begin{equation}\label{eq:stapprox}
 \stateFOMArg{n}{t} \approx \approxstateArg{n}{t}  = \stbasisArg{n}{t} \stgenstateArg{n} + \stateInterceptSTArg{n},
\end{equation}
where $\stgenstateArg{n} \in \RR{\stdimArg{n}}$ are the space--time generalized coordinates over the $n$th window. 
Setting $\stspace^n \leftarrow \stspaceSTArg{n}$ in the
\methodAcronym\ minimization problem~\eqref{eq:tclsrm} implies that
\methodAcronym\ with \spaceTimeAcronym\ trial subspaces sequentially computes
solutions $\stgenstateArg{n}$, $n = 1,\ldots,\nslabs$ that satisfy
\begin{align}\label{eq:obj_gen_slab_spacetime}
\begin{split}
&\underset{\stgenstateyArg{} \in \RR{\stdimArg{n}}}{\text{minimize}}\; \mathcal{J}^n( \stbasisArgnt{n} \stgenstateyArg{} + \stateInterceptSTArg{n} \otimes \onesFunctionArg{n} ).
\end{split}
\end{align}

\subsubsection{Stationary conditions} 
The key difference between \spaceTimeAcronym\ and \spatialAcronym\ trial
	subspaces is as follows: generalized coordinates for \spaceTimeAcronym\ trial subspaces 
	comprise a vector in $\RR{\stdimArg{n}}$, while generalized coordinates for
	the \spatialAcronym\ trial subspaces comprise a
	time-dependent vector in
	$\RR{\romdimArg{n}} \otimes \timeSpaceArg{n}$. 
Thus, in the \spaceTimeAcronym\ case, the optimization problem is no longer
	minimizing a functional with respect to a (time-dependent) function, but is minimizing a function with 
respect to a vector.  As such, the first-order optimality conditions can be
	derived using standard calculus. Differentiating the objective function with respect
	to the generalized coordinates and setting the result equal to zero yields 
\begin{equation}\label{eq:st_stationary}
 \intSlabArg{n} \bigg[ \stbasisDotArg{n}{t}^T  - \stbasisArg{n}{t}^T \bigg[\frac{\partial
\velocity}{\partial \stateyDiscrete} (\stbasisDotArg{n}{t} \stgenstateArg{n} +                    
\stateInterceptSTArg{n},t)\bigg]^T  \bigg] \stweightingMatArg{n} \bigg( \stbasisDotArg{n}{t} \stgenstateArg{n}  - \velocity (\stbasisArg{n}{t} \stgenstateArg{n} + \stateInterceptSTArg{n},t) \bigg) dt = \bz.\end{equation}
Thus, for \spaceTimeAcronym\ trial subspaces, the stationary conditions
	comprise a system of algebraic equations, as opposed to a system of
	differential equations as with \spatialAcronym\ trial subspaces.

%% file: section4/numerical_methods.tex
\section{Numerical solution techniques}\label{sec:numerical_techniques}
We now consider numerical-solution techniques for the \methodAcronym\ \approachKwd.
We focus primarily on 
direct (i.e., discretize-then-optimize) and indirect (i.e., optimize-then-discretize) methods for \spatialAcronym\ trial subspaces, 
and then briefly outline direct and indirect methods for \spaceTimeAcronym\ trial subspaces. 
\subsection{\spatialAcronym: direct and indirect methods}
\methodAcronym\ with \spatialAcronym\ trial subspaces can be viewed as an optimal-control problem.
Numerical solution techniques for
optimal-control problems can be classified as either
\textit{direct} or \textit{indirect}
methods~\cite{conway_optimalcontrolreview}. Rather than working with the first-order optimality conditions~\eqref{eq:lspg_continuous}--\eqref{eq:lspg_bcs}, direct methods ``directly" solve the optimization problem~\eqref{eq:tclsrm} by first
numerically discretizing the objective functional and ``transcribing"
the infinite-dimensional problem into a finite-dimensional 
problem. Direct approaches thus
``discretize then optimize''. 
In contrast, indirect methods compute solutions to the Euler--Lagrange
equations~\eqref{eq:lspg_continuous}--\eqref{eq:lspg_bcs} (which comprise the
first-order optimality conditions). Thus, indirect methods
``optimize then discretize," and solve the optimization problem
``indirectly.'' A variety of both discretization and solution techniques are possible
for both direct and indirect methods. Collocation methods,
finite-element methods, spectral methods, and shooting methods are all examples of
possible solution techniques.  

In the present context, we investigate both direct and indirect methods to solve \methodAcronym\ with \spatialAcronym\ trial
subspaces. In particular, we consider:
\begin{itemize} \item \textit{Direct methods (discretize then optimize)}: A
			direct approach that leverages linear multistep methods to directly solve the optimization problem~\eqref{eq:obj_gen_slab}.
Section~\ref{sec:direct} outlines this approach.
\item \textit{Indirect methods (optimize then discretize)}: An indirect
	approach that leverages the forward--backward sweep algorithm to solve the
		Euler--Lagrange equations~\eqref{eq:lspg_continuous}--\eqref{eq:lspg_bcs}.
Section~\ref{sec:indirect} outlines this technique.
\end{itemize} 
We note that a variety of other approaches exist, and their investigation
comprises the subject of future work.

\input{section4/spatial_numerical}

\input{section4/spacetime_numerical}

%% file: section4/spatial_numerical.tex
\subsection{\spatialAcronym\ trial subspaces: direct solution approach}\label{sec:direct} 

Direct approaches solve optimization problem~\eqref{eq:tclsrm} by
``transcribing" the infinite-dimensional optimization problem into a finite-dimensional one by discretizing the state and objective functional in time.
The minimization problem is then reformulated as a (non)linear 
optimization problem. A variety of direct solution approaches exist, 
including collocation approaches, spectral  methods, and genetic algorithms.  
In the context of \spatialAcronym\ trial subspaces, the most straightforward direct solution approach consists of the 
following steps: (1) numerically discretize the FOM ODE (and hence the
\textit{integrand} of the objective function in problem~\eqref{eq:obj_gen_slab}) and 
(2) select a numerical quadrature rule to evaluate the \textit{integral}
defining the objective function in~\eqref{eq:obj_gen_slab}.
To this end, we define time grids
$\{\timeWindowArg{n}{i}\}_{i=0}^{\nstepsArg{n}}\subset[\timeStartArg{n},\timeEndArg{n}]$,
$n=1,\ldots,\nslabs$ that
satisfy 
$\timeStartArg{n}=\timeWindowArg{n}{0} \le \cdots \le \timeWindowArg{n}{\nstepsArg{n}} 
 = \timeEndArg{n}$. 
Figure~\ref{fig:slab_fig2} depicts such a discretization.
For the purposes of indexing between different windows, we additionally define a function:
$$\indexMapper: (n,i) \mapsto 
\begin{cases}
	(n,i) & n = 1, \; i = 0, \\
	(n,i) & n \geq 1, \; i > 0, \\
\indexMapper(n-1,\nstepsArg{n-1}+i) & n > 1, \; i \le 0.
\end{cases}$$
We now outline the direct solution approach for linear-multistep schemes; the formulation for
other time-integration methods (e.g., Runge--Kutta) follows closely. 
\begin{figure} 
\begin{centering} 
\includegraphics[trim={0.0cm 4.5cm 0cm 3cm},clip,width=1.0\textwidth]{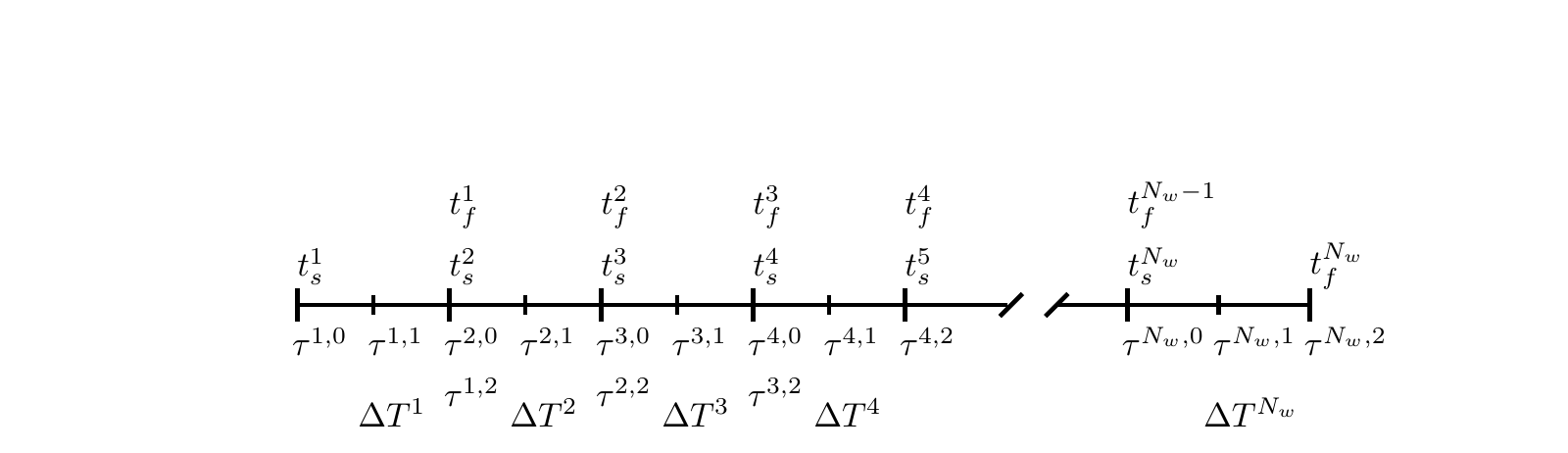} 
	\caption{Depiction of the $\nstepsArg{n}+1$ time instances over each window. In the figure, $\nstepsArg{n} = 2$ for all $n$.} 
\label{fig:slab_fig2} 
\end{centering} 
\end{figure}

\subsubsection{Linear multistep schemes}\label{sec:lmm}
Linear multistep schemes approximate the solution at time instance $\timeWindowArg{n}{i}$ using the previous $\lmsWidthArg{n}{i}$ time instances, where $\lmsWidthArg{n}{i}$ denotes the number of time steps employed by the scheme at the $i$th time instance of the $n$th window. 
Employing such a method to discretize the FOM ODE yields the sequence of FOM
O$\Delta$Es defined over the $n$th time window
\begin{align*}
&\residLMSWArg{n}{i} (\stateFOMDiscreteArg{n,i};\stateFOMDiscreteLMSArg{n}{i-1},\ldots,\stateFOMDiscreteLMSArg{n}{i-\lmsWidthArg{n}{i}}) = \bz, \qquad i=1,\ldots,\nstepsArg{n}
\end{align*}
along with the initial condition $\stateFOMDiscreteArg{1,0} = \stateFOMIC$.
Here, $\stateFOMDiscreteArg{n,i} (\approx
\stateFOMArg{}{\timeWindowArg{n}{i}})
\in \RR{\fomdim}$ and $\residLMSWArg{n}{i}$ denotes the FOM O$\Delta$E
residual over the $i$th time instance of the $n$th window  
defined as
\begin{align*}
\residLMSWArg{n}{i} &: (\stateyDiscreteArgnt{i};\stateyDiscreteArgnt{i-1},\ldots,\stateyDiscreteArgnt{i-\lmsWidthArg{n}{i}}) \mapsto  \frac{1}{\Delta t^{n,i}} \sum_{j=0}^{\lmsWidthArg{n}{i}} \alpha^{n,i}_j \stateyDiscreteArgnt{i-j} -  \sum_{j=0}^{\lmsWidthArg{n}{i}} \beta^{n,i}_j \velocity(\stateyDiscreteArgnt{i-j},\timeWindow^{\indexMapper(n,i-j)}), \\
               &: \RR{\fomdim} \otimes \RR{\lmsWidthArg{n}{i}+1} \rightarrow \RR{\fomdim}. 
\end{align*}
Here, $\Delta t^{n,i} \defeq \timeWindowArg{n}{i} - \timeWindowArg{n}{i-1}$
denotes the time step, and
$\alpha^{n,i}_j,\beta^{n,i}_j\in\RR{}$ denote coefficients that 
define the specific type of multistep scheme at the $i$th time instance of the $n$th window. 
Employing a linear multistep method allows the objective \textit{functional}~\eqref{eq:obj} to be replaced with the objective \textit{function}
\begin{equation*}
\begin{split} 
\objectiveArgLMS{n}_{\text{D}} &\vcentcolon
	(\stateyDiscreteArgnt{1},\ldots,\stateyDiscreteArgnt{\nstepsArg{n}};
	\stateyDiscreteArgnt{0},\ldots,\stateyDiscreteArgnt{-\lmsWidthArg{n}{1}+1})
	\mapsto 
\frac{1}{2} \sum_{i=1}^{\nstepsArg{n}} \quadWeightsLMSScalarArg{n}{i} [\residLMSWArg{n}{i}(\stateyDiscreteArgnt{i};\stateyDiscreteArgnt{i-1},\ldots, \stateyDiscreteArgnt{i-\lmsWidthArg{n}{i})})]^T  \stweightingMatArgt{n}{\timeWindowArg{n}{i}} \residLMSWArg{n}{i}(\stateyDiscreteArgnt{i};\stateyDiscreteArgnt{i-1},\ldots, \stateyDiscreteArgnt{i-\lmsWidthArg{n}{i}}), \\
&\vcentcolon \RR{\fomdim} \otimes \RR{\nstepsArg{n} + \lmsWidthArg{n}{1}} \rightarrow
\RR{}_+, 
\end{split}
\end{equation*}
where $\quadWeightsLMSScalarArg{n}{i} \in \RRplus$ are quadrature weights.  

\methodAcronym\ with the direct approach and a linear multistep method sequentially computes the solutions
$\approxstateDiscreteArg{n,1},\ldots,\approxstateDiscreteArg{n,\nstepsArg{n}}$,
$n=1,\ldots,\nslabs$ that satisfy
\begin{align*}
	&\underset{(\stateyDiscreteArgnt{1},\ldots,\stateyDiscreteArgnt{\nstepsArg{n}}) \in ( \trialspace + \stateInterceptArg{n}) \otimes \RR{\nstepsArg{n}}}{\text{minimize } }
\objectiveArgLMS{n}_{\text{D}} (\stateyDiscreteArgnt{1},\ldots,\stateyDiscreteArgnt{\nstepsArg{n}} ; \approxstateDiscreteArg{\indexMapper(n,0)},\ldots,\approxstateDiscreteArg{\indexMapper(n,-\lmsWidthArg{n}{1}+1)}), 
\end{align*}
with the initial condition $\approxstateDiscreteArg{1,0} = \basisspaceArg{n}\basisspaceTArg{n} (\stateFOMIC - \stateInterceptArg{n}) + \stateInterceptArg{n}$.


Equivalently, \methodAcronym\ with the direct approach and a linear multistep method sequentially computes the generalized
	coordinates
	$\genstateDiscreteArg{n,1},\ldots,\genstateDiscreteArg{n,\nstepsArg{n}}$,
	$n=1,\ldots,\nslabs$ 
with $\genstateDiscreteArg{n,i}(\approx
	\genstateArgnt{n}(\timeWindowArg{n}{i}))\in\RR{\romdimArg{n}}$
	that satisfy
\begin{equation}\label{eq:obj_gen_lms_final}
\begin{split}
	&
	\underset{(\genstateyDiscreteArg{1},\ldots,\genstateyDiscreteArg{\nstepsArg{n}})\in\RR{\romdimArg{n}}\otimes \RR{\nstepsArg{n}}}{\text{minimize } }
\objectiveArgLMS{n}_{\text{D}} (\basisspaceArg{n} \genstateyDiscreteArg{1} + \stateInterceptArg{n},\ldots,\basisspaceArg{n} \genstateyDiscreteArg{\nstepsArg{n}} + \stateInterceptArg{n}; \approxstateDiscreteArg{\indexMapper(n,0)},\ldots,\approxstateDiscreteArg{\indexMapper(n,-\lmsWidthArg{n}{1}+1)}). 
\end{split}
\end{equation}
The optimization problem takes the form of a \textit{weighted least-squares
	problem}. 
We emphasize that optimization problem~\eqref{eq:obj_gen_lms_final} associates
	with an \spatialAcronym\ trial subspace characterized by a reduction in
	spatial complexity, but no reduction in temporal complexity.
 
The minimization problem~\eqref{eq:obj_gen_lms_final} requires specification of the quadrature weights (and hence the integration scheme used to discretize 
the objective functional). Typically, the same integration scheme used to discretize the FOM ODE is employed for consistency~\cite{colloc_review}; e.g., if a  
backward Euler method is used to discretize the FOM ODE, then a backward Euler method is the used to numerically integrate the objective functional.

\begin{remark}
For the limiting case where $\nstepsArg{n} = 1$, $n=1,\ldots,\nslabs$ such that the window size is equivalent to the time step (i.e., $\DeltaSlabArg{n} = \timeWindowArg{n}{1} - \timeWindowArg{n}{0}$), uniform 
quadrature weights are used, a uniform trial space is employed (i.e., $\trialspaceArg{n} = \trialspace$ and $\stateInterceptArg{n} = \stateInterceptArg{}, n=1,\ldots,\nslabs$), the weighting matrices are taken to be
$\stweightingMatOneArg{n} = \lspgWeightingArg{}, n=1,\ldots,\nslabs$, and the
	time instances satisfy $\timeWindowArg{n}{1} = t^n$, $n=1,\ldots,\nslabs$, then 
$\approxstateDiscreteArg{n,1}  = \approxstateLSPG^n$, $n=1,\ldots,\nslabs$ and  
\methodAcronym\ with \spatialAcronym\ trial
	subspaces solved via the direct approach recovers the LSPG approach. 
\end{remark} 
\subsubsection{Solution to the least-squares problem through the Gauss--Newton
	method}
	Problem \eqref{eq:obj_gen_lms_final} corresponds to a discrete least-squares
	problem, which is nonlinear if the full-order-model velocity $\velocity$ is nonlinear in
	its first argument. 
A variety of algorithms exist for solving nonlinear least-squares problems,
	including the Gauss--Newton method, and the Levenberg--Marquardt method.
 The numerical experiments
	presented in this work consider nonlinear dynamical systems and are solved
	via the Gauss--Newton method; as such, we outline this approach here. 

Defining a ``vectorization" function 
\begin{align*}
 \unroll &\vcentcolon (\stateyDiscreteArg{1},\ldots,\stateyDiscreteArg{m} ) \mapsto \begin{bmatrix} [\stateyDiscreteArg{1}]^T & \ldots & [\stateyDiscreteArg{m}]^T \end{bmatrix}^T , \\
&\vcentcolon \RR{p} \otimes \RR{m} \rightarrow \RR{pm},
\end{align*}
the vectorized generalized coordinates over the $n$th time window are 
defined as 
\begin{equation*}
\genstatecollocMatSlabArg{n} \defeq 
\unroll (\genstateDiscreteArg{n,1},\ldots,\genstateDiscreteArg{n,\nstepsArg{n}} ).
\end{equation*}
We now define the weighted space--time residual over the entire window as
\begin{equation*}
\residLMSSlabArg{n} : \genstatecollocMatySlabArg{} \mapsto \begin{bmatrix}
 \sqrt{\frac{\quadWeightsLMSScalarArg{n}{1}}{2}} \stweightingMatOneArg{n} \residLMSArg{n,1}( \basisspaceArg{n} \genstateyDiscreteArgnt{1} + \stateInterceptArg{n}; 
 \approxstateDiscreteArg{\indexMapper(n,0)},
\ldots , \approxstateDiscreteArg{\indexMapper(n,1 - \lmsWidthArg{n}{1}) }   ) \\
\vdots \\
 \sqrt{\frac{\quadWeightsLMSScalarArg{n}{\nstepsArg{n}}}{2}} \stweightingMatOneArg{n} \residLMSArg{n,\nstepsArg{n}}( \basisspaceArg{n} \genstateyDiscreteArgnt{\nstepsArg{n}} + \stateInterceptArg{n}; \basisspaceArg{n} \genstateyDiscreteArgnt{\nstepsArg{n}-1} + \stateInterceptArg{n} ,  \ldots, \basisspaceArg{n} \genstateyDiscreteArgnt{\nstepsArg{n} - \lmsWidthArg{n}{\nstepsArg{n}} } + \stateInterceptArg{n} ) \\
\end{bmatrix},
\end{equation*}
where $\genstatecollocMatySlabArg{} \equiv \unroll (\genstateyDiscreteArg{1},\ldots,\genstateyDiscreteArg{\nstepsArg{n}} )$ and, for $i \le 0$,
\begin{equation}\label{eq:wls_genspatial_direct_args}
\genstateyDiscreteArgnt{n,i}  \equiv 
\begin{cases}
\genstateDiscreteArg{\indexMapper(n,i)}  & n = 2,\ldots,\nslabs,  \\
\genstateICOne  & n = 1. 
\end{cases}
\end{equation}
We note that
\begin{equation*}
\objectiveArgLMS{n}_{\text{D}} \bigg( \approxstateDiscreteArg{n,1} ,\ldots,\approxstateDiscreteArg{n,\nstepsArg{n}}  ; \approxstateDiscreteArg{\indexMapper(n,0)},\ldots,\approxstateDiscreteArg{\indexMapper(n,-\lmsWidthArg{n}{1}+1)}  \bigg) 
=
\bigg[\residLMSSlabArg{n}  (\genstatecollocMatSlabArg{n}) \bigg]^T \bigg[ \residLMSSlabArg{n}(\genstatecollocMatSlabArg{n}) \bigg].
\end{equation*} 
Using these definitions, Algorithm~\ref{alg:colloc_gn} presents the standard
Gauss--Newton method. Each Gauss--Newton iteration consists of three fundamental steps: (1)
compute the FOM O$\Delta$E residual given the current guess, (2) compute the
Jacobian of the residual over the time window, and (3) solve the linear least-squares problem and update the guess. 

The practical implementation of the Gauss--Newton algorithm requires an
efficient method for computing the Jacobian of the residual over the time
window
${\partial \residLMSSlabArg{n}}/{\partial \genstatecollocMatySlabArg{}}$.
For this purpose, we can leverage the fact that this Jacobian is a block lower
triangular matrix with the following sparsity pattern (for $\lmsWidthArg{n}{i} = 1$,
$i=1,\ldots,\nstepsArg{n}$): 
\begin{equation*} \begin{bmatrix*}[l]
\matshapea & \\
 \matshapea & \matshapea & \\
 & \matshapea  & \matshapea & \\
&  & \ddots & \\
 & &  & \matshapea &  \matshapea 
\end{bmatrix*},
\end{equation*}
where each block comprises an $\fomdim \times \romdimArg{n}$ dense matrix.
In particular, solution techniques can leverage this structure, e.g., to
efficiently compute Jacobian--vector products.
Another consequence of this sparsity pattern is that the normal equations
arising at each Gauss--Newton iteration 
comprise a banded block system that can also be exploited.

\begin{remark}\label{remark:gaussnewton}\textit{(Acceleration of the
	Gauss--Newton Method)}\\
	The principal cost of a Gauss--Newton method is often the formation of the Jacobian matrix. A variety of techniques aimed at 
reducing this computational burden exist; Jacobian-free Newton--Krylov
	methods~\cite{jfnk}, Broyden's method~\cite{broyden} (as explored in
	Ref.~\cite{carlberg_thesis}, Appendix A), and frozen Jacobian approximations
	are several such examples. Further, the space--time formulation introduces
	an extra dimension for parallelization that can be exploited to future
	reduce the wall time. The investigation of these additional, potentially
	more efficient, solution algorithms is a topic of future work. 
\end{remark}
\begin{algorithm}
\caption{\spatialAcronym\ trial subspace: algorithm for the direct solution technique with the Gauss--Newton method and a linear multistep method over the $n$th window}
\label{alg:colloc_gn}
\SetKwInOut{Input}{Input}\SetKwInOut{Output}{Output}
\Input{tolerance, $\epsilon$; initial guess, $\genstateGuessDiscreteArg{n,1}{0},\ldots,\genstateGuessDiscreteArg{n,\nstepsArg{n}}{0}$}
\Output{Solution to least squares problem, $\genstatecollocMatSlabArg{n}$} 
\textbf{Online Steps}: \\
$\text{converged} \leftarrow \text{false}$ \Comment{Set convergence checker} \\
$\genstatecollocMatSlabArg{n}_0 \leftarrow \unroll(\genstateGuessDiscreteArg{n,1}{0},\ldots,\genstateGuessDiscreteArg{n,\nstepsArg{n}}{0})$ \Comment{Assemble generalized coordinates over window} \\
$k \leftarrow 0$ \Comment{Set counter}\\
\While{\text{not converged}}
{
$\mathbf{r} \leftarrow \residLMSSlabArg{n}(\genstatecollocMatSlabArg{n}_k)$ \Comment{Compute weighted residual over window} \\
$\mathbf{J} \leftarrow  
\frac{\partial \residLMSSlabArg{n}}{\partial \genstatecollocMatySlabArg{}}(\genstatecollocMatSlabArg{n}_k) 
$ \Comment{Compute weighted residual-Jacobian over window} \\
\uIf{ $\norm{ \mathbf{J}^T\mathbf{r}  } \le \epsilon$ }{
{\text{converged} $\leftarrow$ \text{true}}  \Comment{Check and set convergence based on gradient norm} \\
Return: $\genstatecollocMatSlabArg{n} = \genstatecollocMatSlabArg{n}_{k+1} $ \Comment{Return converged solution}\\
}
\Else{
 Compute  $\Delta  \genstatecollocMatSlabArg{n} $ that minimizes $\norm{ \mathbf{J} \Delta \genstatecollocMatSlabArg{n} +\mathbf{r}}^2$ \Comment{Solve the linear least-squares problem} \\
$\alpha \leftarrow \text{linesearch}(\Delta  \genstatecollocMatSlabArg{n} ,\genstatecollocMatSlabArg{n}_k ) $ \Comment{Compute $\alpha$ based on a line search, or set to 1}\\
$\genstatecollocMatSlabArg{n}_{k+1} \leftarrow \genstatecollocMatSlabArg{n}_k + \alpha \Delta \genstatecollocMatSlabArg{n}$ \Comment{Update guess to the state} \\
}
$k\leftarrow k+1$
}
\end{algorithm}

\subsection{\spatialAcronym\ trial subspaces: indirect solution approach}\label{sec:indirect}
In contrast to the direct approach,
indirect methods ``indirectly" solve the minimization
problem~\eqref{eq:tclsrm} by solving the Euler--Lagrange
equations~\eqref{eq:lspg_continuous}--\eqref{eq:lspg_adjoint} associated with
stationarity. This
system comprises a coupled two-point boundary value
problem. Several techniques have
been devised to solve such problems, 
including shooting methods, multiple shooting
methods~\cite{multiple_shooting}, and the forward--backward sweep
method~\cite{fbs} (FBSM).  This work explores using the FBSM. 


\subsubsection{Forward--backward sweep method (FBSM)}\label{sec:FBSM}


Until convergence, the FBSM alternates between solving the
system~\eqref{eq:lspg_continuous} \textit{forward} in time given a fixed value
of the costate, and solving the adjoint equation~\eqref{eq:lspg_adjoint}
\textit{backward} in time given a fixed value for the generalized coordinates.
Typically, the system~\eqref{eq:lspg_continuous} is solved first given an
initial guess for the costate.
%
Algorithm~\ref{alg:st_iter} outlines the algorithm, which 
contains three parameters: the relaxation factor $\rho \le 1$, the growth factor
$\fbsmGrowth \ge 1$, and the decay factor $\fbsmDecay \ge 1$. The relaxation factor controls the rate at which the costate seen by~\eqref{eq:lspg_continuous} is updated. The
closer $\rho$ is to unity, the faster the algorithm will converge.
For large window sizes, however, a large a value of $\rho$ can lead to an unstable iterative process. 
In practice, a line search is used to compute an acceptable value for the
relaxation factor $\rho$. The line search presented in Algorithm~\ref{alg:st_iter} adapts the relaxation factor
according to the objective. Convergence properties of the FBSM method are
presented in Ref.~\cite{McAsey2012ConvergenceOT}, which shows that the
algorithm will converge for a sufficiently small 
value of $\rho$.

\begin{algorithm} \caption{\spatialAcronym\ trial subspace: algorithm for the FBSM over the $n$th window.} \label{alg:st_iter} 
\SetKwInOut{Input}{Input}\SetKwInOut{Output}{Output}
\Input{tolerance, $\epsilon$; relaxation factor, $\rho \le 1$; growth factor, $\fbsmGrowth \ge 1$; decay
factor, $\fbsmDecay \ge 1$; initial guess for state $\genstate^n_0$; initial guess for costate $\controllerArgnt{n}$} 
\Output{Stationary point, $\genstateArgnt{n}$ }
\textbf{Online Steps:}\\ 
$\text{Compute } \genstateArgnt{n}_1 \text{ satisfying }  \massArg{n} \genstateDotArgnt{n}_1(t)  -  \basisspaceTArg{n} \stweightingMatArg{n}
\velocity(\veloargsromArg{1}) =  \massArg{n} \controllerArg{n}{t}$ 
\Comment{Solve~\eqref{eq:lspg_continuous}}\\ 
$i \leftarrow 1$ \Comment{Set counter}\\
\While{$\epsilon \le \int_{\timeStartArg{n}}^{\timeEndArg{n}} \norm{\genstate^n_{i}(t) - \genstate^n_{i-1}(t) }dt $}{
\small{
\begin{multline*}
\text{ Compute } \adjointArgnt{n} \text{ satisfying }
\massArg{n} \adjointDotArg{n}{t}  + \basisspaceTArg{n} \bigg[\frac{\partial \velocity}{\partial \stateyDiscrete}(\basisspaceArg{n} \genstate^n_i(t) + \stateInterceptArg{n},t) \bigg]^T \stweightingMatArg{n} \basisspaceArg{n} \adjointArg{n}{t}= \\ -\bigg[\basisspaceTArg{n} \bigg[ \frac{\partial \velocity}{\partial \stateyDiscrete} ( \basisspaceArg{n} \genstate^n_i(t) + \stateInterceptArg{n},t) \bigg]^T \stweightingMatArg{n} \bigg( \mathbf{I} -   \basisspaceArg{n} [\massArg{n}]^{-1} \basisspaceTArg{n} \stweightingMatArg{n} \bigg)  \bigg( \basisspaceArg{n} \dot{\genstate}_i^n(t)   -   \velocity( \basisspaceArg{n} \genstate_i^n(t)  +\stateInterceptArg{n},t) \bigg) \bigg] 
\end{multline*} }
\Comment{Solve~\eqref{eq:lspg_adjoint} to obtain guess to costate} \\
$\controllerArgnt{n}  \leftarrow \rho \controllerArgnt{n} + (1 - \rho) \adjointArgnt{n}$ \Comment{Weighted update to costate}\\
$i \leftarrow i+1$ \Comment{Update counter}\\
$\text{Compute }\genstateArgnt{n}_i \text{ satisfying } \massArg{n} \genstateDotArgnt{n}_i(t)   -  \basisspaceTArg{n} \stweightingMatArg{n} \velocity(\basisspaceArg{n} \genstate^n_i(t) + \stateInterceptArg{n},t) = \massArg{n} \controllerArg{n}{t} $
\Comment{Solve~\eqref{eq:lspg_continuous}}\\
\uIf{ $\objectiveArg{n}({\basisspaceArg{n}\genstate_i^n + \stateInterceptArg{n}\otimes \onesFunctionArg{n}}) \le \objectiveArg{n}({\basisspaceArg{n}\genstate_{i-1}^n + \stateInterceptArg{n}\otimes \onesFunctionArg{n}})$}
{
$\rho \leftarrow \text{min}(\rho \fbsmGrowth,1)$ \Comment{Grow the relaxation factor}\\
}
\Else{
$\rho \leftarrow \frac{\rho }{ \fbsmDecay}$ \Comment{Shrink the relaxation factor}\\ 
$\genstate_i^{n} \leftarrow  \genstate_{i-1}^{n}$ \Comment{Reset state to value at previous iteration}
}
}
Return converged solution, $\genstateArgnt{n}= \genstateArgnt{n}_i$
\end{algorithm}
\subsubsection{Considerations for the numerically solving the forward and
backward systems}
The FBSM requires solving the forward~\eqref{eq:lspg_continuous} and
backward~\eqref{eq:lspg_adjoint} systems, both of which are defined at the time-continuous level. 
The numerical implementation of the FBSM requires two main ingredients: (1) temporal discretization schemes for the forward and backward problems and (2) 
an efficient method for computing terms involving the transpose of the
Jacobian of the velocity.

This work employs linear multistep schemes for time discretization of the
forward and backward problems. As described in Section \ref{sec:lmm},
temporal discretization is achieved by introducing $\nstepsArg{n} + 1$ time
instances over each time window.

The second ingredient, namely devising an efficient method for computing terms
involving the transpose of the Jacobian of the velocity,
can be challenging if one
does not have explicit access to this Jacobian or it is too costly to
compute. We discuss two methods that can be used to evaluate such terms
that appear in the forward system~\eqref{eq:lspg_adjoint}:
\begin{enumerate}
\item \textit{Jacobian-free approximation}: A non-intrusive way to evaluate
	these terms is to recognize that all terms including the transpose of the
		Jacobian of the velocity are left multiplied by the transpose of the
		spatial trial basis; this can be manipulated as
$$\basisspaceTArg{n} \bigg[\frac{\partial \velocity}{\partial \stateyDiscrete}
		(\veloargsromn)\bigg]^T = \bigg[  \frac{\partial \velocity}{\partial
		\stateyDiscrete} (\veloargsromn) \basisspaceArg{n} \bigg]^T,$$
		which exposes the ability to approximate rows of this matrix via
		finite differences, e.g., via forward differences as
$$\frac{\partial \velocity}{\partial \stateyDiscrete}(\veloargsromn)
		\basisvec_i^n \approx \frac{1}{\epsilon}\bigg(
		\velocity(\basisspaceArg{n}\genstateArg{n}{t} + \stateInterceptArg{n} +
		\epsilon \basisvec_i^n,t) - \velocity(\veloargsromn) \bigg),\quad
		i=1,\ldots,\romdimArg{n},$$
which requires $\romdimArg{n}+1$ evaluations of the velocity. 

\item \textit{Automatic differentiation}: A more intrusive, but exact
	method for computing these terms is
		through automatic differentiation (AD), which evaluate derivatives of
		functions (e.g., Jacobians, vector-Jacobian products) in a numerically
		exact manner by recursively applying the chain rule. The numerical
		examples presented later in this work leverage AD. The principal drawback
		of this approach is its intrusiveness, which may prevent them from
		practical application, e.g., in legacy codes.
\end{enumerate}

\begin{remark}\label{remark:fbsm}(Acceleration of Indirect Methods)
The FBSM is a simple iterative method for solving the coupled two-point boundary value problem. For large time windows, however, the FBSM may require many 
forward--backward iterations for convergence. More sophisticated solution
	techniques, such as a multiple FBSM method or multiple shooting methods, can 
reduce this cost in principle. Analyzing additional solution techniques is
	the subject of future work.
\end{remark}

%% file: section4/spacetime_numerical.tex
\subsection{\spaceTimeAcronym\ trial subspaces: direct and indirect methods}
We now consider \spaceTimeAcronym\ trial subspaces.  Because the optimization
variables (i.e., the generalized coordinates) in this case are already finite
dimensional, solvers for this type of trial subspace need only develop a
finite-dimensional representation of the objective
functional in problem~\eqref{eq:obj_gen_slab_spacetime}. We describe two
techniques for this purpose: a direct method 
that operates on the FOM O$\Delta$E and an indirect method that operates on the FOM ODE. 

\subsection{\spaceTimeAcronym\ trial subspaces: direct solution approach}
The direct solution technique seeks to minimize the fully discrete objective
function associated with the FOM O$\Delta$E. We consider linear
multistep methods and leverage the discretization introduced in
Section~\ref{sec:direct}. For notational simplicity, we define an
index-mapping function that is equivalent to the mapping function
$\indexMapper$, but outputs only the first argument: 
$$\indexMappern: (n,i) \mapsto 
\begin{cases}
n & n = 1, \; i = 0, \\
n & n \ge 1, \; i > 0, \\
\indexMappern(n-1,\nstepsArg{n-1}+i) & n > 1, \; i \le 0.
\end{cases}$$
\methodAcronym\ with an \spaceTimeAcronym\ trial subspace and the direct approach sequentially computes the generalized coordinates $\stgenstateArg{n}$, $n=1,\ldots,\nslabs$ that satisfy
 \begin{equation}\label{eq:obj_gen_lms_final_st}
\begin{split}
& \underset{\stgenstatey \in \RR{\stdimArg{n}}}{\text{minimize } }
\objectiveArgLMS{n}_{\text{D}} (\stbasisArg{n}{\timeWindowArg{n}{1}} \stgenstatey + \stateInterceptSTArg{n},\ldots,\stbasisArg{n}{\timeWindowArg{n}{\nstepsArg{n}}} \stgenstatey + \stateInterceptSTArg{n} ;\\
& \hspace{2. in}  \approxstateArg{\indexMappern(n,0)}{\timeWindow^{\indexMapper(n,0)}},\ldots, 
 \approxstateArg{\indexMappern(n,-\lmsWidthArg{n}{1}+1)}{\timeWindow^{\indexMapper(n,-\lmsWidthArg{n}{1}+1)}}).
\end{split} 
\end{equation}
The boundary conditions are automatically satisfied through the definition of the \spaceTimeAcronym\ trial subspace. 
Assuming $\text{Rank}(\stweightingMatOneArg{n}) \nstepsArg{n} \ge \stdimArg{n}$, the minimization problem~\eqref{eq:obj_gen_lms_final_st} again yields a least-squares problem.
\begin{remark}
	Comparing optimization problems \eqref{eq:obj_gen_lms_final} and
	\eqref{eq:obj_gen_lms_final_st} reveals that
\methodAcronym\ with the direct solution approach minimizes the same objective
	function in the case of both \spaceTimeAcronym\ and \spatialAcronym\ trial
	subspaces.
\end{remark}
\begin{remark}
For the limiting case where one window comprises the entire domain (i.e.,
	$\DeltaSlabArg{1}\equiv T$), uniform quadrature weights are used, the trial
	subspace is set to be $\stspaceSTArg{1} = \stspaceST$, the weighting matrix
	$\lspgWeightingST = \mathrm{diag}(\stweightingMatOneArg{1})$, and $\nstepsArg{1} = N_t$ time instances are employed that
	satisfy $\timeWindowArg{1}{i} = t^i$, $i=1,\ldots,N_t$, then
	$\stgenstateArg{1} =  \stgenstate_\text{ST-LSPG}$ and direct \methodAcronym\
	with an \spaceTimeAcronym\ trial subspace recovers ST-LSPG. 
\end{remark}

\begin{remark}
To enable equivalence in the case for a general ST-LSPG weighting matrix
	$\lspgWeightingSTArg{\cdot}$, the weighting matrix $\stweightingMatArg{1}$
	must be time dependent matrix-valued, which associates the objective function 
	\eqref{eq:obj_gen_lms_final_st} with a modified space--time
	norm. For notational simplicity, we do not consider this case in the current
	manuscript.
\end{remark}

\subsection{\spaceTimeAcronym\ trial subspaces: indirect solution approach}
As opposed to the direct approach, the indirect approach directly minimizes the continuous objective function~\eqref{eq:obj_gen_slab_spacetime} and sequentially computes solutions $\stgenstateArg{n}$, $n = 1,\ldots,\nslabs$ that satisfy
\begin{equation}\label{eq:obj_gen_slab2}
\begin{split}
 & \underset{\stgenstatey \in \RR{\stdimArg{n}}}{\text{minimize }} \mathcal{J}^n \bigg( \stbasisArgnt{n} \stgenstatey + \stateInterceptSTArg{n} \otimes \onesFunctionArg{n} \bigg) .
\end{split} 
\end{equation}
Numerically solving the minimization problem requires the introduction of a quadrature rule for 
discretization of the integral. To this end, we introduce $\ncollocSTArg{n} \ge \text{ceil}(\stdimArg{n}/ \text{rank}(\stweightingMatOneArg{n}))$ quadrature points over the $n$th window, $ \{ \collocPointSTArg{n}{i} \}_{i=1}^{\ncollocSTArg{n}} \subset [\timeStartArg{n},\timeEndArg{n}]$, $n=1,\ldots,\nslabs$. 
Leveraging these quadrature points, \methodAcronym\ with the indirect method and an \spaceTimeAcronym\ trial subspace computes the generalized coordinates 
$\stgenstateArg{n}$, $n=1,\ldots,\nslabs$ that satisfy
\begin{equation}\label{eq:obj_gen_slab2} 
\begin{split}
&\underset{\stgenstatey \in \RR{\stdimArg{n}}}{\text{minimize }} \objectiveDiscreteSTArg{n} \bigg( \stbasisArgnt{n} \stgenstatey + \stateInterceptSTArg{n} \otimes \onesFunctionArg{n}  \bigg) , 
\end{split} 
\end{equation}
where the discrete objective function is given by
\begin{equation*}
\begin{split}
\objectiveDiscreteSTArg{n} &\vcentcolon \statey \mapsto \frac{1}{2}\sum_{i=1}^{\ncollocSTArg{n}} \zeta^{n,i} 
\bigg[ \dot{\statey}(\collocPointSTArg{n}{i})  - \velocity (\stateyArg{}{\collocPointSTArg{n}{i}},\collocPointSTArg{n}{i} ) \bigg]^T 
\stweightingMatArg{n} 
\bigg[ \dot{\statey}(\collocPointSTArg{n}{i})  - \velocity (\stateyArg{}{\collocPointSTArg{n}{i}},\collocPointSTArg{n}{i} ) \bigg], \\
& \vcentcolon \RR{\fomdim} \otimes \timeSpaceArg{n} \rightarrow \RRplus,
\end{split}
\end{equation*}
and $\zeta^{n,i} \in \RRplus$, $i=1,\ldots,\ncollocSTArg{n}$ are quadrature
weights over the $n$th time window. The optimization problem~\eqref{eq:obj_gen_slab2} again comprises a least-squares problem.

\begin{remark}
\methodAcronym\ with \spaceTimeAcronym\ trial subspaces solved via the indirect approach naturally achieves ``collocation" in time as the full-order model residual needs to be queried at only the quadrature points. 
\end{remark}
\begin{remark}
For the limiting case where one window comprises the entire domain (i.e.,
	$\DeltaSlabArg{1}\equiv T$), the trial
	subspace is set to the span of full solution trajectories, $\stspaceSTArg{1}
	= \text{span}\{ \stateFOMArgnt{}_i\}_{i=1}^{\stdimArg{1}}$ (where
	$\stateFOMArgnt{}_i$ are obtained, e.g.,  from training simulations at
	different parameter instances), and the weighting matrix is set to
	$\stweightingMatArg{1} = \mathbf{I}$, \methodAcronym\ with
	\spaceTimeAcronym\ trial subspaces and the indirect approach closely
	resembles the model reduction procedure proposed in
	Ref.~\cite{constantine_strom}; the approaches differ only in that
	Ref.~\cite{constantine_strom} imposes the constraint
	$\sum_{i=1}^{\stdimArg{1}} \stgenstateArg{1}_i = 1$ in the associated
	minimization problem. 
\end{remark}
\subsection{\spaceTimeAcronym\ trial subspaces: summary}
\spaceTimeAcronym\ trial subspaces yield a series of space--time systems of algebraic equations over each window. As a variety of work has examined space--time reduced-order models with \spaceTimeAcronym\ trial subspaces, 
a detailed exposition of solution techniques for these systems is not pursued here. It is sufficient to say that \methodAcronym\ with \spaceTimeAcronym\ trial subspaces yields a series of dense systems to be solved over each window.

%% file: analysis.tex
\section{Analysis}\label{sec:analysis}
This section provides theoretical analyses of the \methodAcronym\ \approachKwd. First, we demonstrate equivalence conditions 
between \methodAcronym\  with (uniform) \spatialAcronym\ trial subspaces and
the Galerkin ROM in the limit $\DeltaSlabArg{n} \rightarrow 0$.
Next, we derive \textit{a priori} error bounds for autonomous systems. 
\subsection{Equivalence conditions}
\begin{theorem}\label{theorem:galerkin_equiv}\textit{(Galerkin equivalence)}
For sequential minimization over infinitesimal time windows and uniform
	\spatialAcronym\ trial spaces, i.e., $\trialspaceArg{n} = \trialspace$
	and $\stateInterceptArg{n} = \stateInterceptArg{}$, $n=1,\ldots,\nslabs$, the \methodAcronym\ approach (weakly) recovers
	Galerkin projection.
\end{theorem}
\begin{proof}
The \methodAcronym\ approach with uniform \spatialAcronym\ subspaces comprises solving the following sequence of minimization problems for $\genstateArgnt{n}$, $n=1,\ldots,\nslabs$,
\begin{equation}\label{eq:obj_proof}
\begin{split}
      & \underset{\genstateyArgnt{} \in \RR{\romdim} \otimes \timeSpaceArg{n}}{\text{minimize}}\; \mathcal{J}^n(\basisspace \genstateyArgnt{} + \stateIntercept \otimes \onesFunctionArg{n}), \\ 
      & \text{subject to }\; \genstateyArg{}{\timeStartArg{n}} =
\begin{cases} \genstateArg{n-1}{\timeEndArg{n-1}} & n = 2,\ldots,\nslabs \\
\genstateIC & n=1. \end{cases} 
\end{split}
\end{equation}
Following the derivation of the Euler--Lagrange equations presented in Appendix~\ref{appendix:eulerlagrange} leads to Eq.~\eqref{eq:euler_lagrange_analysis}. Setting $a = \timeStartArg{n}$, $b = \timeEndArg{n}$, and $\mathcal{I} = \minintegrandArg{n}$ in Eq.~\eqref{eq:euler_lagrange_analysis} yields the sequence of systems to be solved for $\genstate^n$ (and, implicitly, $\genstateDotArgnt{n}$) over $t \in [\timeStartArg{n},\timeEndArg{n}]$:
\begin{multline}\label{eq:g_equiv_1}
 \int_{\timeStartArg{n}}^{\timeEndArg{n}} \bigg( \frac{\partial \minintegrandArg{n}  }{\partial \genstateyDiscrete }(\genstateArg{n}{t},\genstateDotArg{n}{t},t)  \variationArgntt{n}{t}  - \frac{d}{dt}\bigg( \frac{\partial \minintegrandArg{n}}{\partial \genstateyDiscreteDot} (\genstateArg{n}{t},\genstateDotArg{n}{t},t ) \bigg) \variationArgntt{n}{t} \bigg)dt +\\ \bigg(\frac{\partial \minintegrandArg{n}}{\partial \genstateyDiscreteDot}(\genstateArg{n}{\timeEndArg{n}},\genstateDotArg{n}{\timeEndArg{n}},\timeEndArg{n}) \bigg) \variationArgntt{n}{\timeEndArg{n}}   = 0,
\end{multline}
for all functions $\variationArgn{n}: [\timeStartArg{n},\timeEndArg{n}] \rightarrow
	\RR{\romdim}$ that satisfy
	$\variationArgntt{n}{\timeStartArg{n}} = \bz$,
with the boundary conditions
\begin{equation*}
 \genstate^n(\timeStartArg{n})  = 
\begin{cases}
\genstate^{n-1}(\timeEndArg{n-1}) & n=2,\ldots,\nslabs, \\
\genstateIC & n=1. \end{cases}  
\end{equation*}
To examine what happens for infinitesimal time windows, we take a uniform
	window size and let $\timeEndArg{n} = \timeStartArg{n} + \zeta$,
	$n=1,\ldots,\nslabs$ such that $\timeStartArg{n} = \zeta (n-1)$ and
	$\timeEndArg{n} = \zeta n$, $n=1,\ldots,\nslabs$. 
Taking the limit $\zeta \rightarrow 0^+$ and noting that $\variationArgn{n}$
	is an arbitrary function, we obtain the following sequence of problems for
	$n=1,\ldots,\nslabs$:
\begin{equation*}
 \genstateArg{n}{\zeta(n-1)} = 
\begin{cases}
\genstate^{n-1}(\zeta(n-1)) & n=2,\ldots,\nslabs,\\
\genstateIC & n=1, \end{cases} \qquad
\bigg[ \frac{\partial \minintegrandArg{n}}{\partial \genstateyDiscreteDot}(\genstateArg{n}{\zeta n},\genstateDotArg{n}{\zeta n},\zeta n) \bigg]^T= \boldsymbol 0,
\end{equation*}
where the first term in Eq.~\eqref{eq:g_equiv_1} has vanished, as $
	\lim_{\zeta \rightarrow 0^+}  \int_{\zeta (n-1)}^{\zeta n} h(t) dt = 0 $ for
	any continuous function $h$.
Noting that the derivative evaluates to
\begin{equation*}
\bigg[ \frac{\partial \minintegrandArg{n}}{\partial \genstateyDiscreteDot}(\genstateArg{n}{\zeta n},\genstateDotArg{n}{\zeta n},\zeta n) \bigg]^T=
\basisspace^T \stweightingMatArgt{}{\zeta n} \basisspace \genstateDotArg{n}{\zeta n} -  \basisspace^T \stweightingMatArgt{}{\zeta n} \velocity(\basisspace \genstateArg{n}{\zeta n} + \stateIntercept ,\zeta n), 
\end{equation*}
we have 
\begin{equation*}
\basisspace^T \stweightingMatArgt{}{\zeta n} \basisspace \genstateDotArg{n}{\zeta n} -  \basisspace^T \stweightingMatArgt{}{\zeta n} \velocity(\basisspace \genstateArg{n}{\zeta n} + \stateIntercept ,\zeta n) = \bz, \qquad n=1,\ldots,\nslabs, 
\end{equation*}
with the boundary conditions  $\genstateArg{n}{\zeta (n-1)} = \genstateArg{n-1}{\zeta(n-1)}$ for $n=2,\ldots,\nslabs$ and $\genstateArg{1}{0} = \genstateIC $. 
In the limit of $\zeta \rightarrow 0^+$ (and hence $\nslabs
\rightarrow \infty$) this is a (weak) statement of the Galerkin ROM. 
\end{proof}

\subsection{\textit{A priori} error bounds}
We now derive \textit{a priori} error bounds for \spatialAcronym\ trial
subspaces in the case that no weighting matrix is
employed (i.e., $\stweightingMat = \mathbf{I}$).
We denote the error in the \methodAcronymROM\ solution over the $n$th window
as
\begin{align*}
	\errorArg{n}\vcentcolon& \timeDummy \mapsto \stateFOMSolArg{}(\timeDummy)-
	\stateROMSolArg{n} (\timeDummy),\\
	:&[\timeStartArg{n},\timeEndArg{n}]\rightarrow\RR{\fomdim},
\end{align*}
$n=1,\ldots,\nslabs$.
Additionally, we denote $\stateFOMProjSolArg{n}$, $n=1,\ldots,\nslabs$ to be
the $\elltwo$-optimal solution over the $n$th window
$$\stateFOMProjSolArg{n} = \underset{\statey \in
\stspace^n}{\text{arg}\,\text{min } } \intSlabArg{n} \norm{ \statey(t) - \stateFOMSolArgt{}{t} }^2 dt.$$ 
We employ the following assumptions.

\begin{itemize}
\item \textbf{A1:} The residual is Lipshitz continuous in the first argument, i.e., there exists
	$\lipshitz>0$ such that
$$ \norm{\resid(\statew,\timeDummy) - \resid(\statey,\timeDummy) } \le
		\lipshitz \norm{\statew(\timeDummy) - \statey(\timeDummy)},\quad \forall
		\statew,\, \statey\in\RR{N} \otimes \timeSpaceArg{}, \tau \in [0,T],$$
where  
\begin{align*}
\resid &: \, (\statey,\timeDummy) \mapsto \dot{\statey}(\timeDummy) - \velocity(\stateyArg{}{\timeDummy},\timeDummy) ,\\
&: \, \RR{N} \otimes \timeSpaceArg{} \times [0,T] \mapsto \RR{N}.
\end{align*}
\item \textbf{A2:} The velocity is Lipshitz continuous in its first argument, i.e., there exists
	$\lipshitzf>0$ such that
$$ \norm{\velocity(\statewDiscrete,\timeDummy) - \velocity(\stateyDiscrete,\timeDummy) } \le
		\lipshitzf \norm{\statewDiscrete - \stateyDiscrete},\quad \forall
		\statewDiscrete,\, \stateyDiscrete \in\RR{\fomdim} ,\tau \in [0,T].$$

\item \textbf{A3:} The integrated residual is inverse Lipshitz continuous in its first argument over
	each time window, i.e., there exist $\lipshitziArg{n}>0$,
		$n=1,\ldots,\nslabs$ such that
$$  \intSlabArg{n} \norm{\statew(t) - \statey(t)} dt \le  \lipshitziArg{n}
		\intSlabArg{n} \norm{\resid(\statew,t) - \resid(\statey,t) } dt,\quad
		\forall \statew,\statey \in \stspaceArg{n}_*,$$
		where
$\stspaceArg{n}_* = \{ \statew \in \RR{\fomdim} \otimes \timeSpaceArg{n}\, |\,
		\statew(\timeStartArg{n}) = \stateFOMArg{}{\timeStartArg{n}} \}$.
\item \textbf{A4:} The FOM solution at the start of each time window lies
	within the range of the trial subspace, i.e.,
$$ \stateFOMSolArgt{n}{\timeStartArg{n}} \in \trialspaceArg{n} +
		\stateInterceptArg{n},\quad n=1,\ldots,\nslabs.$$

\end{itemize} 
%

\begin{theorem}(\textit{A priori} error bounds)\label{theorem:apriori_bound}
Under Assumptions A1--A4, the error in the solution computed by the \methodAcronymROM\ \approachKwd\ with
	\spatialAcronym\ trial subspaces over the $n$th window is bounded as
\begin{equation}\label{eq:apriori_bound}
\intSlabArg{n} \norm{\errorArgt{n}{t}} dt \le  \norm{\errorArgt{n}{\timeStartArg{n}}} 
 \bigg(\frac{ e^{ \DeltaSlabArg{n}(\lipshitz + \lipshitzf)} - 1}{ \lipshitz + \lipshitzf} \bigg)
+ \lipshitziArg{n} \intSlabArg{n} \norm{ \resid(\stateFOMProjSolArg{n},t) } dt.
\end{equation}

\end{theorem}
\begin{proof}
To obtain an error bound over the $n$th window, we must account for the 
fact that the initial conditions into the $n$th window can be incorrect. To this end, 
we define new quantities $\stateROMStarSolArg{n}$, $n=1,\ldots,\nslabs$, where
	$\stateROMStarSolArg{n}$ is the
	solution to the minimization problem
\begin{equation}\label{eq:min_correct}
\begin{split}
& \underset{\statey \in \stspaceArg{n}}{\text{minimize } }
\objectiveArg{n}(\statey),\\
& \text{subject to } \statey(\timeStartArg{n})= \stateFOMSolArgt{}{\timeStartArg{n}}.
\end{split}
\end{equation}
Note that minimization problem~\eqref{eq:min_correct} is equivalent to the \methodAcronym\ minimization problem~\eqref{eq:tclsrm}, but uses the FOM solution for the initial conditions. 
Additionally, define $\adjointROMStarSolArg{n}$, $n=1,\ldots,\nslabs$ to be the costate solution associated with optimization problem~\eqref{eq:min_correct}.
 The error in the solution obtained by the 
\methodAcronymROM\ over the $n$th window at time $t \in [\timeStartArg{n},\timeEndArg{n}]$ can be written as
\begin{equation*}
\norm{ \stateROMSolArgt{n}{t} - \stateFOMSolArgt{}{t}} = 
\norm{\stateROMSolArgt{n}{t} - \stateROMStarSolArgt{n}{t} + \stateROMStarSolArgt{n}{t} -  \stateFOMSolArgt{}{t} }.
\end{equation*}
Applying triangle inequality yields
\begin{equation*}
\norm{ \stateROMSolArgt{n}{t} - \stateFOMSolArgt{}{t}} \le 
\norm{\stateROMSolArgt{n}{t} - \stateROMStarSolArgt{n}{t}} + \norm{ \stateROMStarSolArgt{n}{t} -  \stateFOMSolArgt{}{t} }.
\end{equation*}
Integrating over the $n$th window and using the definition of the error yields
$$\intSlabArg{n} \norm{\errorArgt{n}{t}} dt \le \intSlabArg{n} \norm{\stateROMSolArgt{n}{t} - \stateROMStarSolArgt{n}{t}} dt +  \intSlabArg{n} \norm{\stateROMStarSolArgt{n}{t} - \stateFOMSolArgt{}{t}}dt.$$
Applying Assumption A3 and 
	$\resid(\stateFOMSolArg{},t) = \bz$, $\forall t \in [0,T]$
	yields
\begin{equation*}
\intSlabArg{n} \norm{\errorArgt{n}{t}} dt \le \intSlabArg{n} \norm{\stateROMSolArgt{n}{t} - \stateROMStarSolArgt{n}{t}} dt + \lipshitziArg{n} \intSlabArg{n} \norm{ \resid(\stateROMStarSolArg{n},t) } dt.
\end{equation*}
Leveraging the residual-minimization property of \methodAcronym\ and noting
	that $\stateFOMProjSolArg{n}(\timeStartArg{n}) =
	\stateFOMSolArgt{n}{\timeStartArg{n}}$ by Assumption A4, we have 
$$ \intSlabArg{n} \norm{\resid(\stateROMStarSolArg{n},t)}dt \le \intSlabArg{n} \norm{\resid(\stateFOMProjSolArg{n},t)}dt.$$
	This leads to the following expression for the error over the $n$th window,
\begin{equation}\label{eq:boundtmp}
\intSlabArg{n} \norm{\errorArgt{n}{t}} dt \le \intSlabArg{n} \norm{\stateROMSolArgt{n}{t} - \stateROMStarSolArgt{n}{t}} dt + \lipshitziArg{n} \intSlabArg{n} \norm{ \resid(\stateFOMProjSolArg{n},t) } dt.
\end{equation}
We now derive an upper bound for $\intSlabArg{n} \norm{\stateROMSolArgt{n}{t} - \stateROMStarSolArgt{n}{t}} dt$.
Defining $\errorStarArg{n} = \genstateROMSolArg{n} - \genstateROMStarSolArg{n}$, $n=1,\ldots,\nslabs$, where $\genstateROMStarSolArg{n}$ are the generalized coordinates of $\stateROMStarSolArg{n}$ (i.e., $\stateROMStarSolArgt{n}{t} = \basisspaceArg{n} \genstateROMStarSolArgt{n}{t} + \stateInterceptArg{n}$), the differential equation for $\errorStarArg{n}$ is given by
$$\errorStarDotArgt{n}{t} = \basisspaceTArg{n}[\velocity(\basisspaceArg{n} \genstateROMSolArgt{n}{t} + \stateInterceptArg{n},t) - \velocity(\basisspaceArg{n} \genstateROMStarSolArgt{n}{t} + \stateInterceptArg{n},t) ] + \adjointROMSolArgt{n}{t} - \adjointROMStarSolArgt{n}{t},$$
for $t \in [\timeStartArg{n},\timeEndArg{n}]$ and with the initial condition $\errorStarArgt{n}{\timeStartArg{n}}.$ We have used the notation $ \errorStarDotArg{n}\equiv d \errorStarArg{n} / d\tau$. 
Taking the norm of both sides and applying triangle inequality yields
$$\norm{ \errorStarDotArgt{n}{t} } \le \norm{[\velocity(\basisspaceArg{n} \genstateROMSolArgt{n}{t} + \stateInterceptArg{n},t) - \velocity(\basisspaceArg{n} \genstateROMStarSolArgt{n}{t} + \stateInterceptArg{n},t) ]} + \norm{\adjointROMSolArgt{n}{t} - \adjointROMStarSolArgt{n}{t}},$$
with the initial condition $\norm{\errorStarArgt{n}{\timeStartArg{n}}}$. We note we have used $\norm{\basisspaceArg{n}} = 1$. Using the definition of the costate~\eqref{eq:costate_def} yields 
$$\norm{ \errorStarDotArgt{n}{t} } \le \norm{[\velocity(\basisspaceArg{n} \genstateROMSolArgt{n}{t} + \stateInterceptArg{n},t) - \velocity(\basisspaceArg{n} \genstateROMStarSolArgt{n}{t} + \stateInterceptArg{n},t) ]} + \norm{\basisspaceTArg{n} \big( \resid(\stateROMSolArg{n},t) - \resid(\stateROMStarSolArg{n},t) \big)}.$$
Employing assumptions A1-A2 yields the bound
$$\norm{ \errorStarDotArgt{n}{t} } \le (\lipshitzf + \lipshitz)  \norm{ \errorStarArgt{n}{t} } .$$
We use the fact that $ d \norm{ \errorStarArg{n} } / d\tau \le \norm{ \errorStarDotArg{n} } $ to get 
$$\big(\frac{d \norm{\errorStarArg{n}}}{d\tau} \big)(t)  \le (\lipshitzf + \lipshitz)  \norm{ \errorStarArgt{n}{t} } .$$
The above is a linear homogeneous equation for the bound of $\norm{\errorStarArg{n}}$ and has the solution for $t \in [\timeStartArg{n},\timeEndArg{n}]$
$$\norm{\errorStarArgt{n}{t}} \le \norm{ \errorStarArgt{n}{\timeStartArg{n}} }e^{(\lipshitz + \lipshitzf)(t - \timeStartArg{n})}.$$
Noting that $\norm{\errorStarArgt{n}{\timeStartArg{n}}} = \norm{\errorArgt{n}{\timeStartArg{n}}}$ we get the bound
\begin{equation}\label{eq:ustarbound} 
\intSlabArg{n} \norm{\errorStarArgt{n}{t}} dt \le \norm{\errorArgt{n}{\timeStartArg{n}}} \bigg(\frac{ e^{ \DeltaSlabArg{n}(\lipshitz + \lipshitzf)} - 1}{ \lipshitz + \lipshitzf} \bigg) .
\end{equation}
Substituting bound~\eqref{eq:ustarbound} into bound~\eqref{eq:boundtmp} and noting that $\norm{\stateROMSolArg{n} - \stateROMStarSolArg{n}} = \norm{\genstateROMSolArg{n} - \genstateROMStarSolArg{n}}$ gives the upper bound
\begin{equation*}
\intSlabArg{n} \norm{\errorArgt{n}{t}} dt \le  \norm{\errorArgt{n}{\timeStartArg{n}}} 
 \bigg(\frac{ e^{ \DeltaSlabArg{n}(\lipshitz + \lipshitzf)} - 1}{ \lipshitz + \lipshitzf} \bigg)
+ \lipshitziArg{n} \intSlabArg{n} \norm{ \resid(\stateFOMProjSolArg{n},t) } dt.
\end{equation*}

\end{proof}
\begin{corollary}
	For the case of one time window $\DeltaSlabArg{1} = T$, then under Assumptions A3--A4  
	the error in the solution computed by the \methodAcronymROM\ \approachKwd\ with
	\spatialAcronym\ trial subspaces is bounded as
\begin{equation*}
\int_0^T \norm{\errorArgt{1}{t}} dt \le \lipshitziArg{1} \int_0^T \norm{\resid(\stateFOMProjSolArg{1},t)}dt. 
\end{equation*}
\end{corollary}
\begin{proof}
Setting $n=1$ in~\eqref{eq:apriori_bound} with the time intervals
	$\timeStartArg{1}=0$, $\timeEndArg{1}=T$, noting that the initial conditions are known and employing Assumption A4 yields the desired result.
\end{proof}

\subsection{Discussion} 
Theorem~\ref{theorem:apriori_bound} provides \textit{a priori} bounds on the integrated normed error for \methodAcronym\ employing \spatialAcronym\ trial subspaces. We make several 
observations. First, it is observed \methodAcronym\ is subject to recursive error bounds (through the first term on the RHS in the upper bound~\eqref{eq:apriori_bound}). As the number of time windows grows, so does the recursive growth of error. Second, we observe that when a single window spans the entire domain, the error in the \methodAcronym\ with \spatialAcronym\ trial subspaces is bounded \textit{a priori} by the residual of the $\elltwo$-orthogonal projection of the FOM solution. 

%% file: numerical_experiments.tex
\section{Numerical experiments}\label{sec:numerical_experiments}
We now analyze the performance of \methodAcronymROMs\ leveraging \spatialAcronym\ trial subspaces on two benchmark problems: the Sod shock tube and compressible flow in a cavity. In each experiment, we compare \methodAcronymROMs\ to the Galerkin and LSPG ROMs. The purpose of the numerical experiments is to assess the impact of minimizing the residual over an arbitrarily sized time window on the solution accuracy. We additionally assess the impact of the time step and time scheme on \methodAcronym. In both experiments, the spatial basis is equivalent for each time window, e.g., $\basisspaceArg{n} \equiv \basisspace$, $n=1,\ldots,\nslabs$. We also note that both experiments are designed to test the \textit{reproductive} ability of the ROMs. We do not consider future state prediction and prediction at new parameter instances as these problems introduce factors that confound the solution accuracy with the solution methodology (e.g., accuracy of the basis). 
 
\subsection{Sod shock tube}
We first consider reduced-order models of the Sod shock tube problem, which is governed by the compressible Euler equations in one dimension, 
\begin{equation}\label{eq:euler_1D}
    \frac{\partial \boldsymbol u}{\partial t} + \frac{\partial \boldsymbol F}{\partial \mathsf{x}} = 0, \quad
    \boldsymbol u= 
    \begin{Bmatrix} \rho \\ \rho u \\ \rho E \end{Bmatrix}, \quad 
    \boldsymbol F = \begin{Bmatrix} \rho u \\ \rho u^2 + p \\  u(\rho E + p) \end{Bmatrix},
\end{equation}
where $\boldsymbol u : \Omega \times [0,T] \rightarrow \RR{3}$ comprise the density, $\mathsf{x}$-momentum, and total energy, $\mathsf{x} \in \Omega \defeq  [0,1]$ is the spatial domain, and 
$T = 1$ the final time. 
The problem setup is given by the initial conditions

\begin{equation*}
\rho = 
\begin{cases} 
      1 & \mathsf{x}\leq 0.5 \\
      0.125 & \mathsf{x} > 0.5 
   \end{cases},
\qquad
p = 
\begin{cases} 
      1 & \mathsf{x}\leq 0.5 \\
      0.1 & \mathsf{x} > 0.5 
   \end{cases},
\qquad
u = 
\begin{cases} 
      0 & \mathsf{x}\leq 0.5 \\
      0 & \mathsf{x} > 0.5 
   \end{cases},
\end{equation*}
along with reflecting boundary conditions at $\mathsf{x}=0$ and $\mathsf{x}=1$. 

\subsubsection{Description of FOM and generation of \spatialAcronym\ trial subspace}\label{sec:sod_fom}
We solve the 1D compressible Euler equations with a finite volume method. We partition the domain into 500 cells of uniform width and employ the Rusanov flux~\cite{rusanov} at the cell interfaces. We employ the Crank--Nicolson (CN) scheme, which is a linear multistep method defined by the coefficients $\alpha_0 = 1,\alpha_1 = -1, \beta_0 = \beta_1 = 1/2$, for temporal integration. We evolve the FOM for $t \in [0.0,1.0]$ at a time-step of $\Delta t = 0.002$. We construct the \spatialAcronym\ trial subspace by executing Algorithm~\ref{alg:pod} with inputs $N_{\text{skip}}=2, \stateIntercept = \bz, K=46 $. The resulting trial subspace corresponds to an energy criterion of $99.99\%$.

\subsubsection{Description of reduced-order models}
We consider reduced-order models based on Galerkin projection, LSPG
projection, and the \methodAcronym\ approach. No hyper-reduction is considered 
in this example, i.e., $\stweightingMatArg{n} = \mathbf{I}$, $n=1,\ldots,\nslabs$. Details on the implementation of the 
different reduced-order models is as follows:
\begin{itemize}
\item \textit{Galerkin ROM}: We obtain the Galerkin ROM through Galerkin projection of the FOM and evolve the Galerkin ROM in time with the CN time scheme at a constant time step of $ \Delta t = 0.002$.

\item \textit{LSPG ROM:} We construct the LSPG ROM on top of the FOM
	discretization leveraging the CN time scheme as previously described. Unless
		noted otherwise, we employ a constant time step size of $\Delta t =
		0.002$ for the LSPG ROM. We solve the nonlinear least-squares problem arising at each time instance
		via the Gauss--Newton method, and solve the linear least-squares problems
		arising at each Gauss--Newton iteration via the normal
		equations. We deem the Gauss--Newton iteration converged when the gradient norm is less than $10^{-4}$.  We compute all Jacobians via automatic
		differentiation~\cite{adolc}. 
\item \textit{\methodAcronymROM:} We consider \methodAcronymROMs\ solved via the
	direct and indirect methods with two different solution techniques:
\begin{itemize}
	\item \textit{Direct method}: We consider \methodAcronymROMs\ solved via the direct method for both the same CN discretization employed in the FOM and LSPG, as well as for the second-order explicit Adams Bashforth (AB2) discretization using a constant time step of $\Delta t = 0.0005$. We solve the nonlinear least-squares problem 
arising over each window with the Gauss--Newton method, and solve the linear
		least-squares problems arising at each Gauss--Newton iteration 
		via the normal equations. We again compute all Jacobians via automatic
		differentiation, and deem the Gauss--Newton algorithm converged when
		the gradient norm is less than $10^{-4}$ (i.e., we use the parameter $\epsilon = 10^{-4}$ in Algorithm~\ref{alg:colloc_gn}). Critically, we note that we assemble the (sparse) Jacobian 
matrix over a window by computing local (dense) Jacobians. We store the Jacobian matrix over a window in a compressed sparse row format. We employ uniform quadrature weights for evaluating the integral in~\eqref{eq:obj_gen_slab}. 

\item \textit{Indirect method}: We consider two \methodAcronymROMs\ solved via the indirect method. The first uses the same CN discretization at at time step of $\Delta t = 0.002$, while the second uses the AB2 discretization using a time step of $\Delta t = 0.0005$. We solve the coupled two-point boundary 
problem via the forward--backward sweep method, and compute the action of the Jacobian transpose on vectors via automatic differentiation. We use parameters $\epsilon = 10^{-6}$, $\fbsmGrowth = 1.1$, and $\fbsmDecay=2$ in Algorithm~\ref{alg:st_iter}. 
\end{itemize}
\end{itemize}

%

\subsubsection{Numerical results}
We first assess the impact of the window size on the performance of the \methodAcronymROMs.  We consider a set of \methodAcronymROMs\ that minimize the residual over windows of constant size 
$\DeltaSlabArg{n} \equiv \DeltaSlabArg{} =  .002$, $0.004$ ,$0.008$, $0.02$, $0.04$, $0.10$, $0.20$, $1.0$. We additionally consider the standard Galerkin and LSPG ROMs. We first show results for \methodAcronymROMs\ using the direct method with CN time discretization; a comparison of different time-marching methods and direct/indirect solution techniques will be provided later in this section. 
First, Figure~\ref{fig:sod_density} presents the density solutions produced by the various ROMs at $t = 0.5$ and $1.0$. Figure~\ref{fig:sod_xt} shows 
$\mathsf{x}-t$ diagrams for the same density solutions. From Figures~\ref{fig:sod_density} and~\ref{fig:sod_xt}, we observe that the LSPG and Galerkin ROMs accurately characterize 
the system: they correctly track the shock location, expansion waves, etc. We observe both predictions, however, to be highly oscillatory. These oscillations are 
not physical and can lead to numerical instabilities; e.g., due to negative pressure. We observe the \methodAcronymROMs\ to produce less oscillatory solutions than both 
the Galerkin and LSPG ROMs. Critically, we see that the solution becomes less oscillatory as the window size over which the residual is minimized grows. The solution displays no oscillations when the residual is minimized over the entire space--time domain. 

\begin{figure}
\begin{center}
\begin{subfigure}[t]{0.45\textwidth}
\includegraphics[width=1.\linewidth]{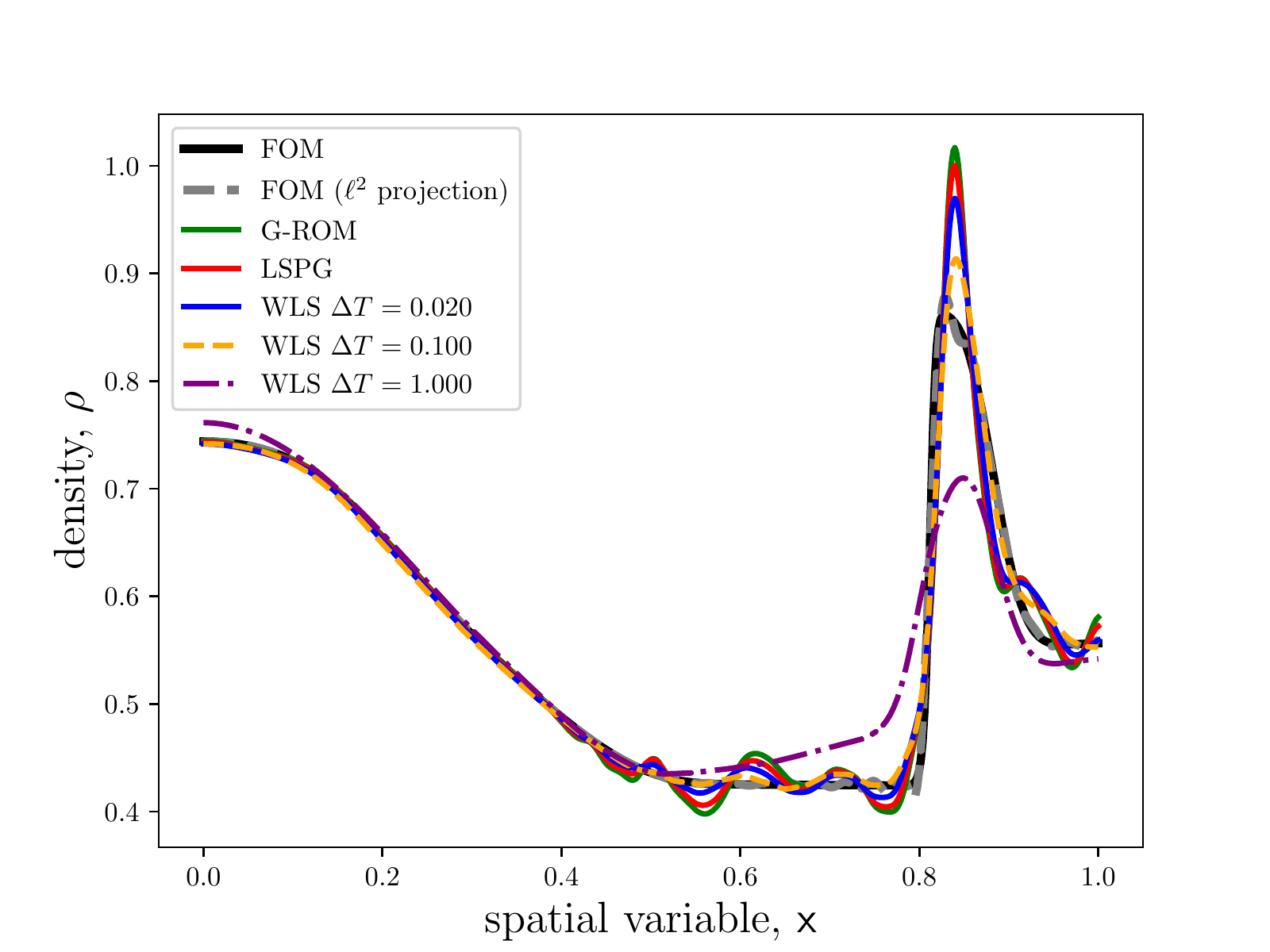}
\caption{$t=0.5$}
\end{subfigure}
\begin{subfigure}[t]{0.45\textwidth}
\includegraphics[width=1.\linewidth]{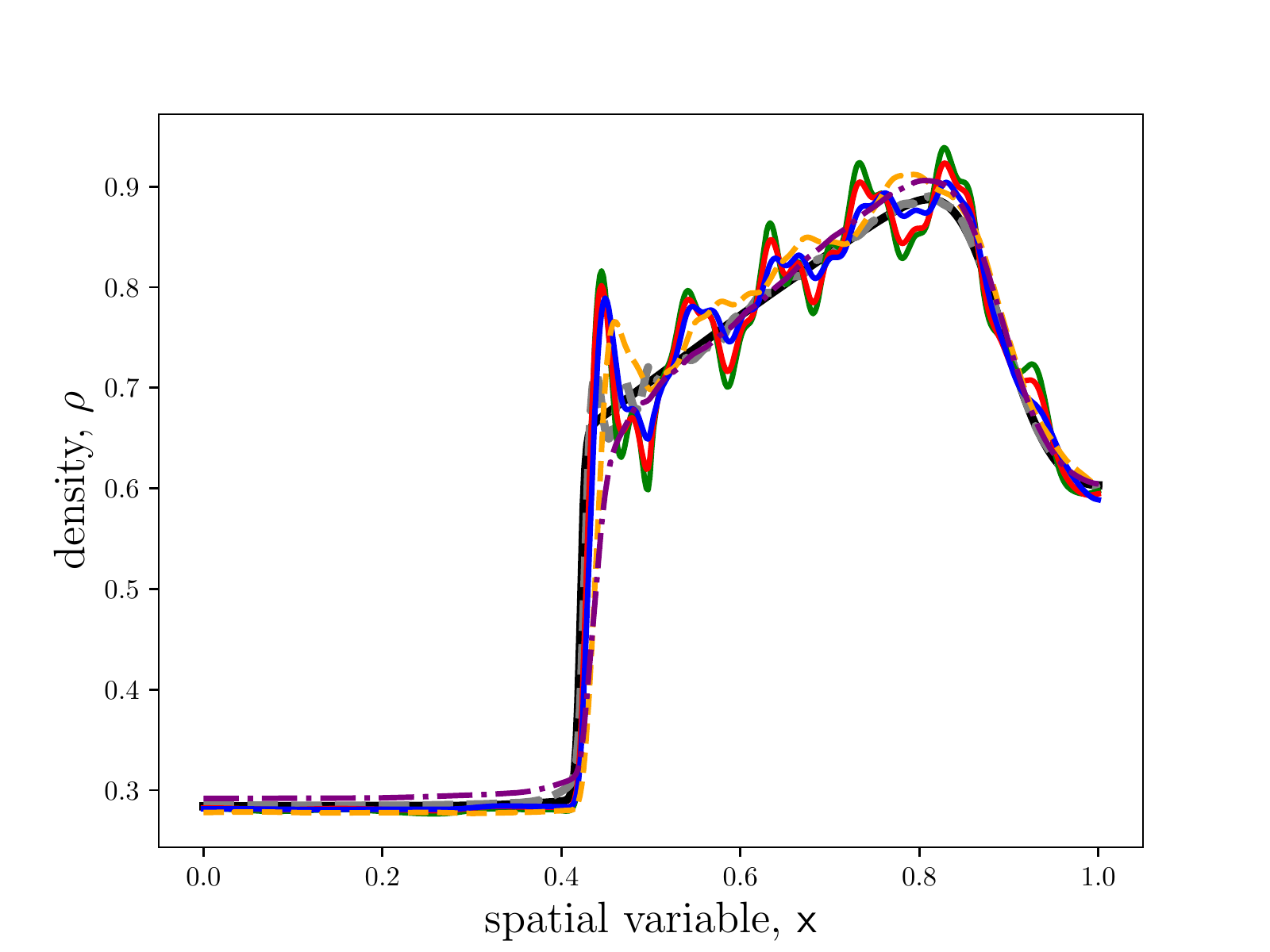}
\caption{$t=1$}
\end{subfigure}
\caption{Density profiles at various time instances.}
\label{fig:sod_density}
\end{center}
\end{figure}

\begin{figure}
\begin{center}
\begin{subfigure}[t]{0.48\textwidth}
\includegraphics[width=1.\linewidth]{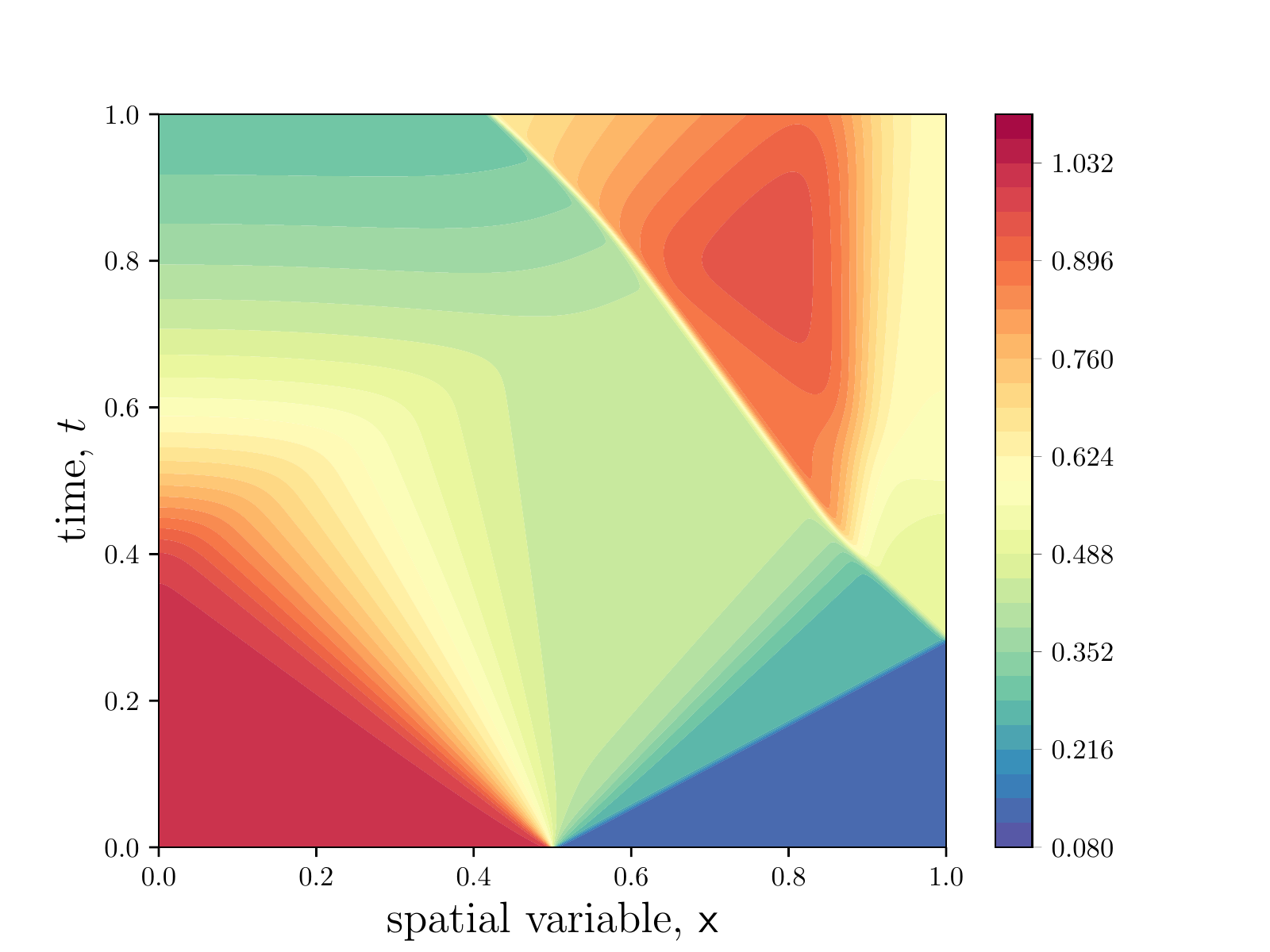}
\caption{Full-order model}
\end{subfigure}
\begin{subfigure}[t]{0.48\textwidth}
\includegraphics[width=1.\linewidth]{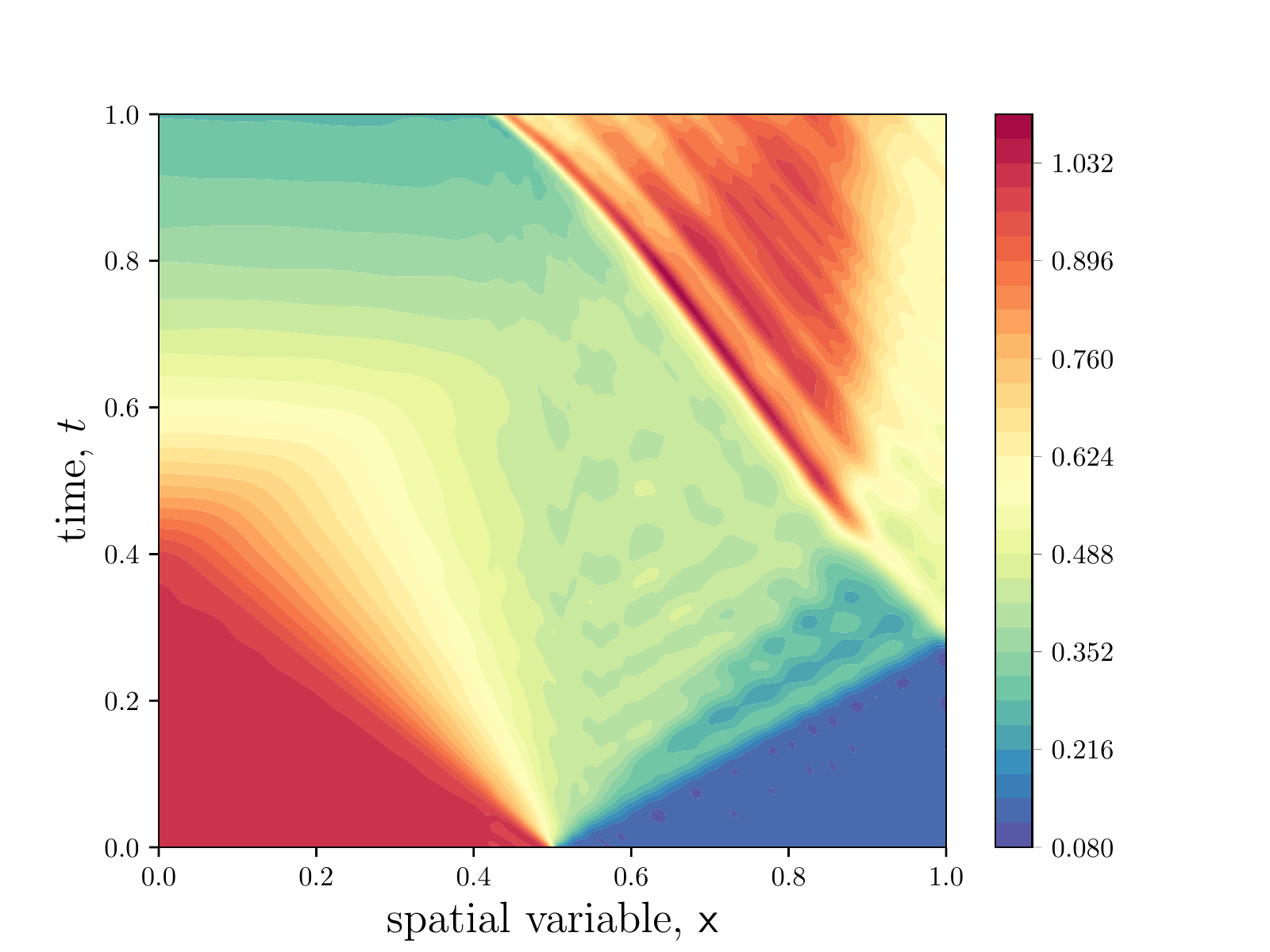}
\caption{G-ROM}
\end{subfigure}
\begin{subfigure}[t]{0.48\textwidth}
\includegraphics[width=1.\linewidth]{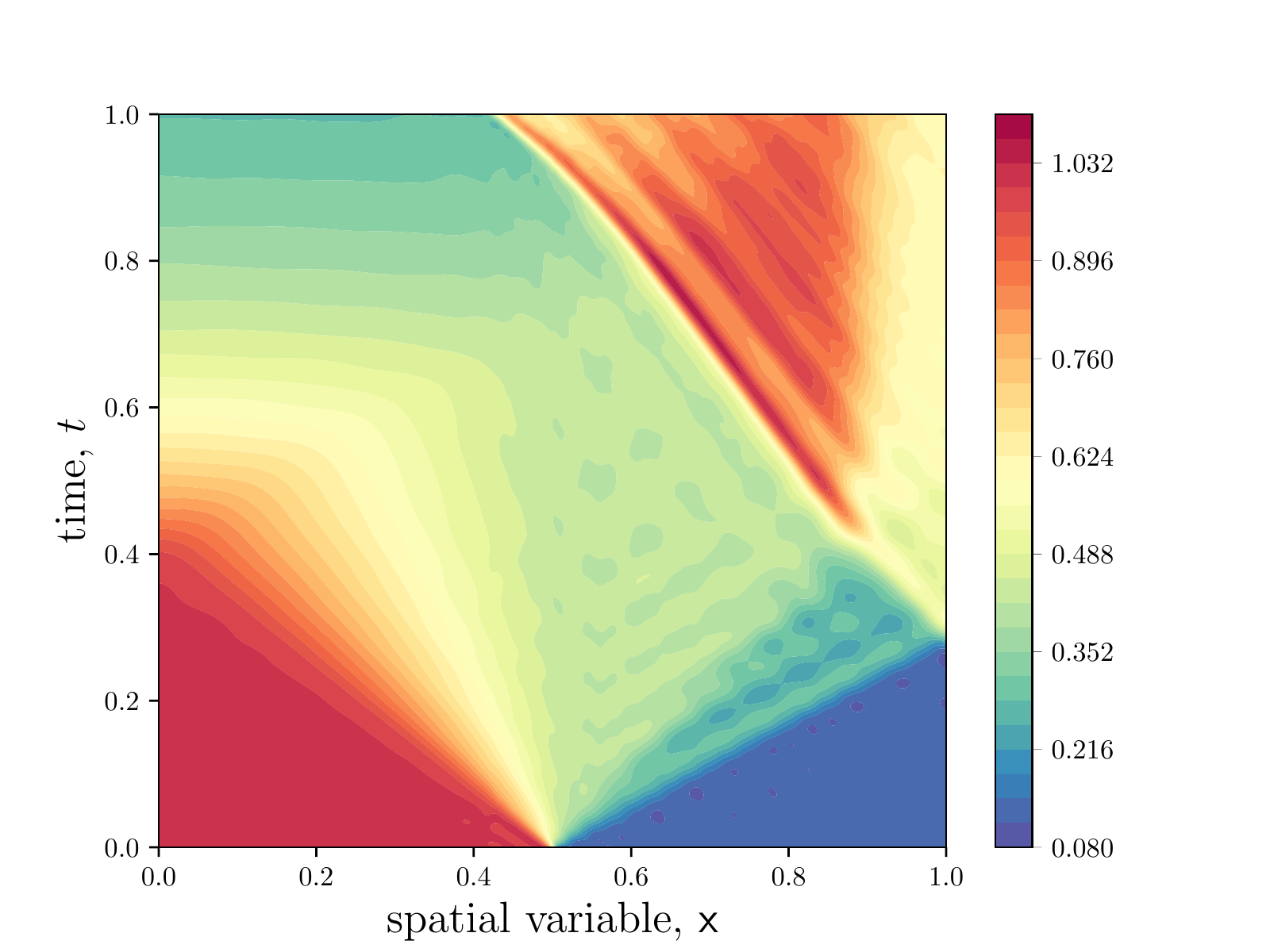}
\caption{LSPG}
\end{subfigure}
\begin{subfigure}[t]{0.48\textwidth}
\includegraphics[width=1.\linewidth]{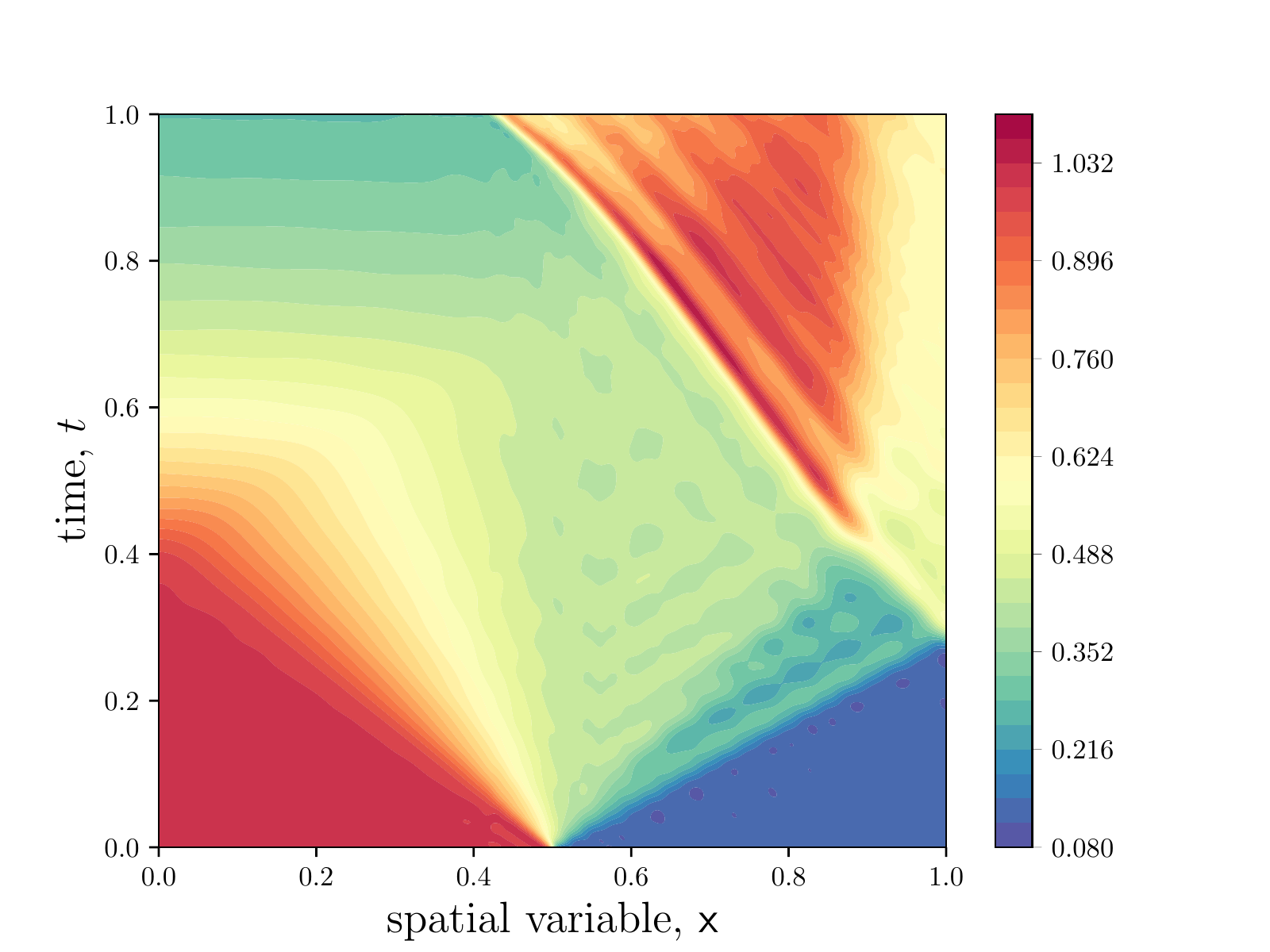}
\caption{\methodAcronym: $\DeltaSlabArg{}= 0.002$}
\end{subfigure}
\begin{subfigure}[t]{0.48\textwidth}
\includegraphics[width=1.\linewidth]{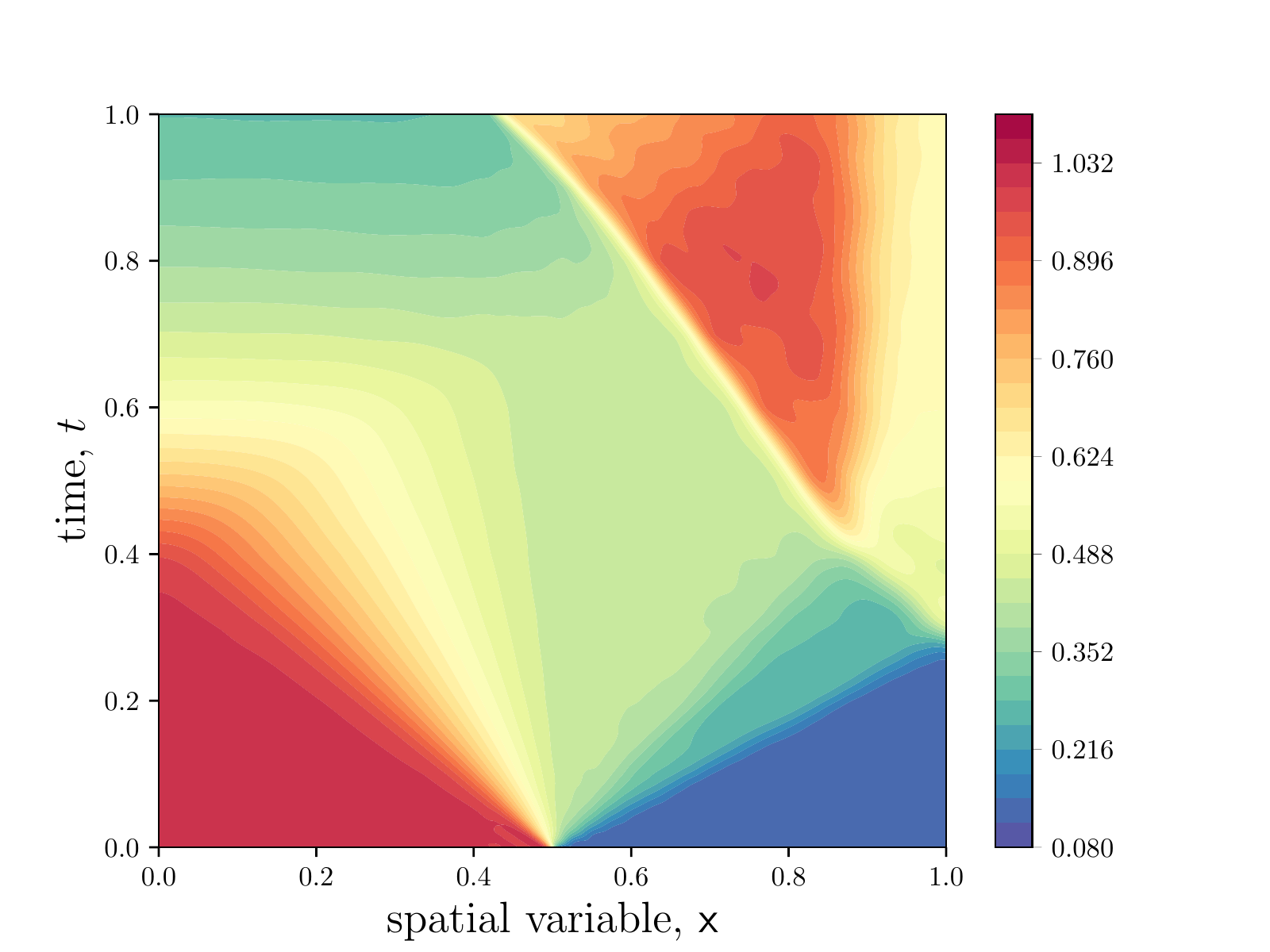}
\caption{\methodAcronym: $\DeltaSlabArg{} = 0.1$}
\end{subfigure}
\begin{subfigure}[t]{0.48\textwidth}
\includegraphics[width=1.\linewidth]{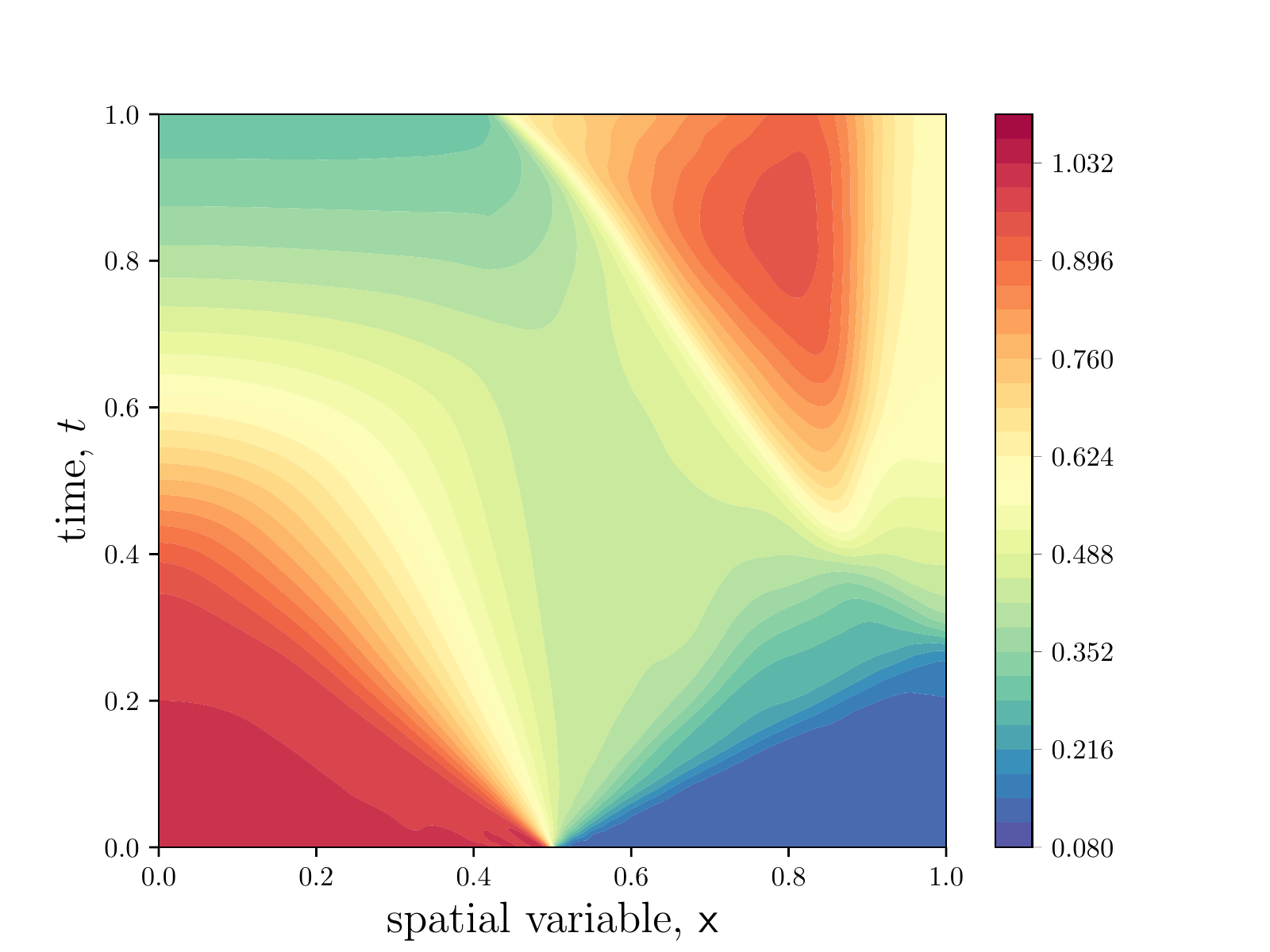}
\caption{\methodAcronym: $\DeltaSlabArg{} = 1.0$}
\end{subfigure}
\caption{$\mathsf{x}-t$ diagrams for the density fields as predicted by the FOM, G-ROM,
	LSPG-ROM, and \methodAcronymROMs.}
\label{fig:sod_xt}
\end{center}
\end{figure}

Next, Figure~\ref{fig:sod_error} shows space--time state errors and the objective function~\eqref{eq:obj} (i.e., the space--time residual norm) over $t \in [0,1]$ for the various ROMs. Most notably, we observe 
that increasing the window size over which the residual is minimized does \textit{not} lead to a monotonic decrease in the space--time error as measured in the $\elltwo$-norm (we refer to this as the $\elltwo$-error). As expected, increasing the window size does lead to a monotonic decrease in the space--time residual, however. We additionally note that, although the space--time
$\elltwo$-error of 
the projected FOM solution is significantly lower than that of the various ROM solutions, the space--time residual norm of the projected FOM solution is \textit{higher} 
than all ROMs. Thus, although the projected FOM solution is more accurate in the $\elltwo$-norm, it does not satisfy the governing equations as well the ROM solutions. 

\begin{figure}
\begin{center}
\begin{subfigure}[t]{0.45\textwidth}
\includegraphics[width=1.\linewidth]{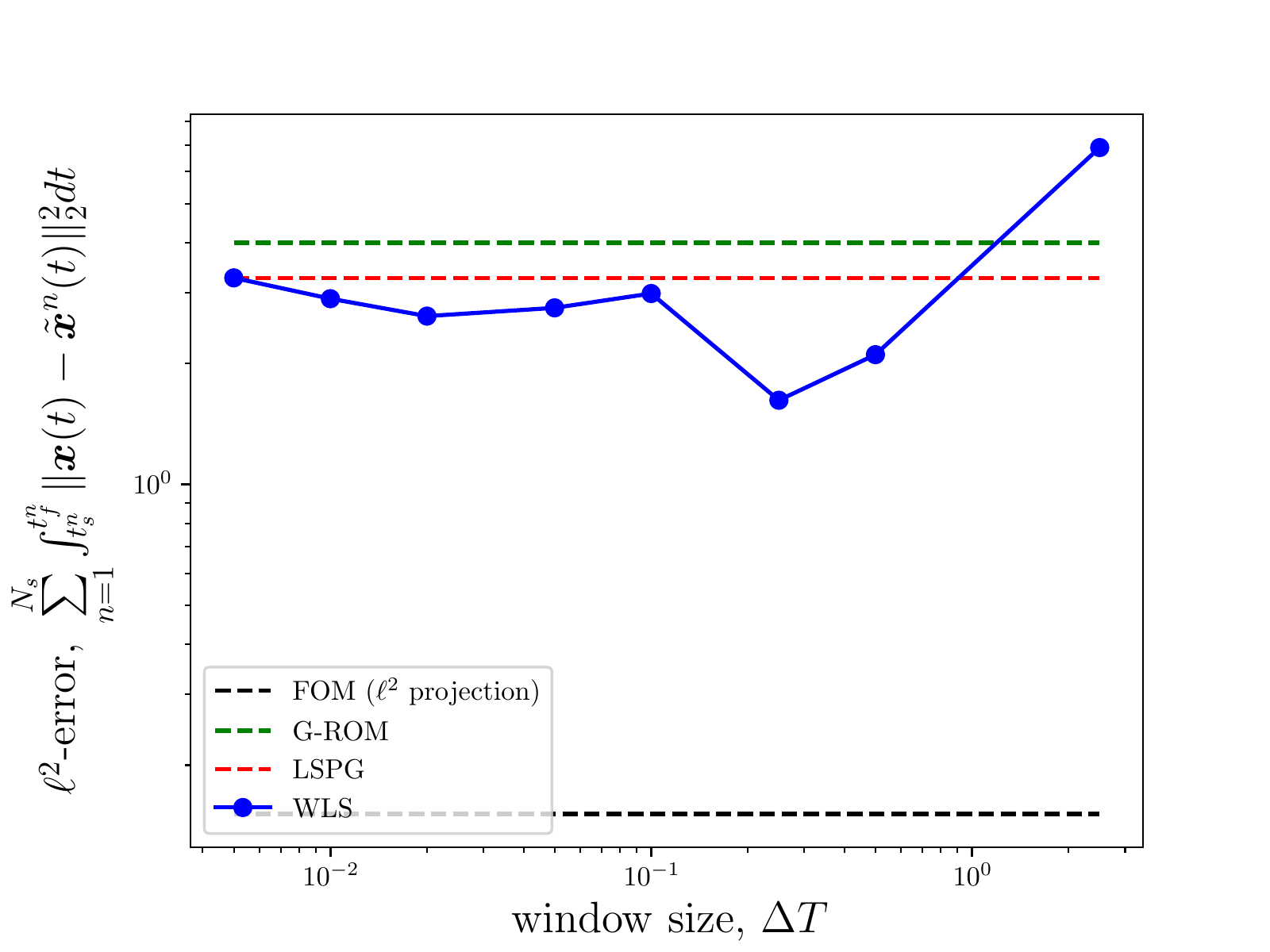}
\caption{Space--time $\elltwo$-error of various ROMs.}
\label{fig:sod_error_a}
\end{subfigure}
\begin{subfigure}[t]{0.45\textwidth}
\includegraphics[width=1.\linewidth]{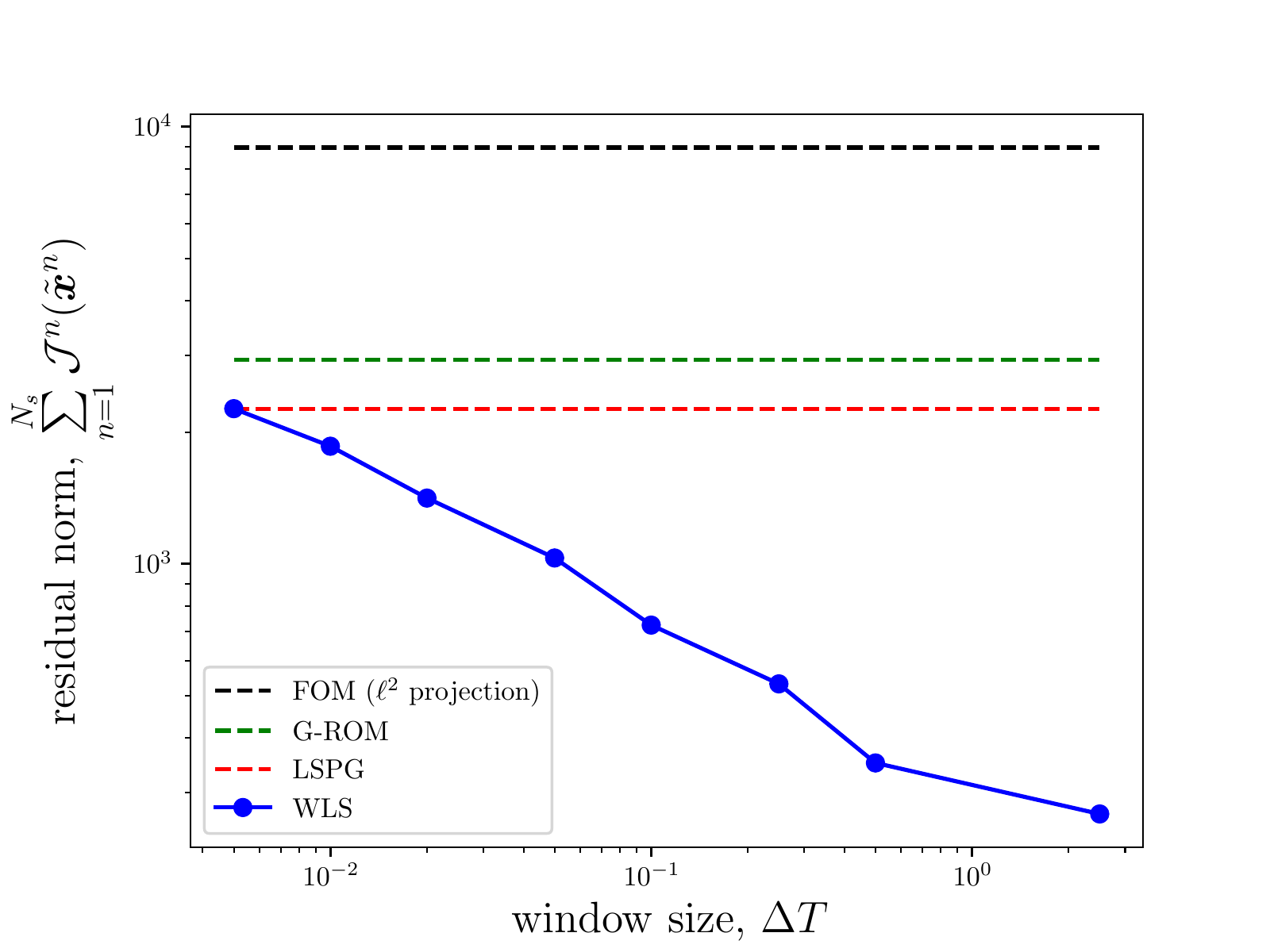}
\caption{Integrated residual norm of various ROMs.} 
\label{fig:sod_error_b}
\end{subfigure}
\caption{Integrated performance metrics of the Galerkin, LSPG, and \methodAcronymROMs.  Note that the Galerkin and LSPG ROMs do not depend on the window size.} 
\label{fig:sod_error}
\end{center}
\end{figure}

We now examine the comparative performance of the direct and indirect solution techniques for the \methodAcronymROMs\ using various time discretization techniques. 
Figure~\ref{fig:sod_error_methods} 
shows the same space--time $\elltwo$-errors and residual norms as in Figure~\ref{fig:sod_error}, but this time results are shown for 
the various \methodAcronymROMs. In both Figures~\ref{fig:sod_error_methods_a} and~\ref{fig:sod_error_methods_b}, we observe that the \methodAcronym\ method 
is relatively insensitive to the solution technique (direct vs indirect) and underlying discretization scheme, although some minor differences are observed.
In particular, ROMs using the second order explicit AB2 scheme with a time step of $\Delta t = 0.0005$ provide similar results to the 
ROMs using the second-order CN scheme at a time step of $\Delta t = 0.002$. All methods display similar dependence on the window size: the optimal 
$\elltwo$-error occurs when the window size is $\DeltaSlabArg{} = 0.1$, and the residual decreases monotonically as the window size grows. 
\begin{figure}
\begin{center}
\begin{subfigure}[t]{0.45\textwidth}
\includegraphics[width=1.\linewidth]{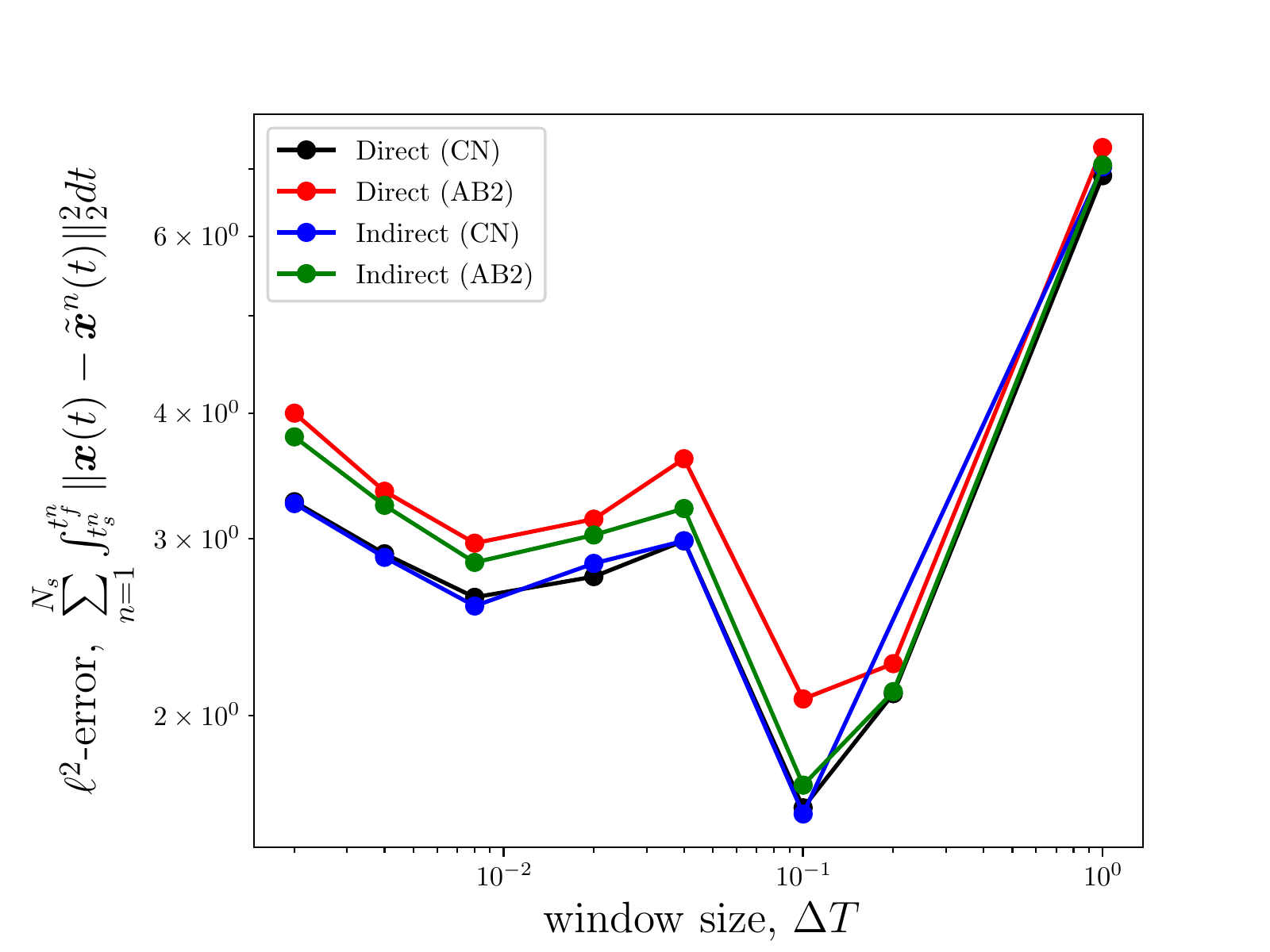}
\caption{Integrated $\elltwo$-error of the \methodAcronymROMs.}
\label{fig:sod_error_methods_a}
\end{subfigure}
\begin{subfigure}[t]{0.45\textwidth}
\includegraphics[width=1.\linewidth]{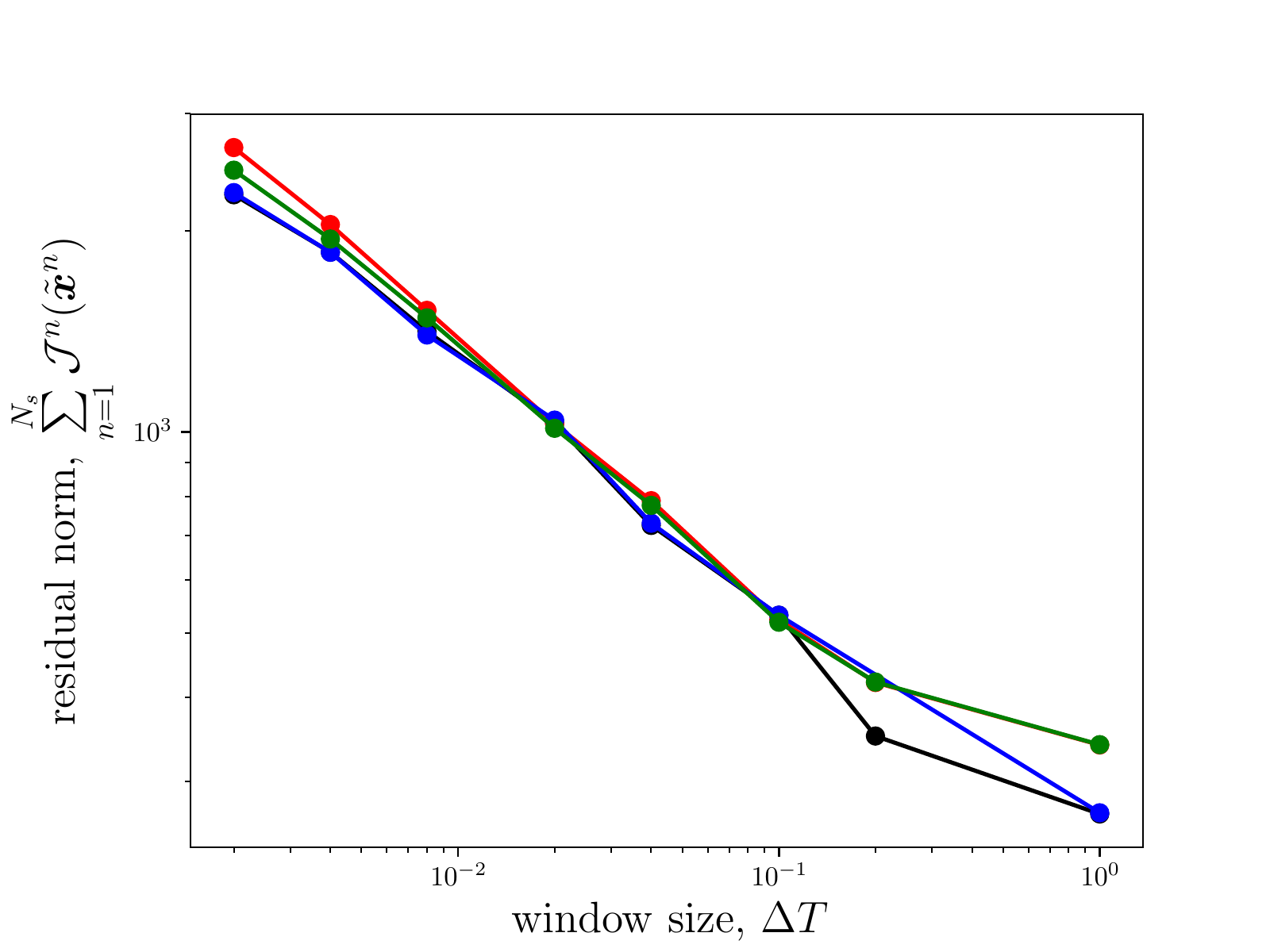}
\caption{Integrated residual norm of the \methodAcronymROMs.} 
\label{fig:sod_error_methods_b}
\end{subfigure}
\caption{Integrated performance metrics of various \methodAcronymROMs.} 
\label{fig:sod_error_methods}
\end{center}
\end{figure}

Next, Figure~\ref{fig:sod_walltime} provides the CPU wall-clock times for the various \methodAcronymROMs. We observe that the computational cost of all methods grows
as the window size is increased. For indirect methods, this increase in cost is due to the fact that, as the window size 
grows, more iterations of the FBSM are required for convergence. For direct methods, the increase in cost is due to (1) the cost 
associated with forming and solving the normal equations at each Gauss--Newton iteration and (2) the increased number of 
Gauss--Newton iterations required for convergence. We observe the cost of the FBSM method to increase more rapidly 
than the direct methods. Encouragingly, \methodAcronymROMs\ based on the CN discretization
minimizing the full space--time residual (comprising 
500 time instances) only  
cost approximately one order of magnitude more than the case where the residual is minimized over a single time step (i.e., LSPG). Also of interest is the 
fact that the direct method utilizing the AB2 time-discretization scheme costs only slightly more per a given window size than the direct method 
utilizing the CN time discretization. This is despite the fact that AB2 is evolved at a time step of $\Delta t = 0.0005$, while CN uses $\Delta t = 0.002$; thus 
AB2 contains 4 times more temporal degrees of freedom than CN.  Finally, we emphasize that the 
results presented here are for standard algorithms (e.g., Gauss--Newton and the FBSM). As mentioned in 
Remarks~\ref{remark:gaussnewton} and~\ref{remark:fbsm}, we expect that the computational cost 
of both indirect and direct methods can be decreased through the use and/or development of 
algorithms tailored to the windowed minimization problem.

\begin{figure}
\begin{center}
\begin{subfigure}[t]{0.45\textwidth}
\includegraphics[width=1.\linewidth]{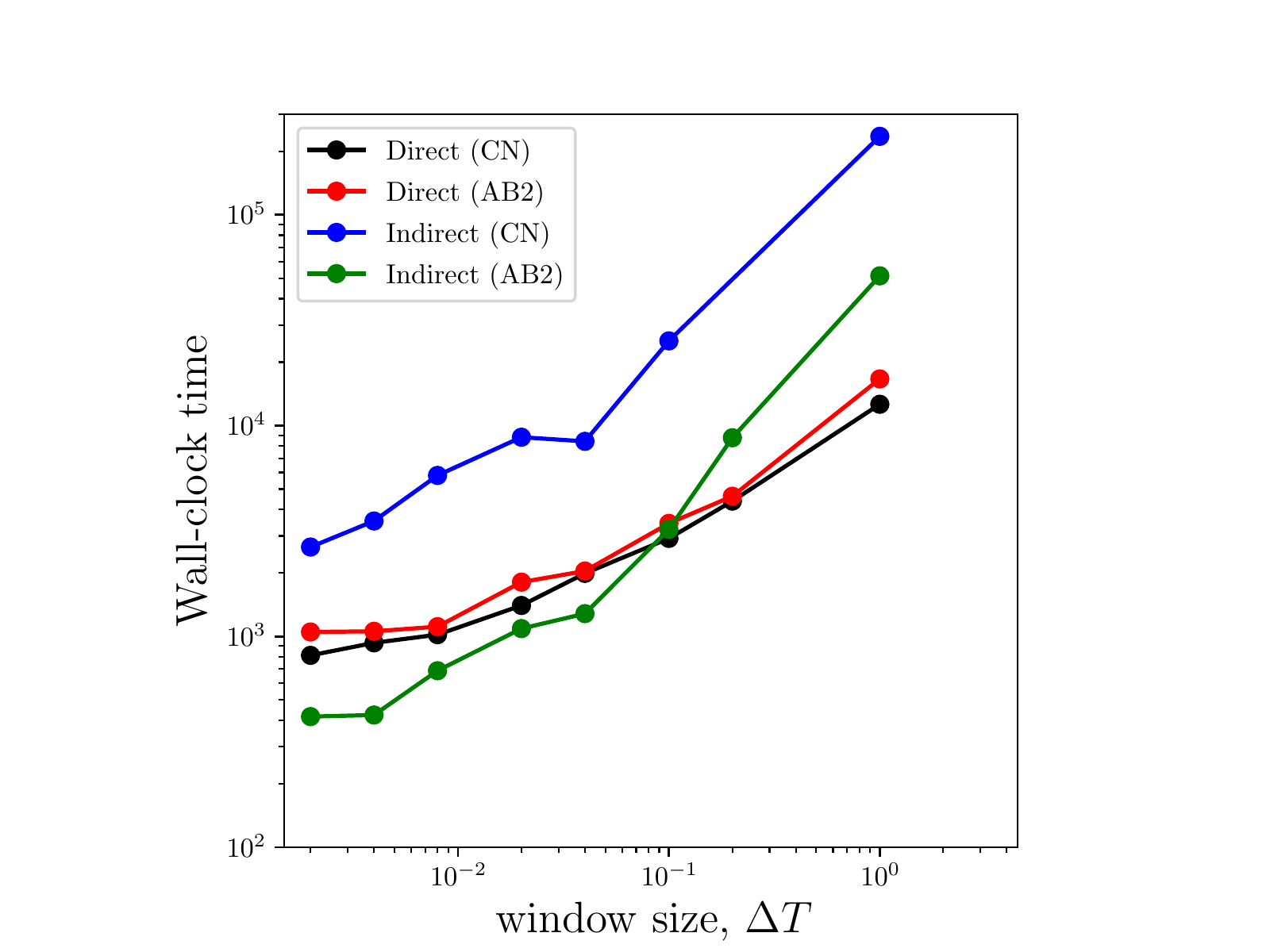}
\label{fig:sod_error_a}
\end{subfigure}
\caption{Comparison of wall-clock times of the direct and indirect \methodAcronymROMs\ as a function of window size.} 
\label{fig:sod_walltime}
\end{center}
\end{figure}

Finally, we study the impact of the time step on the \methodAcronymROM\ results. We examine \methodAcronymROMs\ that use a
window size of $\Delta T = 0.1$, with time steps $\Delta t =$  $0.001$, $0.002$, $0.005$, $0.01$, $0.05$, $0.1$ (i.e.,  
$100$, $50$, $20$, $10$, $2$, and $1$ time instances per window). We additionally 
consider LSPG ROMs leveraging the same set of time steps. To assess time step convergence, we compare results to a new full-order model, which is as described in Section~\ref{sec:sod_fom} but uses a fine time step of $\Delta t = 10^{-4}$. Figure~\ref{fig:time_step_study} shows the $\elltwo$-error and residual norm  
of the various ROMs. We observe that the $\elltwo$-error and residual norm 
of the \methodAcronymROMs\ decrease and converge as the time step decreases. This is in direct contrast to the LSPG ROMs, in where the $\elltwo$-error and residual norm display a complex dependence on the time step. This result 
demonstrates that \methodAcronym\ overcomes the time-step dependence inherit to the LSPG approach. Lastly, Figure~\ref{fig:walltime_dtvar}
shows the relative wall-clock times of the \methodAcronymROMs\ with respect to the LSPG ROMs. It is seen that for all time steps considered, the \methodAcronymROMs\ 
are less than 6x the cost of LSPG; this is despite the window sizes comprising up to 100 time instances. 
\begin{figure}
\begin{center}
\begin{subfigure}[t]{0.45\textwidth}
\includegraphics[width=1.\linewidth]{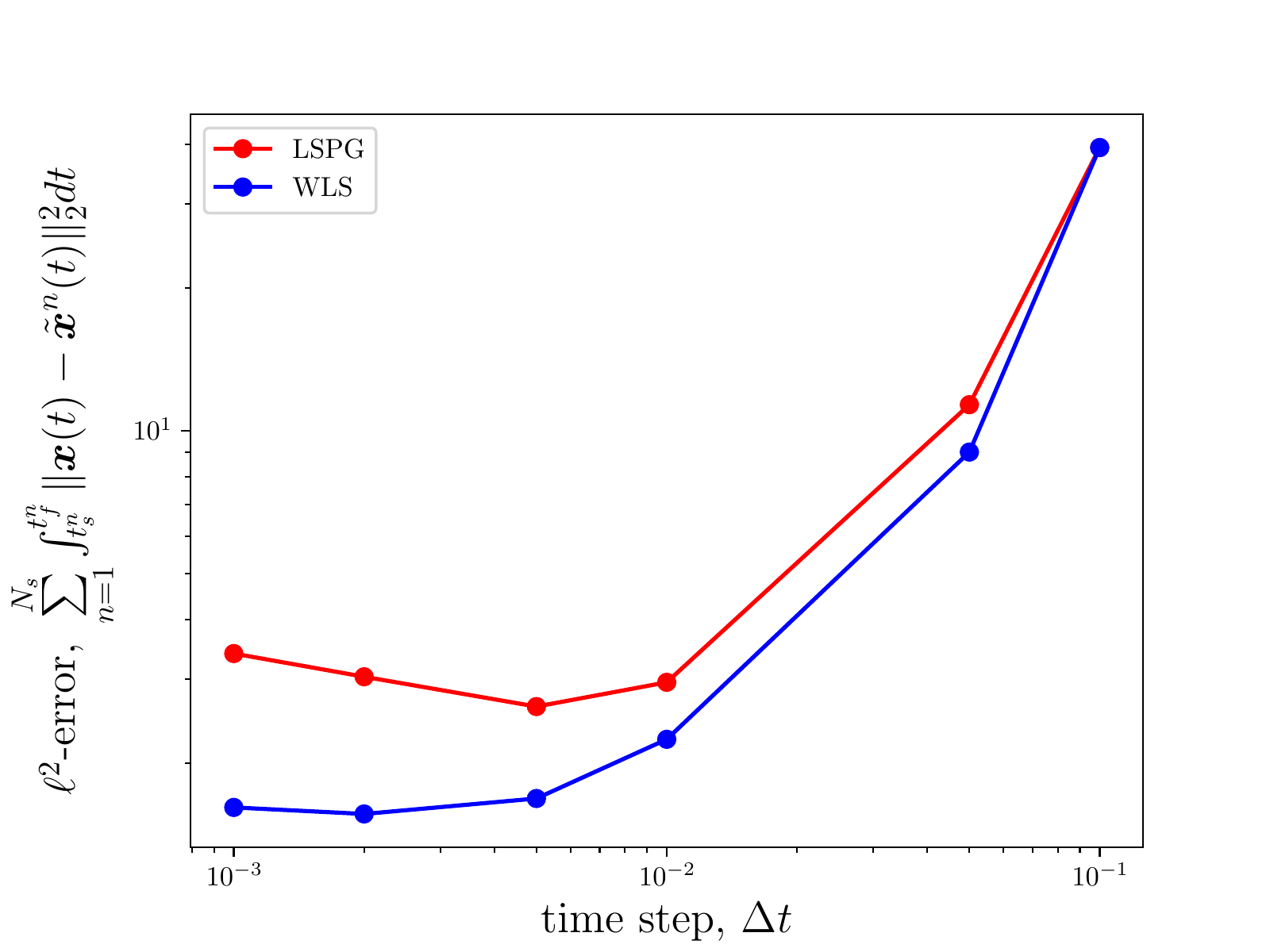}
\caption{Space--time $\elltwo$-error of the \methodAcronymROM.}
\label{fig:sod_error_methods_a}
\end{subfigure}
\begin{subfigure}[t]{0.45\textwidth}
\includegraphics[width=1.\linewidth]{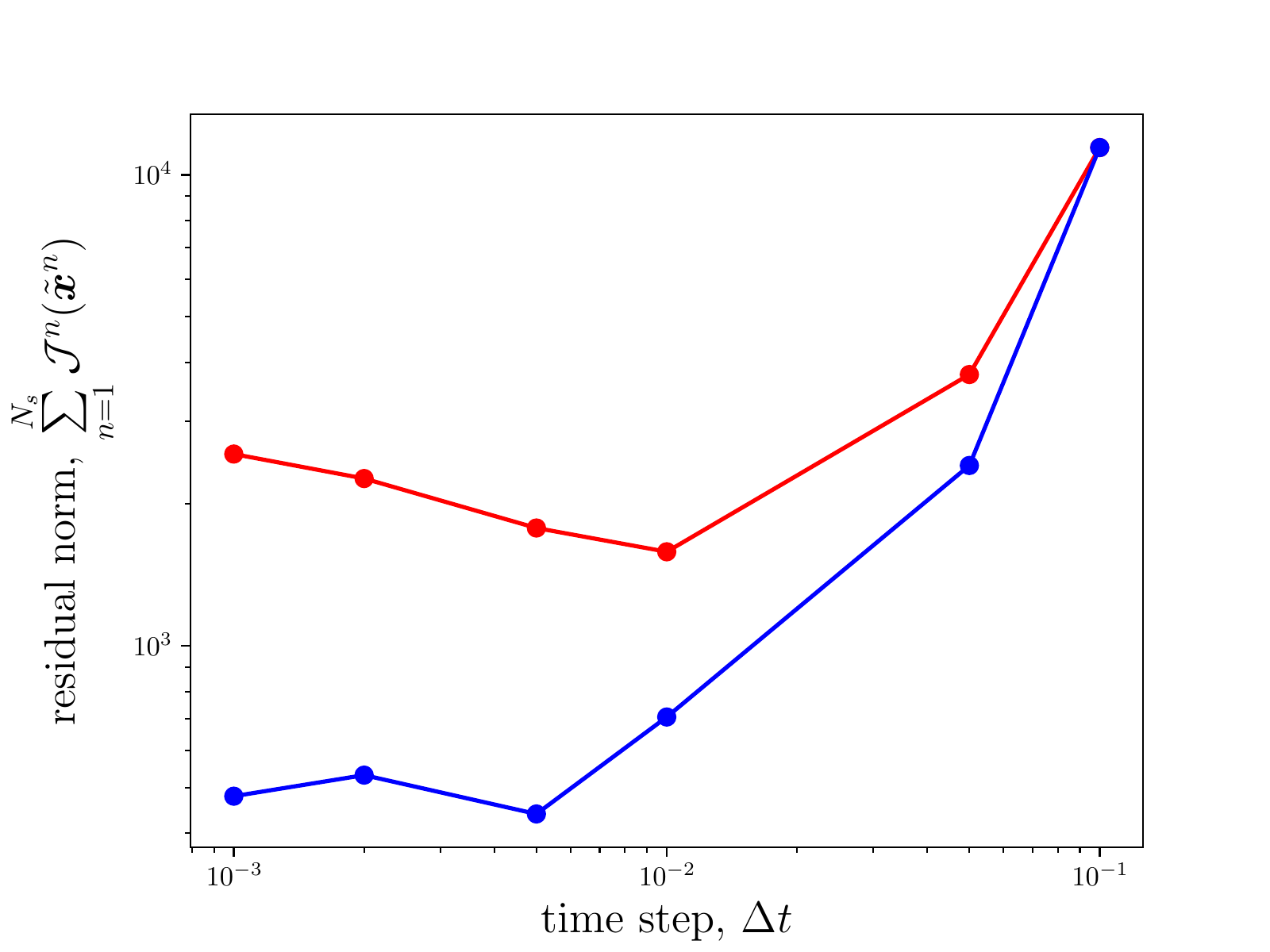}
\caption{Space--time residual norm of the \methodAcronymROM.} 
\label{fig:sod_error_methods_b}
\end{subfigure}
\caption{Performance metrics of the Galerkin, LSPG, and \methodAcronymROMs.} 
\label{fig:time_step_study}
\end{center}
\end{figure}

\begin{figure}
\begin{center}
\begin{subfigure}[t]{0.45\textwidth}
\includegraphics[width=1.\linewidth]{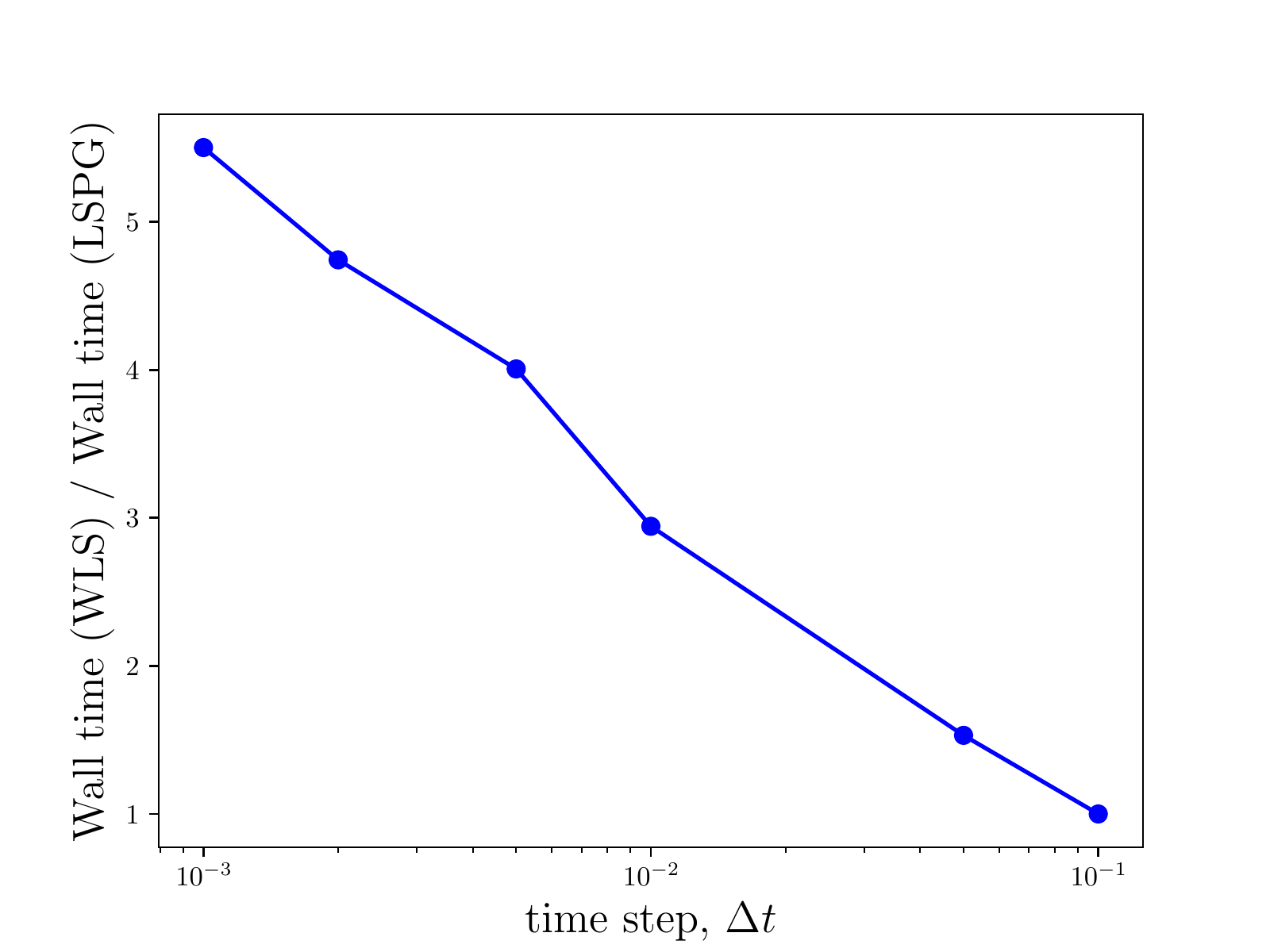}
\label{fig:sod_error_a}
\end{subfigure}
\caption{Comparison of wall-clock times of the \methodAcronymROMs\ to LSPG ROMs as a function of the time step. WLS ROMs have a fixed window size of $\Delta T = 0.1$.} 
\label{fig:walltime_dtvar}
\end{center}
\end{figure}

\subsubsection{Summary of numerical results for the Sod shock tube}
The key observations from the results of the first numerical example are: 
\begin{enumerate}
\item Increasing the window size over which the residual is minimized led to more physically relevant solutions. Specifically, we observed that as the window size over which the residual was minimized grew, \methodAcronym\ led to less oscillatory solutions.
\item Increasing the window size over which the residual is minimized does not necessary lead to a lower space--time error as measured by the $\elltwo$-norm. We observed that minimizing the residual over an intermediary window size led to the lowest space--time error in the $\elltwo$-norm. 
\item \methodAcronym\ displays time-step convergence: both the $\elltwo$-error and residual norm decreased and displayed time-step convergence as the time step decreased. This is in contrast to LSPG.
\item For all examples considered, in where the windows comprised up to 2000 time instances, \methodAcronym\ with the direct method is between 1x and 10x the cost of the LSPG. The 
direct method was observed to be slightly more efficient and robust than the indirect method.
\end{enumerate} 

\input{cavity}

%% file: cavity.tex
\subsection{Cavity flow}
The next numerical example considers collocated ROMs\footnote{Collocation is a form of hyper-reduction which requires sampling the full-order model at only select grid points.} of a viscous, compressible flow in a two-dimensional cavity. 
The flow is described by the two-dimensional compressible Navier-Stokes equations,
\begin{equation}\label{eq:compressible_ns}
\frac{\partial \mathbf{u}}{\partial t} + \nabla \cdot \big( \mathbf{F}(\mathbf{u} ) - \mathbf{F}_v (\mathbf{u},\nabla \mathbf{u} )      \big) =\bz,
\end{equation}
where $\mathbf{u}: [0,T] \times \Omega \rightarrow \RR{4}$ comprise the density, $\mathsf{x}_1$ and $\mathsf{x}_2$ momentum, and total energy. The terms $\mathbf{F}$ and $ \mathbf{F}_v$ are the inviscid and viscous fluxes, respectively. For a two-dimensional flow, the state vector and inviscid fluxes are
$$
\mathbf{u} = \begin{Bmatrix}
\rho \\ \rho u_1 \\ \rho u_2 \\ \rho E \end{Bmatrix}, \qquad \mathbf{F}_{1} = \begin{Bmatrix} \rho u_1 \\ \rho u_1^2 +      p \\ \rho u_1 u_2 \\ u_1(E + p) \end{Bmatrix}, 
\qquad \mathbf{F}_{2} = \begin{Bmatrix} \rho u_2 \\ \rho u_1 u_2  \\ \rho u_2^2 + p \\ u_2(E + p) \end{Bmatrix}.
$$
The viscous fluxes are given by
$$
\qquad \mathbf{F}_{v_1} = \begin{Bmatrix} 0 \\ \tau_{11} \\ \tau_{12}  \\ u_j \tau_{j1} + c_p \frac{\mu}{\text{Pr}} \frac{\partial T}{\partial x_1}  \end{Bmatrix}, 
\qquad \mathbf{F}_{v_2} = \begin{Bmatrix} 0 \\ \tau_{21} \\ \tau_{22}  \\ u_j \tau_{j2} + c_p \frac{\mu}{\text{Pr}} \frac{\partial T}{\partial x_2}  \end{Bmatrix},
$$
where $\mu \in \RRplus$ is the dynamic viscosity, $\text{Pr} = 0.72$ is the Prandtl number, $T : \Omega \rightarrow \RRplus$ the temperature, and $c_p \in \RRplus$ the heat capacity ratio.
We assume a Newtonian fluid, which leads to a viscous stress tensor of the form
\begin{equation*}
\tau_{ij} = 2\mu S_{ij},
\end{equation*}
where
\begin{equation*}
 S_{ij} = \frac{1}{2} \big( \frac{\partial u_i}{\partial \mathsf{x}_j} + \frac{\partial u_j}{\partial \mathsf{x}_i} \big) - \frac{1}{3} \frac{\partial      u_k}{\partial \mathsf{x}_i} \delta_{ij}.
\end{equation*}
We close the Navier--Stokes equations with a constitutive relationship for a calorically perfect gas,
$$p = (\gamma - 1)( \rho E - \frac{1}{2} \rho u_1^2 - \frac{1}{2} \rho u_2^2 \big),$$
where $\gamma = 1.4$ is the heat-capacity ratio.

Figure~\ref{fig:cav_fig} depicts the domain $\Omega$ and flow conditions. The Reynolds number is defined as $\text{Re} = \rho_{\infty} \norm{\mathbf{v}_{\infty}} L / \mu$ with a characteristic length set at $L=1$ and $\mathbf{v} = [u_1 \; \; u_2]^T$, the speed of sound is defined by $a_{\infty} = \sqrt{\gamma p_{\infty}/\rho_{\infty}}$, and $\infty$ subscripts refer to free-stream conditions. We employ free-stream boundary conditions at the inlet, outlet, and top wall of the cavity. We enforce no-slip boundary conditions 
on the bottom wall of the cavity. 

\subsubsection{Description of FOM and generation of \spatialAcronym\ trial subspace}
The full-order model comprises a discontinuous-Galerkin discretization. We obtain the discretization 
by partitioning the domain into $100$ elements in the flow direction and $40$ elements 
in the wall-normal direction. The discretization represents the solution to third order over each element using tensor product polynomials of order $p=2$, 
resulting in $36000$ unknowns for each conserved variable. Spatial discretization via the discontinuous Galerkin method yields a dynamical system 
of the form
$$\frac{d \state}{dt} = \mass_{\text{DG}}^{-1} \velocity_{\text{DG}} (\state),$$
where $\mass_{\text{DG}} \in \RR{N \times N}$ is the (block diagonal) DG mass matrix and $\velocity_{\text{DG}}: \RR{N} \rightarrow \RR{N}$ is the DG velocity operator containing 
surface and volume integrals. By the definition of the FOM~\eqref{eq:FOM}, the dynamical system velocity is $\velocity = \mass_{\text{DG}}^{-1} \velocity_{\text{DG}}$. 
Figure~\ref{fig:cav_mesh} shows the computational mesh. The 
DG method uses the Rusanov flux at the cell interfaces~\cite{rusanov} and uses the first form of Bassi and Rebay~\cite{br1} for the viscous fluxes. Time integration 
is performed via a third-order strong stability preserving Runge-Kutta method~\cite{ssp_rk3} with a time step of $\Delta t = 0.001$. 
 
The reduced-order models leverage POD to construct the \spatialAcronym\ trial basis and use q-sampling~\cite{qdeim_drmac} based on snapshots of the FOM velocity to select the sampling points (and as a result the weighting 
matrix $\stweightingMatArg{}$); we use a constant weighting matrix across all windows, e.g., $\stweightingMatArg{n} \equiv \stweightingMatArg{}$, $n=1,\ldots,\nslabs$. The process used to construct the initial conditions, trial subspaces, and weighting matrices for the ROMs is as follows:
\begin{enumerate}
\item Initialize the FOM with uniform free-stream conditions.
\item Evolve the FOM for $t \in [0,400]$.
\item Reset the time coordinate, $0 \leftarrow t$, and execute Algorithm~\ref{alg:pod} with $\stateInterceptArg{} = \bz, N_{\text{skip}} = 100$ and $K = \{136,193\}$ over $t \in [0,100]$ to construct two trial subspaces comprising $\approx 95\%$ and $\approx 97\%$ of the snapshot energy, respectively. 
\item Execute Algorithm~\ref{alg:qdeim} with $N_{\text{skip}} = 100$, $n_s = 129$ to obtain the sampling point matrix of dimension $\stweightingMatOneArg{} \in \{0,1\}^{1603 \times \fomdim }$. 
\end{enumerate}
Figure~\ref{fig:cav_sampmesh} shows the resulting sample mesh used in the ROM simulations. To depict the nature of the flow, Figure~\ref{fig:fom_sols_cav} shows snapshots of the vorticity field generated by the FOM for several time instances used in training.  
\begin{table}[]
\begin{centering}
\begin{tabular}{c c c c}
\hline
Basis \# & Trial Basis Dimension ($K$) &  Energy Criterion & Sample Points ($n_S$) \\
\hline
1    & $136$ & $95.0\%$ & $1603$ \\
2    & $193$ & $97.0\%$ & $1603$ \\
\hline
\end{tabular}
\caption{Summary of the various basis sizes employed in the cavity flow example.}
\label{tab:rom_basis_details}
\end{centering}
\end{table}

\subsubsection{Description of reduced-order models}
We consider collocated ROMs based on the Galerkin, least-squares Petrov--Galerkin, and \methodAcronym\ approaches. Details on the implementation of the methods is as follows:
\begin{itemize}
\item \textit{Galerkin ROM with collocation}: We consider a Galerkin ROM with collocation and evolve the ROM in time with the CN time scheme at a time step of $\Delta t =0.1$.

\item \textit{LSPG ROM with collocation}: We consider a collocated LSPG ROM, which is built on top of the FOM discretization using the CN scheme for temporal 
discretization. We employ a time step of $\Delta t  = 0.1$. The implementation is the same as previously described. 
 
\item \textit{\methodAcronymROMs\ with collocation}: We consider \methodAcronymROMs\ solved via the direct method. The ROMs use the CN time discretization with a time step of 
$\Delta t = 0.1$. The implementation is the same as previously described. 
\end{itemize}

\begin{figure}
\begin{center}
\includegraphics[trim={2cm 7cm 4cm 6cm},clip,width=0.95\linewidth]{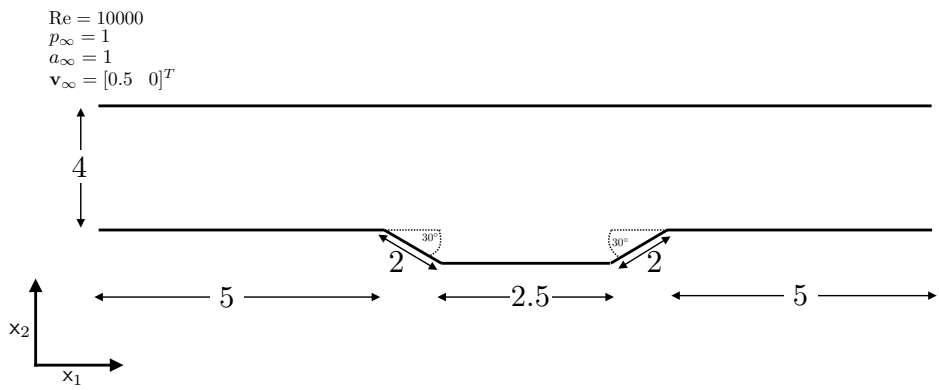}
\caption{Figure depicting geometry and flow conditions of the cavity flow problem.} 
\label{fig:cav_fig}
\end{center}
\end{figure}

\begin{figure}
\begin{center}
\includegraphics[trim={0cm 14cm 0cm 14cm},clip,width=1.\linewidth]{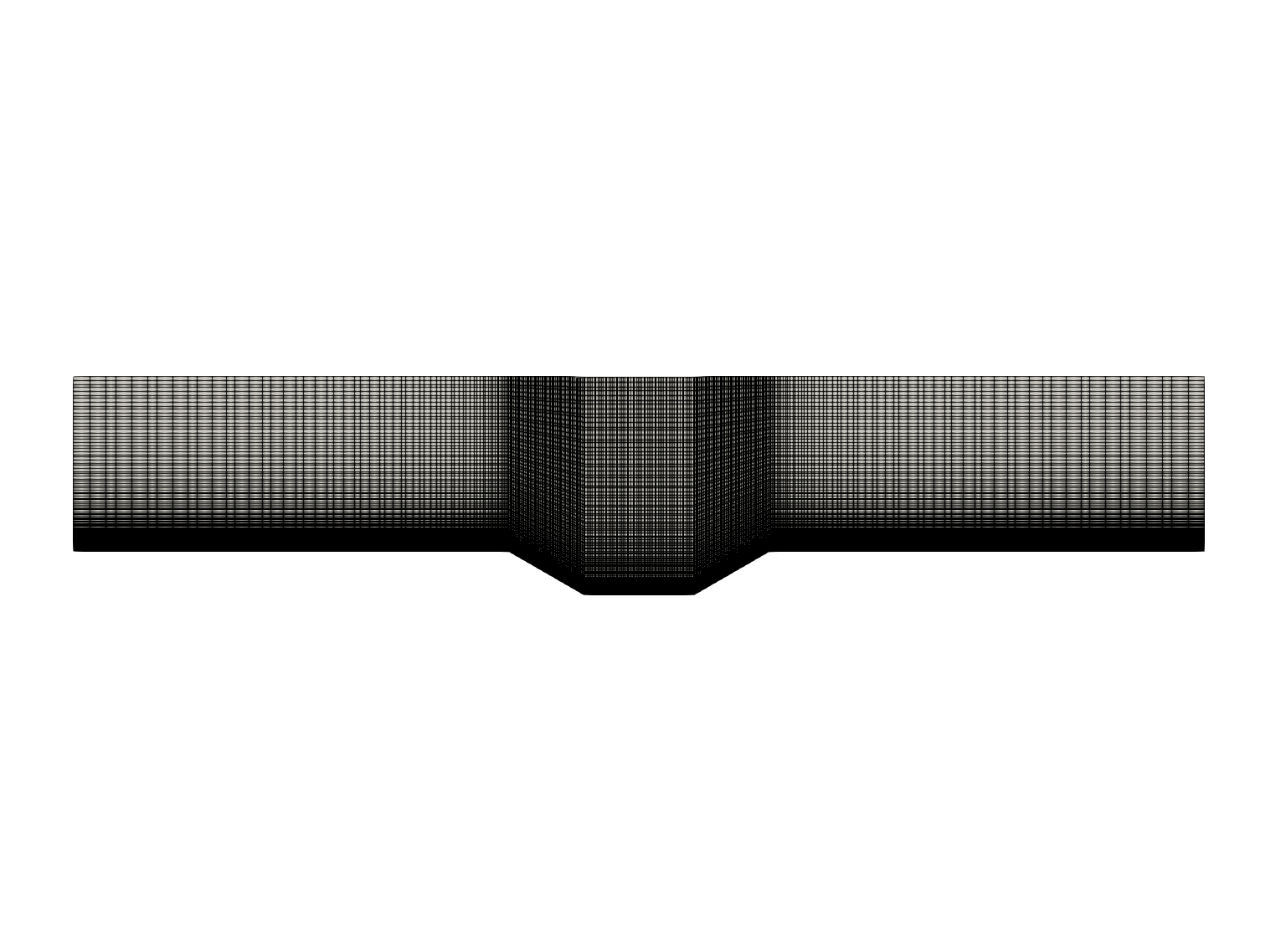}
\caption{Computational mesh employed in cavity flow simulations.}
\label{fig:cav_mesh}
\end{center}
\end{figure}

\begin{figure} 
\begin{center}
\includegraphics[trim={0cm 0cm 0cm 0cm},clip,width=0.65\linewidth]{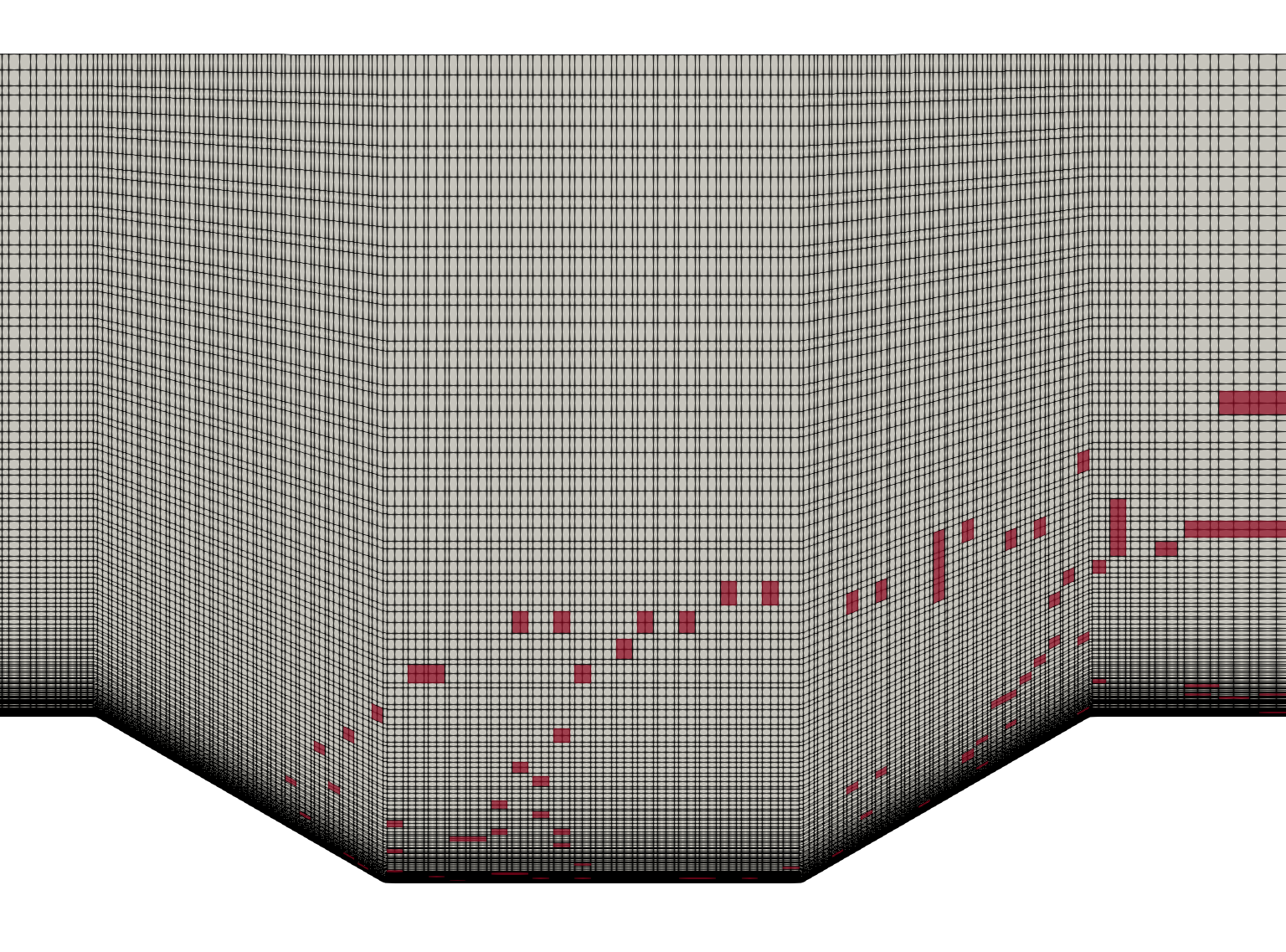}
\caption{Close up of computational mesh with highlighted collocation cells.} 
\label{fig:cav_sampmesh}
\end{center}
\end{figure}

\begin{figure}
\begin{center}
\begin{subfigure}[t]{0.49\textwidth}
\includegraphics[trim={18cm 16cm 18cm 15cm},clip,width=1.0\linewidth]{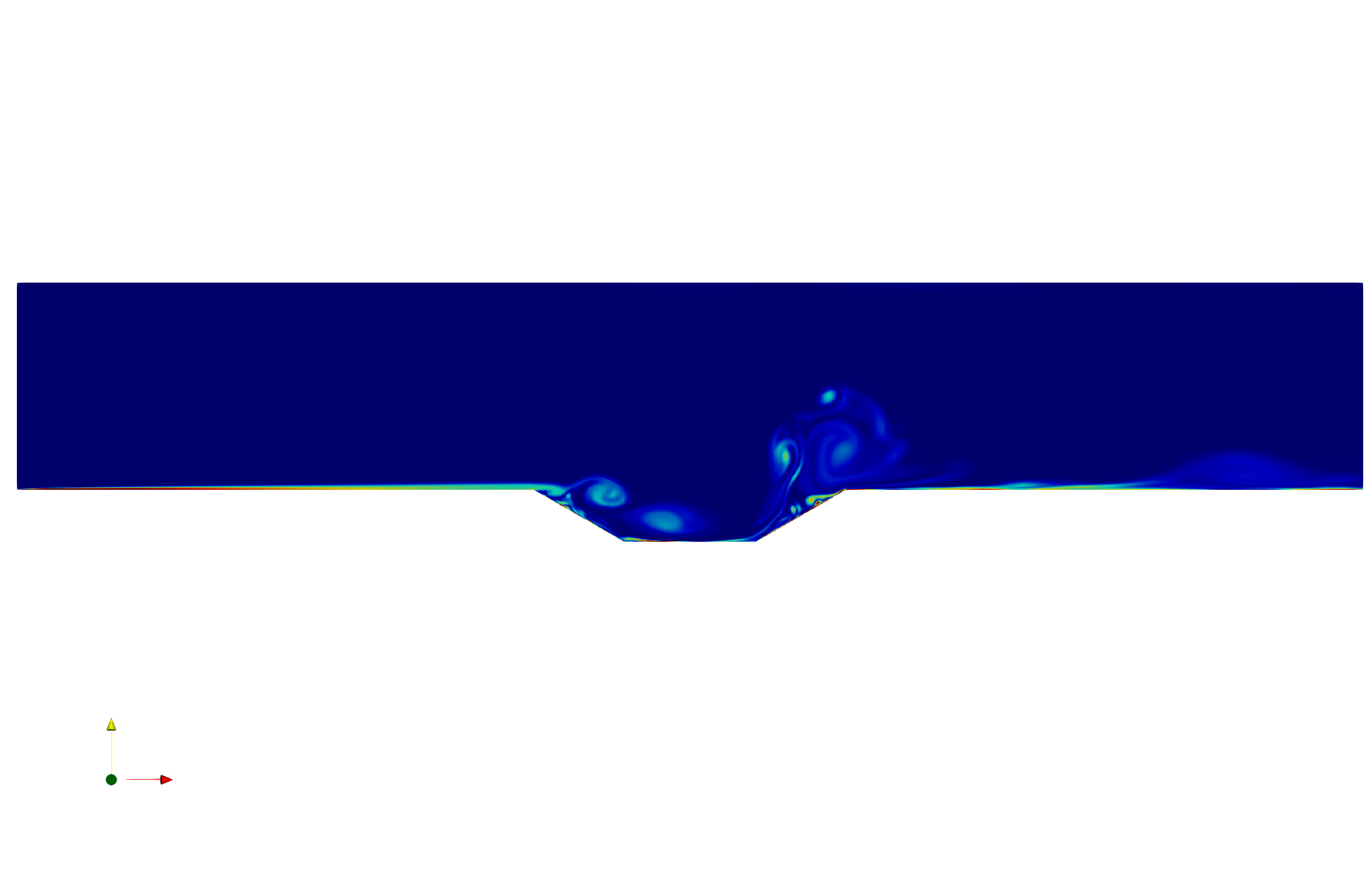}
\caption{$t=0.0$}
\end{subfigure}
\begin{subfigure}[t]{0.49\textwidth}
\includegraphics[trim={18cm 16cm 18cm 15cm},clip,width=1.0\linewidth]{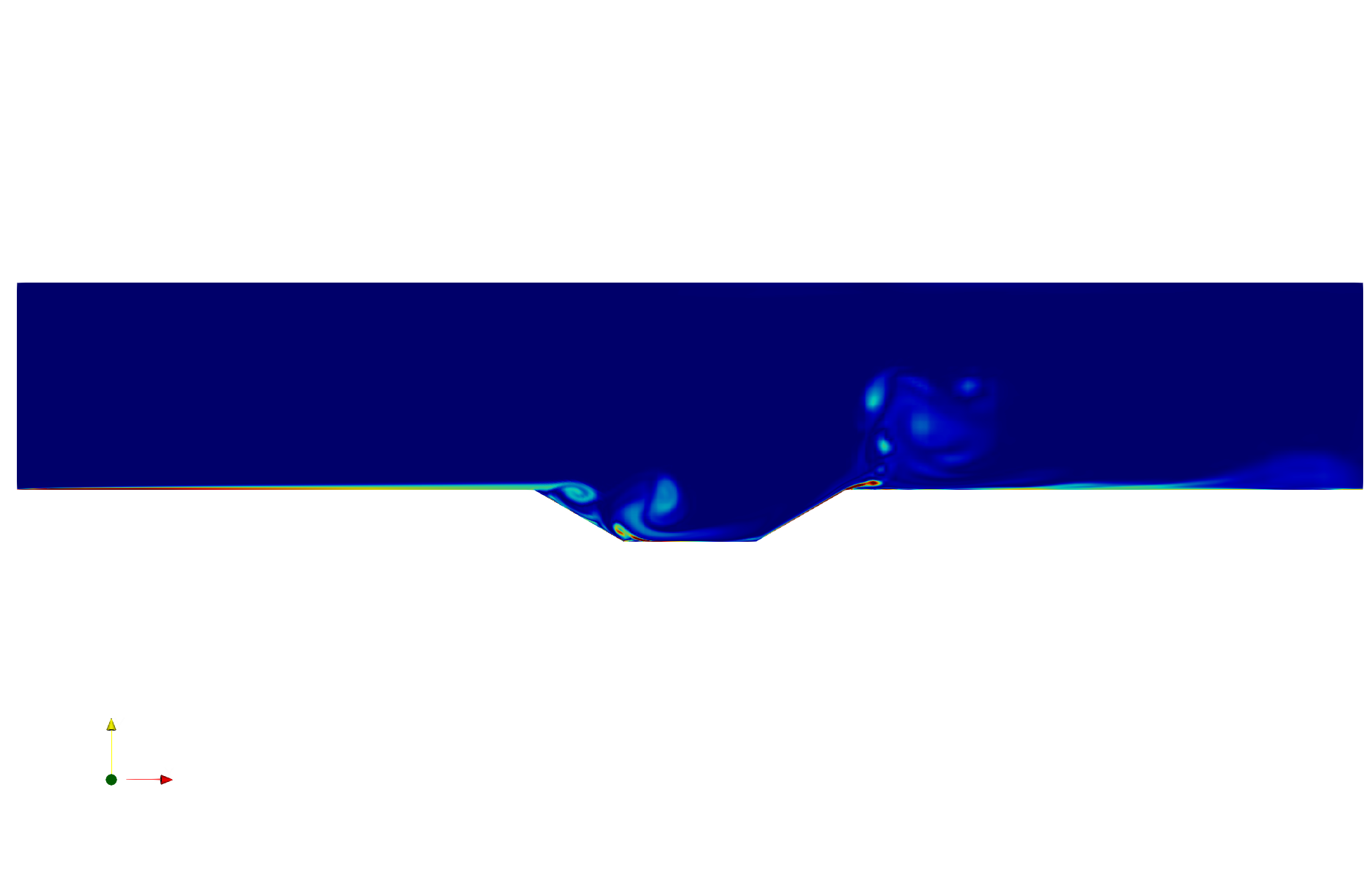}
\caption{$t=4.0$}
\end{subfigure}
\begin{subfigure}[t]{0.49\textwidth}
\includegraphics[trim={18cm 16cm 18cm 15cm},clip,width=1.0\linewidth]{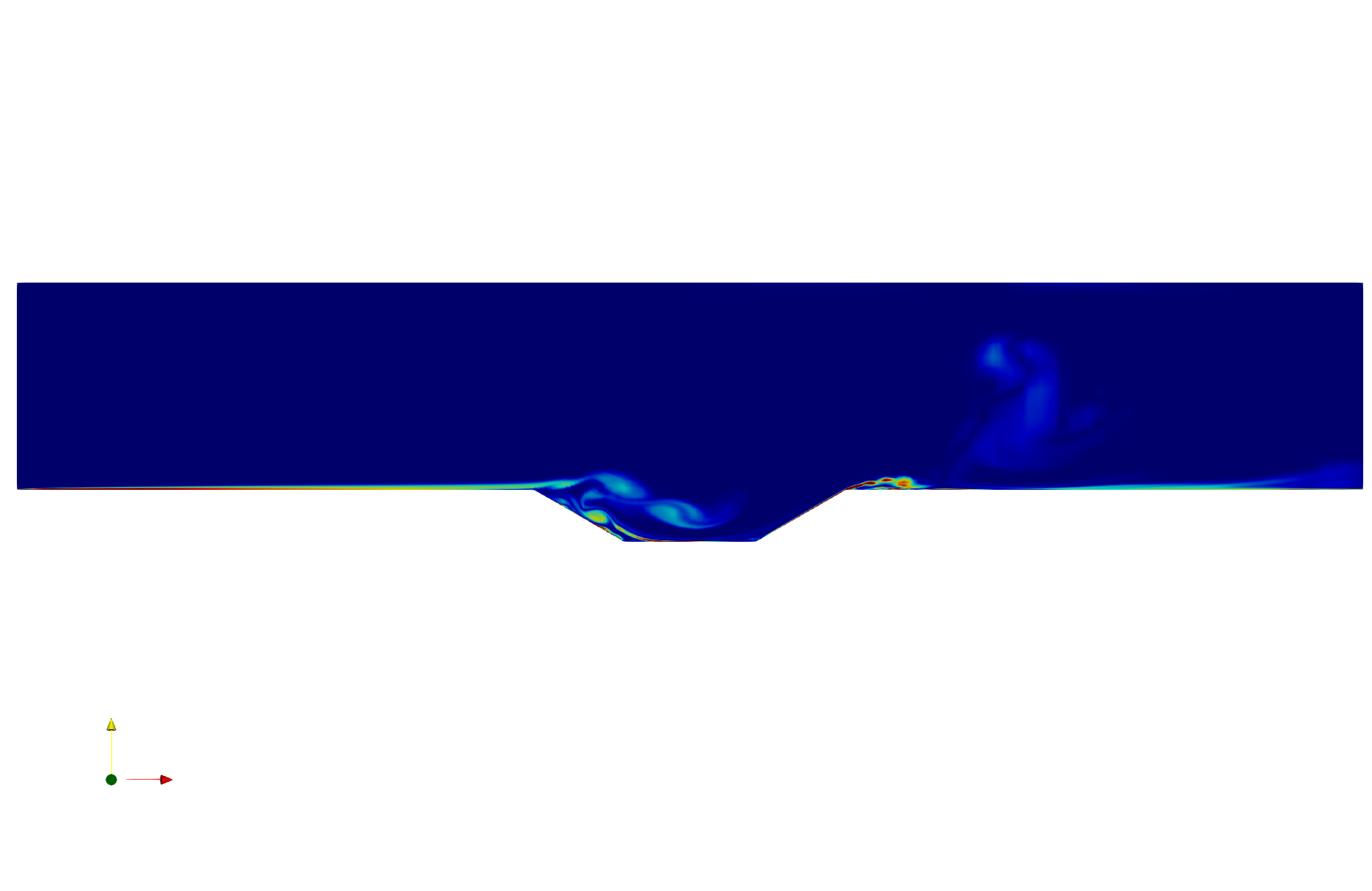}
\caption{$t=8.0$}
\end{subfigure}
\begin{subfigure}[t]{0.49\textwidth}
\includegraphics[trim={18cm 16cm 18cm 15cm},clip,width=1.0\linewidth]{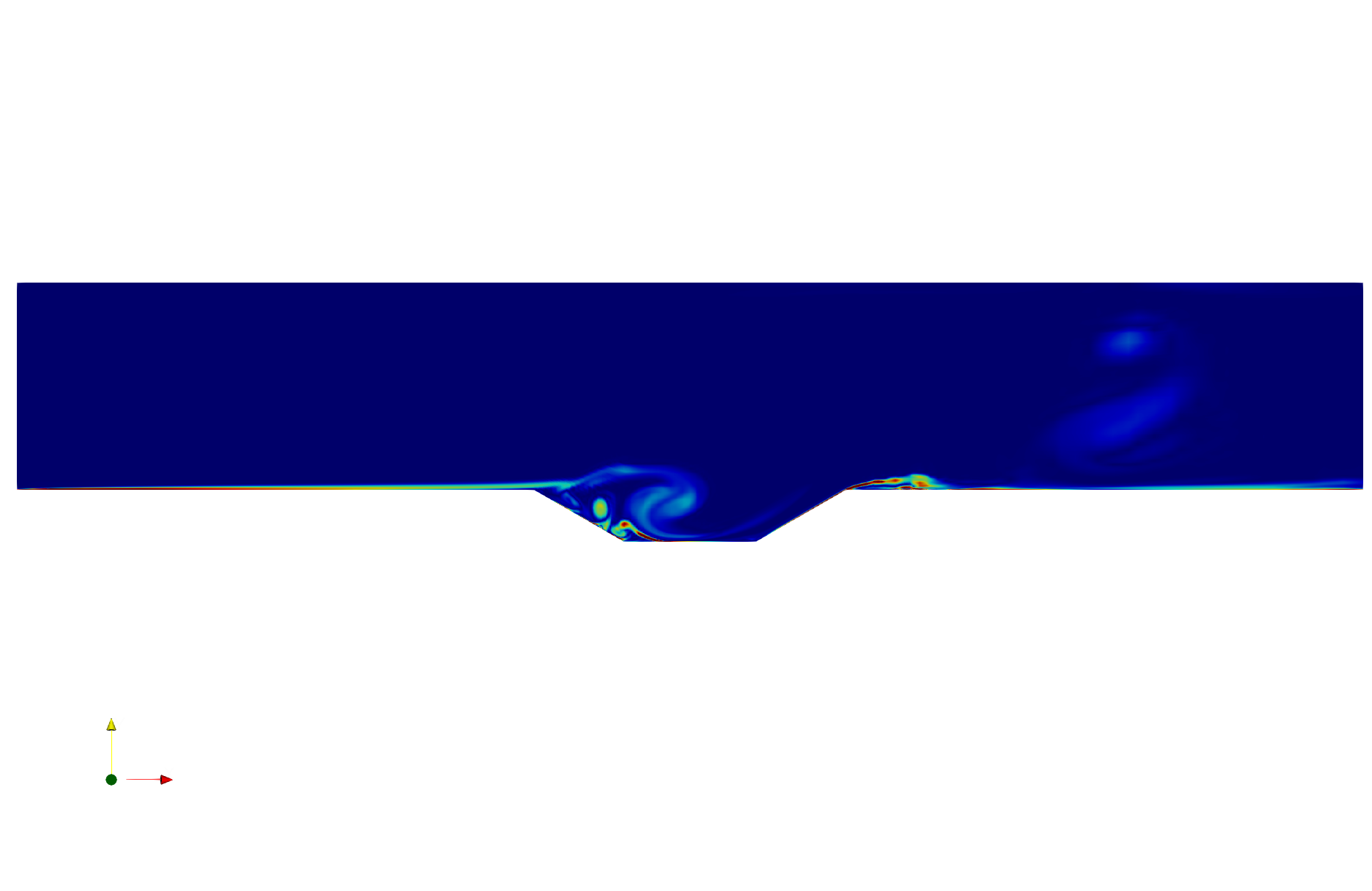}
\caption{$t=12.0$}
\end{subfigure}
\begin{subfigure}[t]{0.49\textwidth}
\includegraphics[trim={18cm 16cm 18cm 15cm},clip,width=1.0\linewidth]{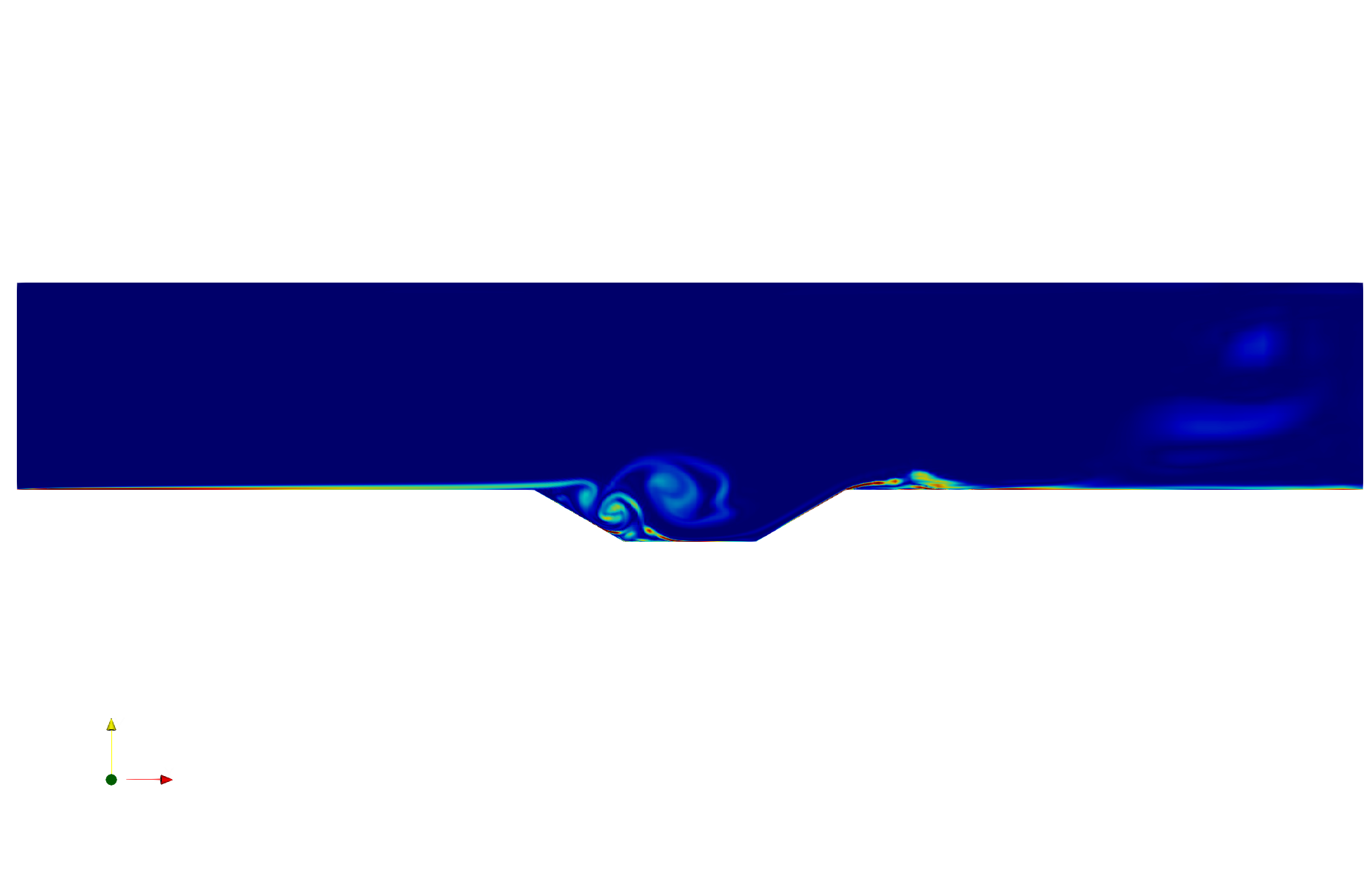}
\caption{$t=16.0$}
\end{subfigure}
\begin{subfigure}[t]{0.49\textwidth}
\includegraphics[trim={18cm 16cm 18cm 15cm},clip,width=1.0\linewidth]{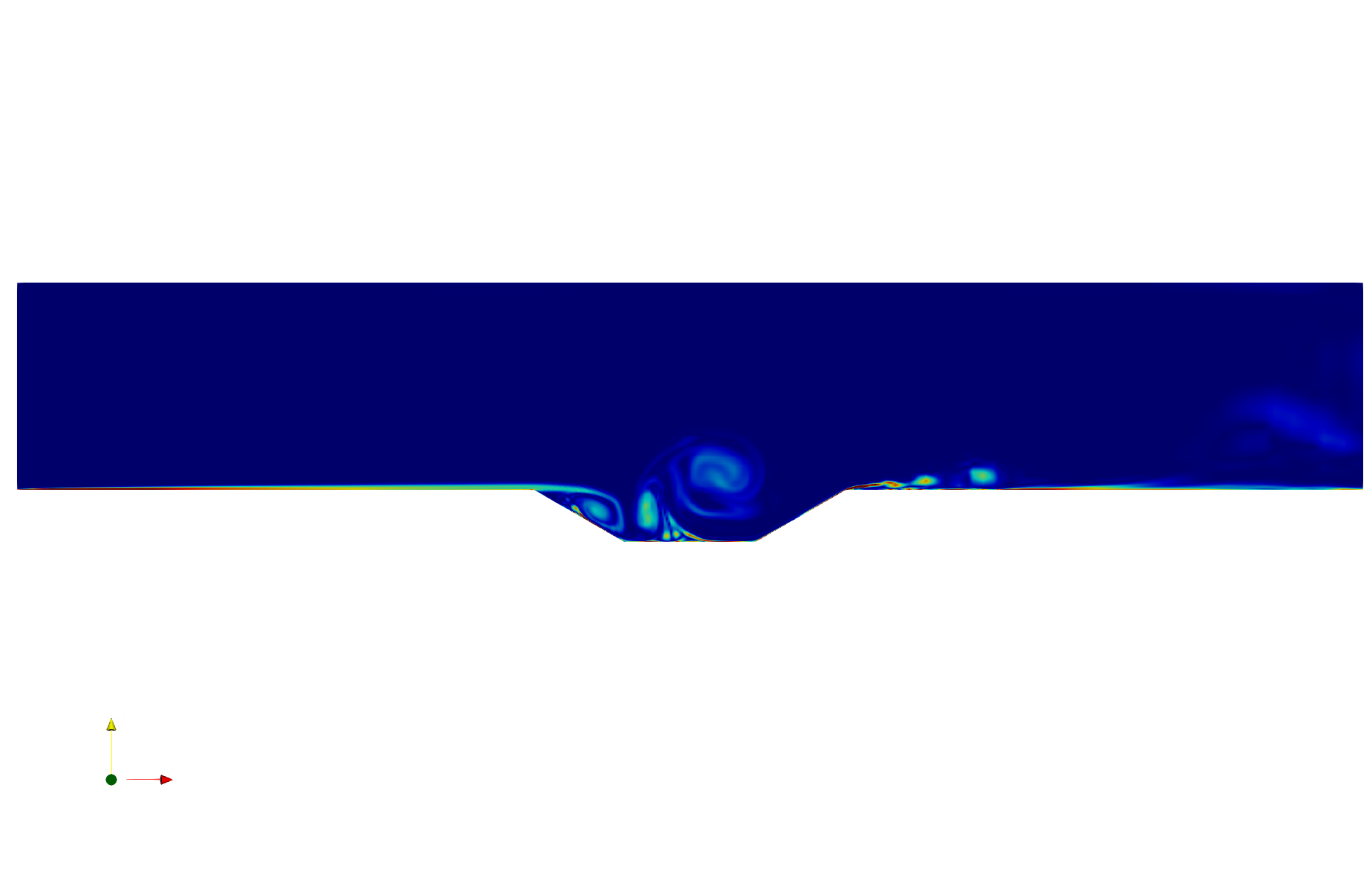}
\caption{$t=20.0$}
\end{subfigure}
\caption{Vorticity snapshots from the FOM simulation at various time instances.} 
\label{fig:fom_sols_cav}
\end{center}
\end{figure}

\subsubsection{Numerical results}
We first assess the performance of \methodAcronymROMs\ with varying window sizes. To this end, we consider \methodAcronymROMs\ with uniform window 
sizes of $\Delta T^n \equiv \Delta T = 0.2,0.5,1.0,$ and $2.0$, along with the Galerkin and LSPG ROMs. We first consider results for basis \#1 as described in 
Table~\ref{tab:rom_basis_details}. For all ROMs, we evolve 
the solution for $t \in [0,100]$. This comprises the same time interval used to construct the trial subspace. First, Figure~\ref{fig:cav_results1a} depicts the evolution of the pressure at the bottom wall in the midpoint of the computational domain, while Figure~\ref{fig:cav_results1b} depicts the evolution of the normalized $\elltwo$-error of the various reduced-order models. Both the collocated Galerkin and LSPG ROMs blow up/fail to converge within the first several time units. The \methodAcronymROM\ minimizing the residual over a window of size $\Delta T = 0.2$ also fails to converge. The \methodAcronymROMs\ that minimize the residual over window sizes of $\Delta T \ge 0.5$ are seen to all be stable and accurate; the pressure response is well characterized and the normalized state errors are less than $10\%$. The most notable discrepancy between the \methodAcronymROM\ and FOM solutions is a phase difference.

\begin{figure}
\begin{center}

\begin{subfigure}[t]{0.95\textwidth}
\includegraphics[trim={0cm 3.0cm 0cm 3cm},clip,width=1.\linewidth]{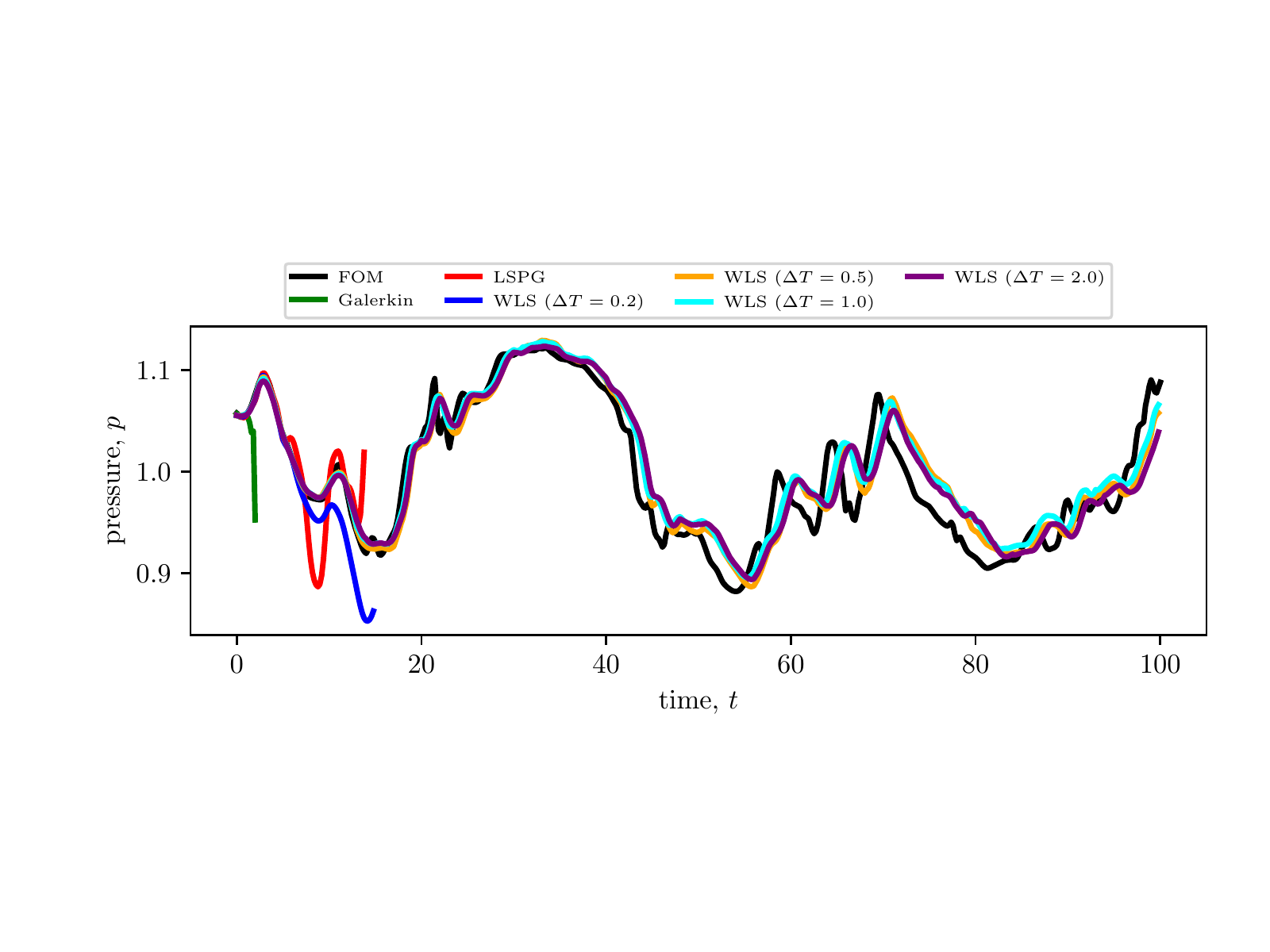}
\caption{Pressure} 
\label{fig:cav_results1a}
\end{subfigure}

\begin{subfigure}[t]{0.95\textwidth}
\includegraphics[trim={0cm 2.8cm 0cm 3cm},clip,width=1.\linewidth]{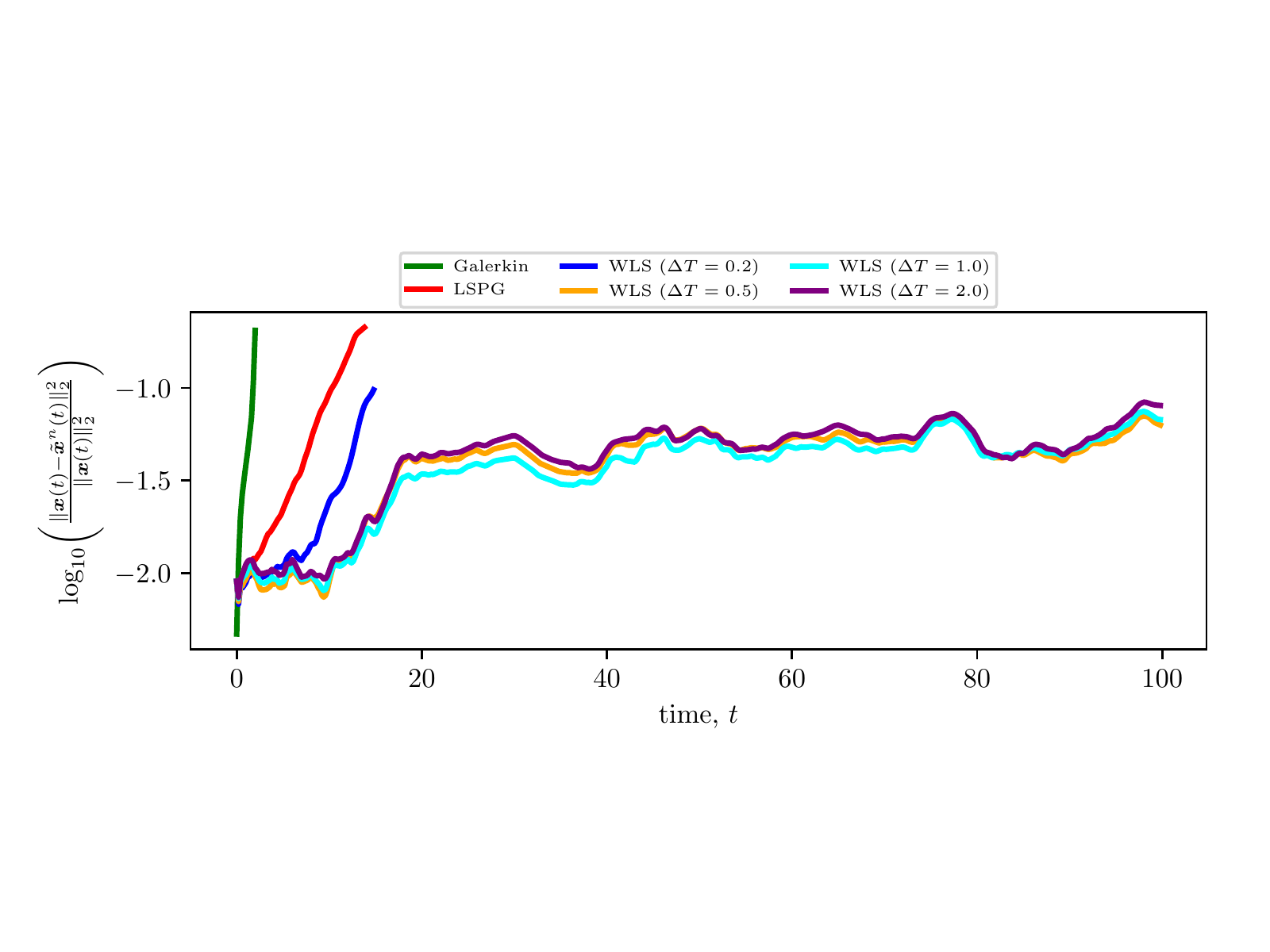}
\caption{Normalized $\elltwo$-error}
\label{fig:cav_results1b}
\end{subfigure}

\end{center}
\caption{Comparison of the pressure profiles obtained at the midpoint of the bottom wall (top) and normalized $\elltwo$-errors (bottom) of various collocated ROMs to the full-order model solution.}
\label{fig:cav_results1}
\end{figure}

Figure~\ref{fig:cav_results2} shows the space--time error and objective function of the stable \methodAcronymROMs. We observe that growing the window size over 
which the residual is minimized leads to a lower space--time residual, but not necessarily a lower $\elltwo$-error. This result is consistent with the previous numerical example. Next, Figure~\ref{fig:cav_wallclock}
shows the wall-clock times of the \methodAcronymROMs\ for $t \in [0,10]$ as compared to the LSPG ROM\footnote{We note that LSPG failed to converge at $t \approx 16.0$, so we focus on the first ten time units.}. As expected, increasing $\Delta T$ again leads to an increase in computational cost; minimizing the residual over a window comprising 20 time instances yields a 2.5x increase in cost over LSPG.

\begin{figure}
\begin{center}
\begin{subfigure}[t]{0.45\textwidth}
\includegraphics[trim={0cm 0cm 0cm 0cm},clip,width=1.\linewidth]{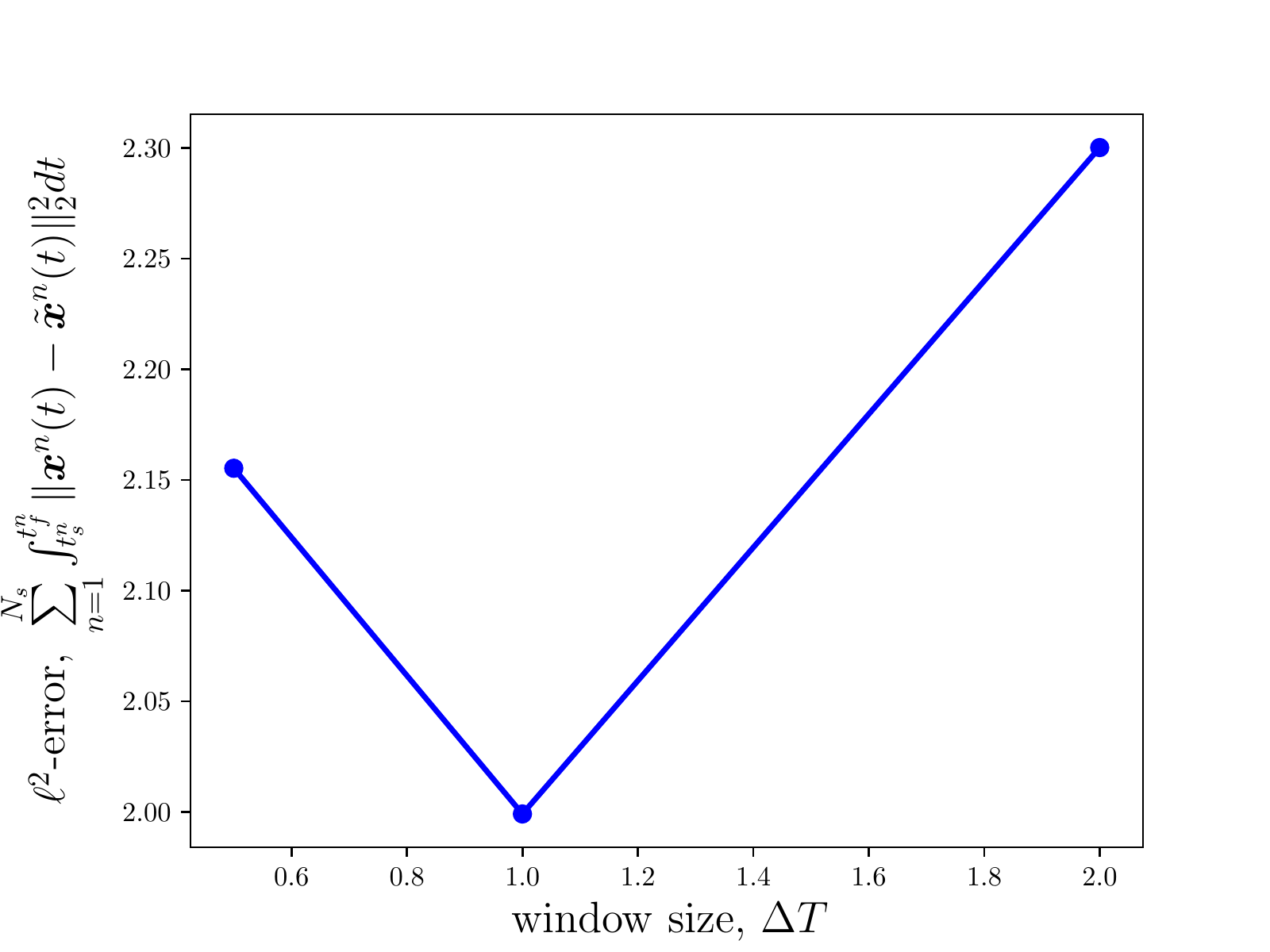}
\caption{Space--time $\elltwo$-error}
\label{fig:cav_results2a}
\end{subfigure}
\begin{subfigure}[t]{0.45\textwidth}
\includegraphics[trim={0cm 0cm 0cm 0cm},clip,width=1.\linewidth]{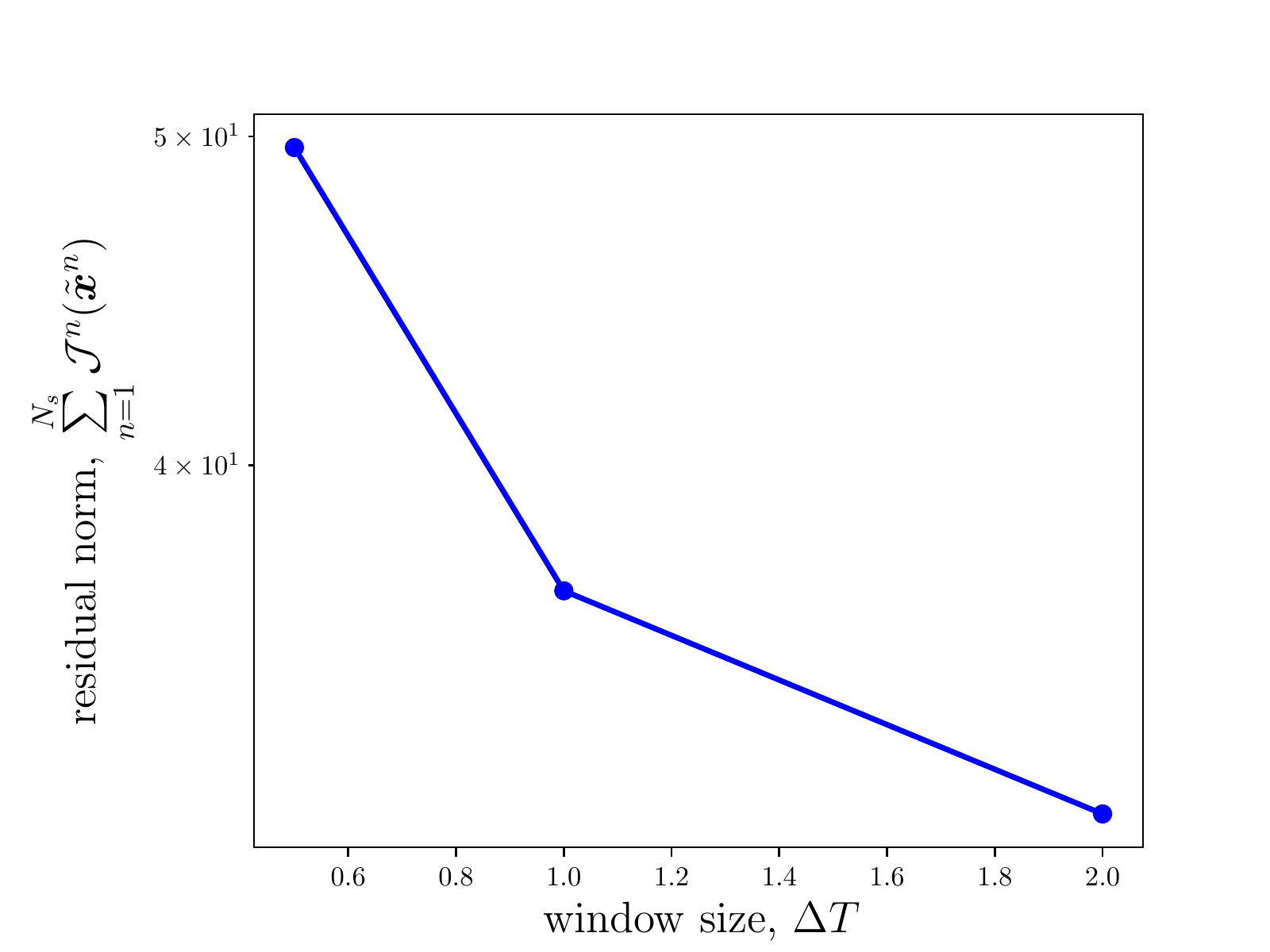}
\caption{Objective function} 
\label{fig:cav_results2b}
\end{subfigure}
\end{center}
\caption{Space--time error (left) and objective function (right) as a function of window size.}
\label{fig:cav_results2}
\end{figure}

\begin{figure}
\begin{center}
\includegraphics[trim={0cm 0cm 0cm 0cm},clip,width=0.49\linewidth]{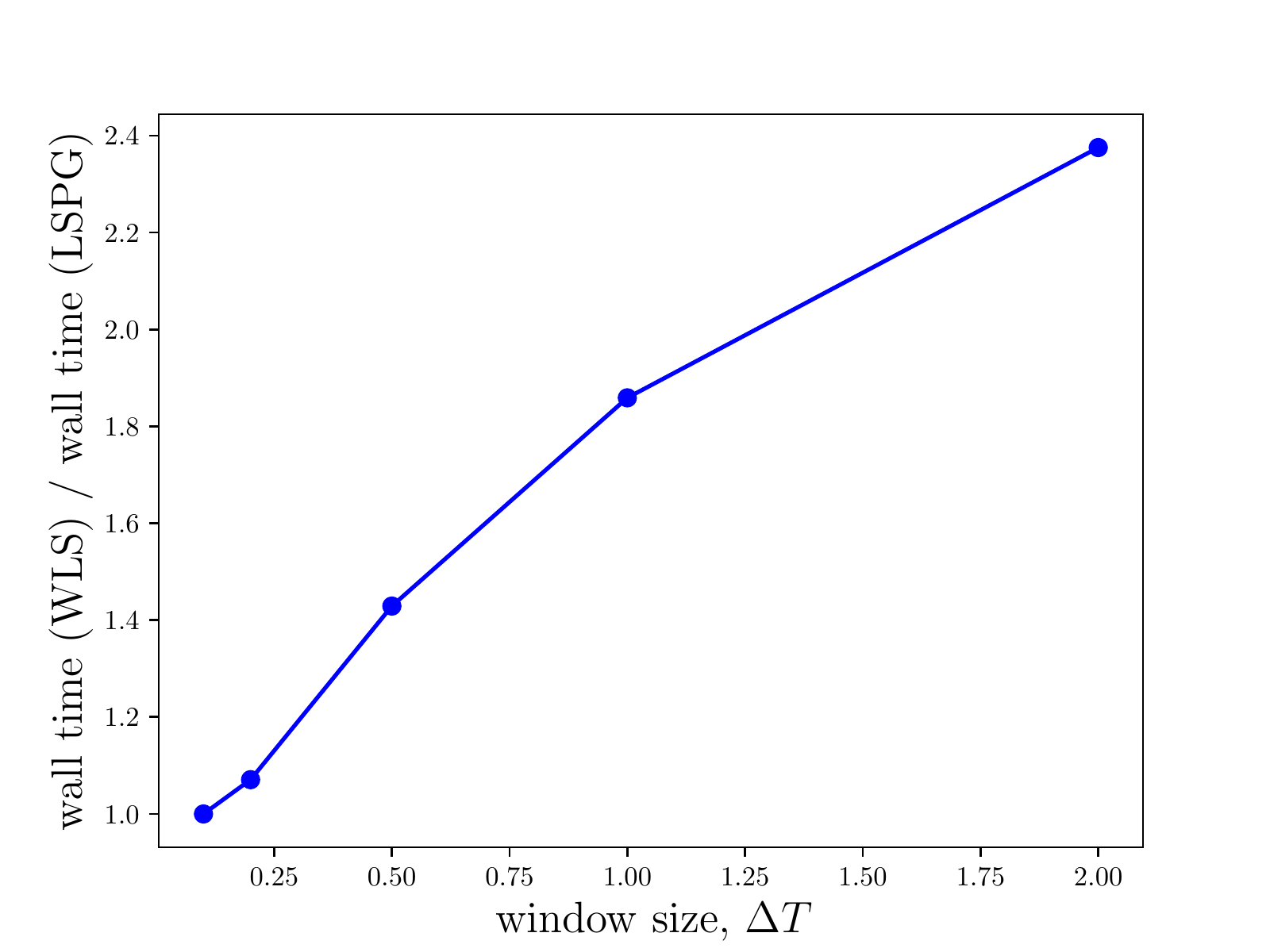}
\caption{Wall-clock times of \methodAcronymROMs\ with respect to the LSPG ROM.}
\label{fig:cav_wallclock}
\end{center}
\end{figure}

Figure~\ref{fig:cav_snapshots} shows vorticity fields for the FOM, LSPG ROM, and \methodAcronymROMs\ with $\Delta T = 0.5,1.0$ for the time instance $t = 5.0$. LSPG is observed to exhibit artificial oscillations; the Gauss-Newton method fails to converge at $t \approx 16.0$. The \methodAcronymROMs\ at $\Delta T = 0.5,1.0$ are able to capture the important features of the flow, including the points of flow separation at the start and end of the ramp, and remain stable for the entire time interval.  


\begin{figure}
\begin{center}
\begin{subfigure}[t]{0.49\textwidth}
\includegraphics[trim={10cm 7cm 10cm 12cm},clip,width=1.0\linewidth]{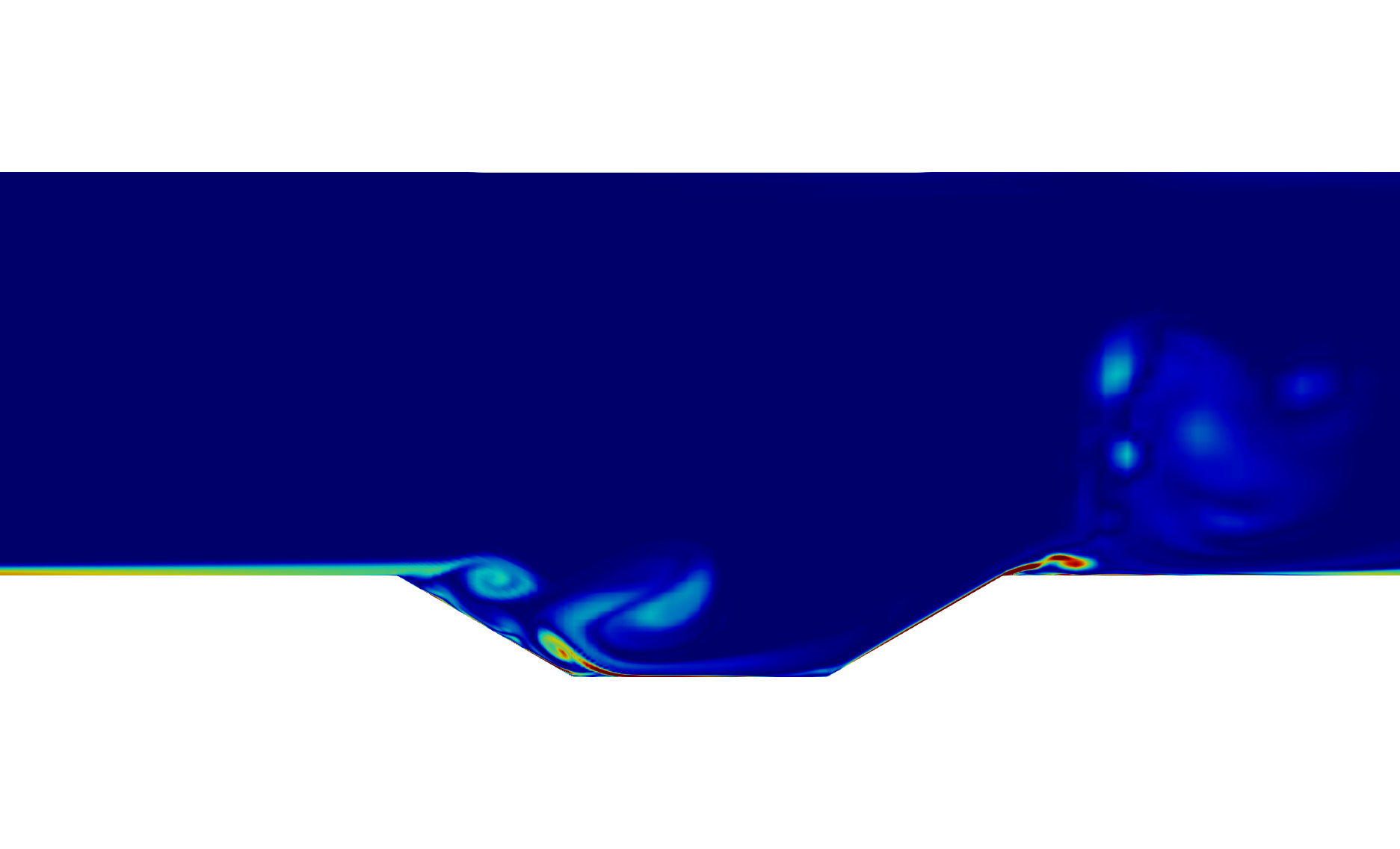}
\caption{FOM}
\end{subfigure}
\begin{subfigure}[t]{0.49\textwidth}
\includegraphics[trim={10cm 7cm 10cm 12cm},clip,width=1.0\linewidth]{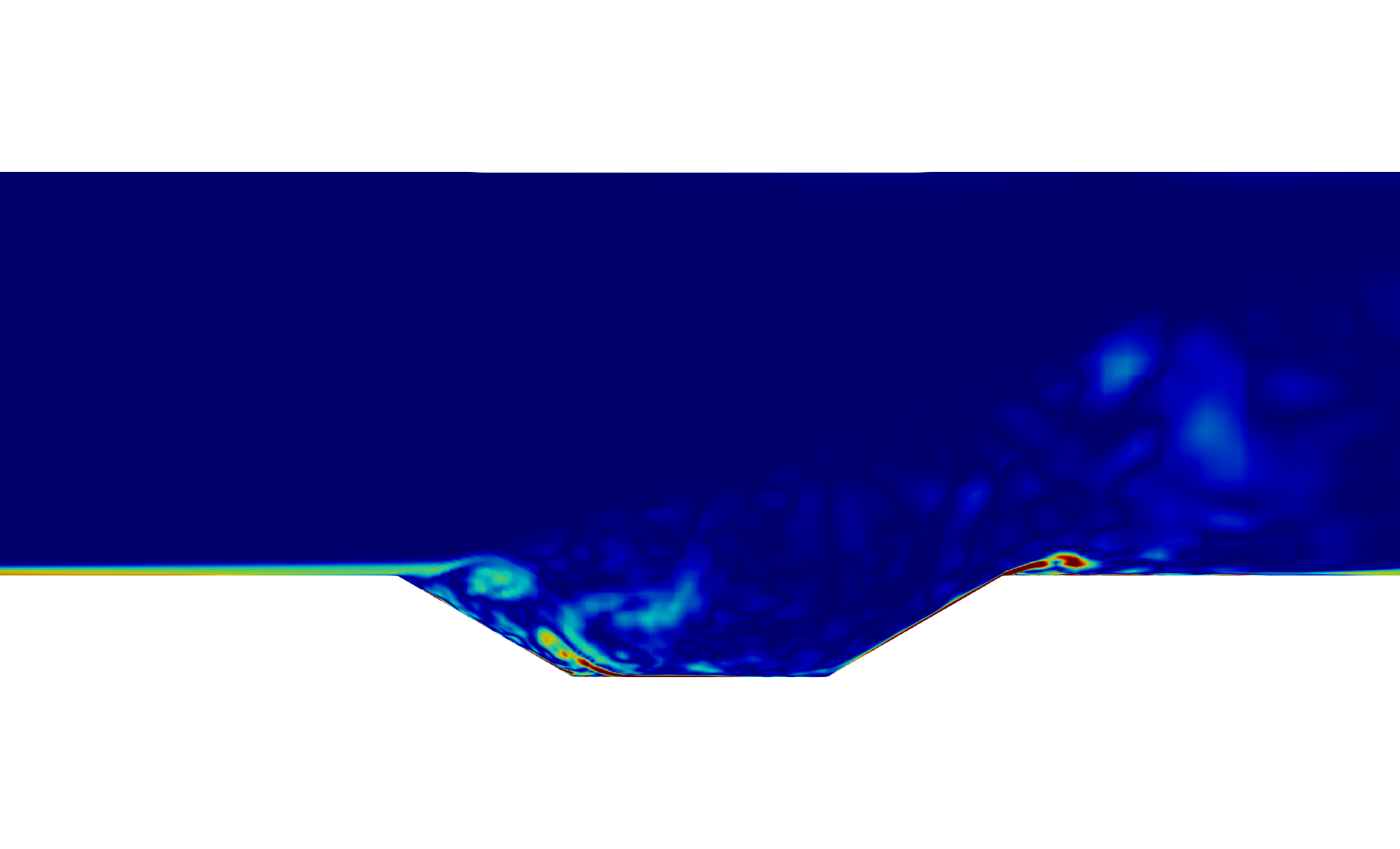}
\caption{LSPG}
\end{subfigure}
\begin{subfigure}[t]{0.49\textwidth}
\includegraphics[trim={10cm 7cm 10cm 12cm},clip,width=1.0\linewidth]{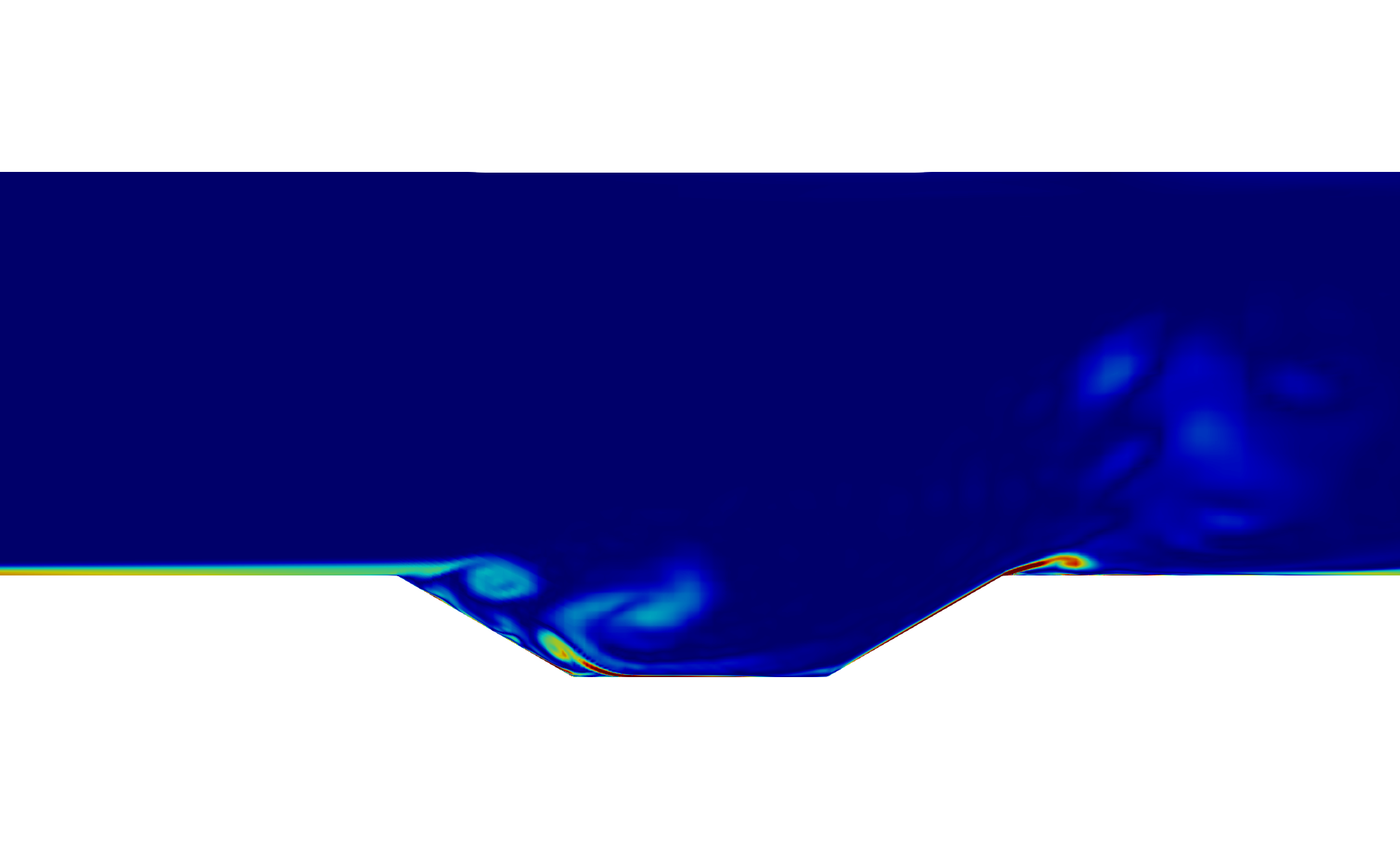}
\caption{\methodAcronym, $\Delta T = 0.5$}
\end{subfigure}
\begin{subfigure}[t]{0.49\textwidth}
\includegraphics[trim={10cm 7cm 10cm 12cm},clip,width=1.0\linewidth]{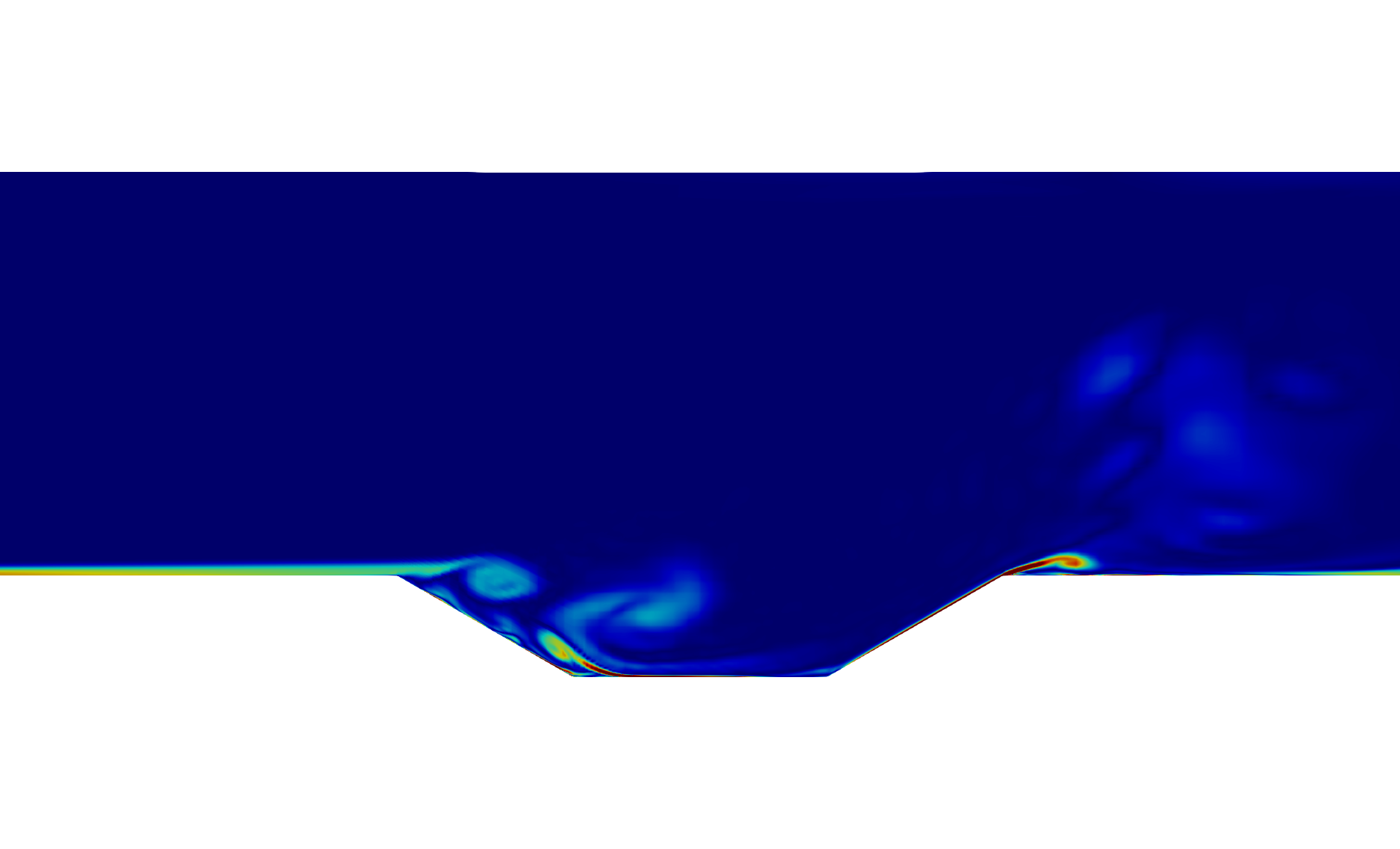}
\caption{\methodAcronym, $\Delta T = 1.0$}
\end{subfigure}
\caption{Vorticity snapshots from the FOM and ROM simulations at $t=5.0$.} 
\label{fig:cav_snapshots}
\end{center}
\end{figure}

Next, we assess the performance of the various ROMs for basis \#2 as described in Table~\ref{tab:rom_basis_details}, which comprises a richer spatial basis.
 Figure~\ref{fig:cav_results1_basis1} shows the same pressure and error profiles as Figure~\ref{fig:cav_results1}, but for the enriched basis. The LSPG and Galerkin ROMs blow up faster as compared to Figure~\ref{fig:cav_results1}; LSPG fails to converge around $t \approx 8$ (opposed to $t \approx 16$), while Galerkin blows up almost immediately. The 
\methodAcronymROMs\ again yield improved performance: \methodAcronymROMs\ minimizing the residual over window sizes of $\Delta T \ge 0.5$ are seen to all be stable and accurate; the pressure response is well characterized and the normalized state errors are less than $5\%$. The \methodAcronymROMs\ employing basis \#2 yield more accurate results than \methodAcronymROMs\ employing basis \#1.

\begin{figure}
\begin{center}

\begin{subfigure}[t]{0.95\textwidth}
\includegraphics[trim={0cm 2.5cm 0cm 2.5cm},clip,width=1.\linewidth]{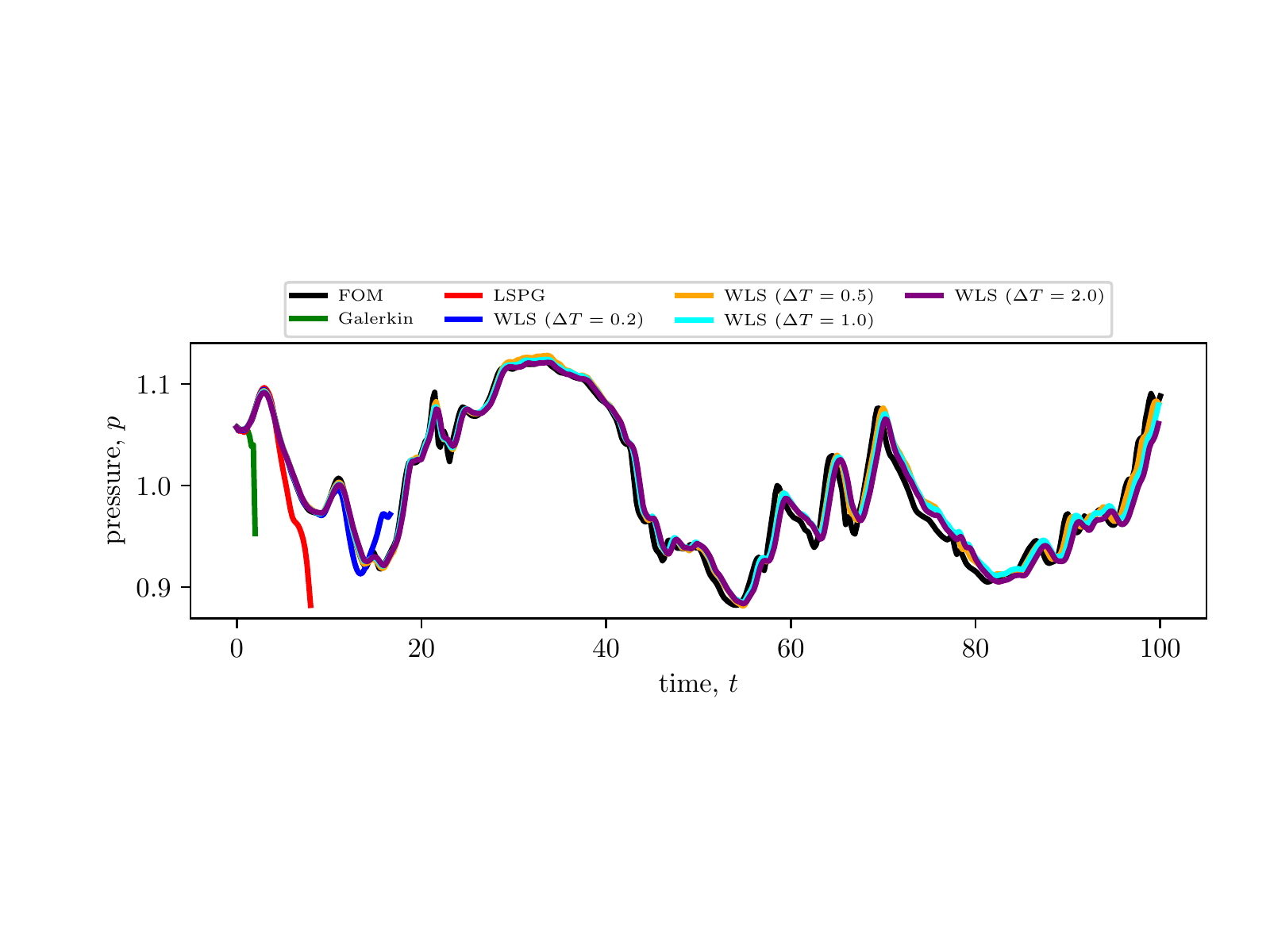}
\caption{Pressure} 
\label{fig:cav_results1a_basis1}
\end{subfigure}

\begin{subfigure}[t]{1.0\textwidth}
\includegraphics[trim={0cm 2.5cm 0cm 3cm},clip,width=1.\linewidth]{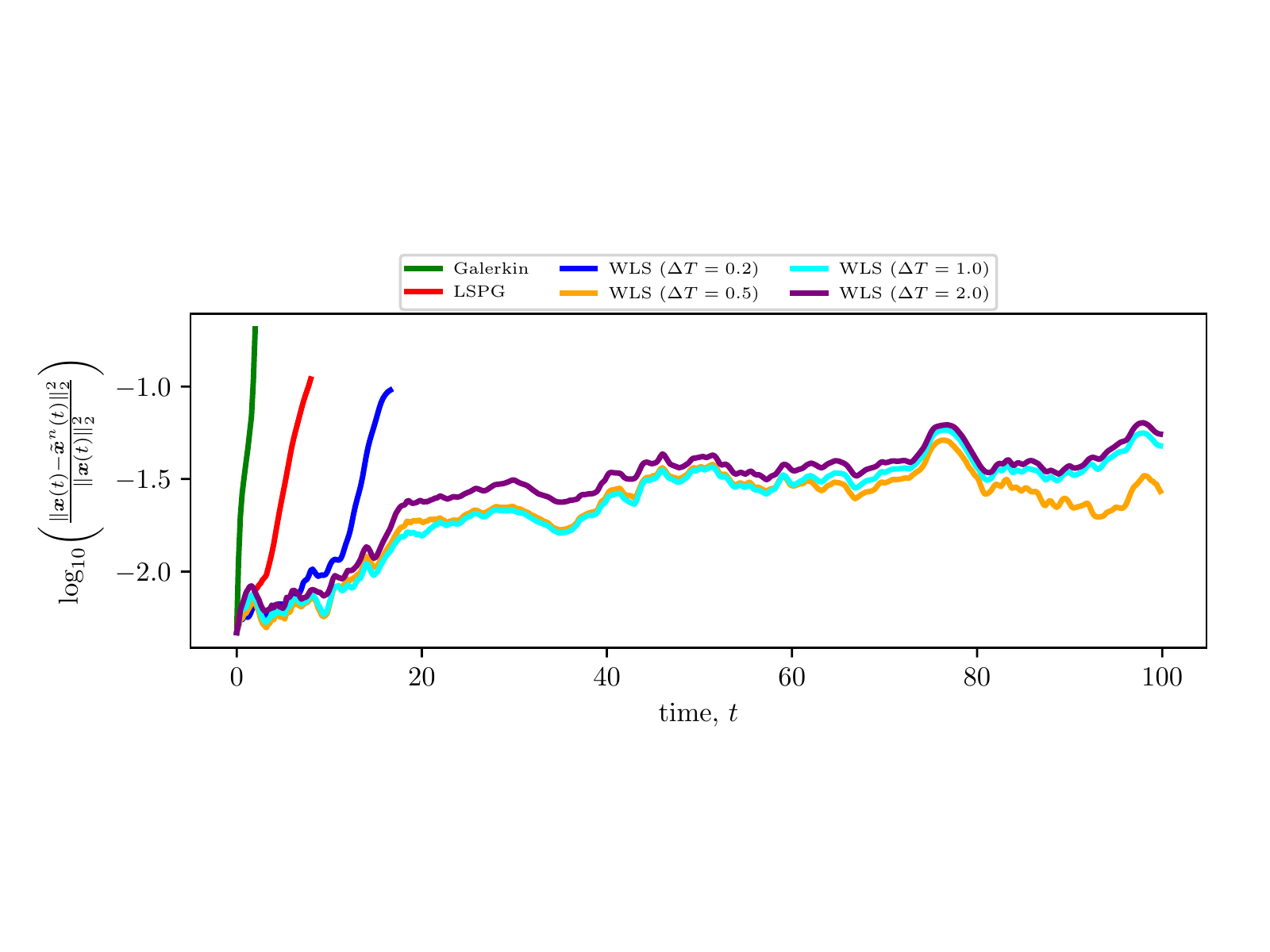}
\caption{Normalized $\elltwo$-error}
\label{fig:cav_results1b_basis1}
\end{subfigure}
\end{center}
\caption{Comparison of the pressure profiles obtained at the midpoint of the bottom wall (top) and normalized $\elltwo$-errors (bottom) of various collocated ROMs to the full-order model solution.}
\label{fig:cav_results1_basis1}
\end{figure}

\begin{figure}
\begin{center}
\begin{subfigure}[t]{0.45\textwidth}
\includegraphics[trim={0cm 0cm 0cm 0cm},clip,width=1.\linewidth]{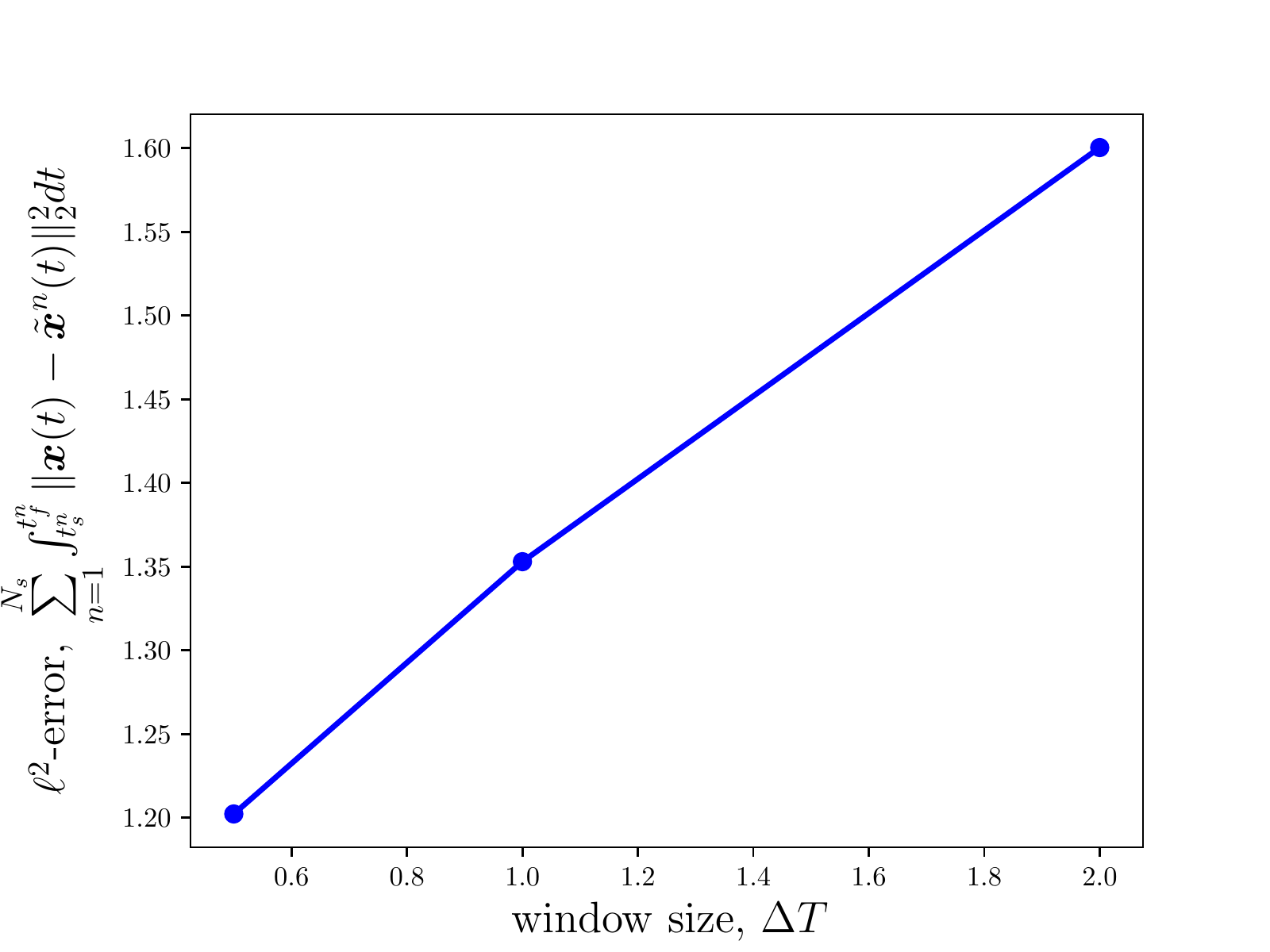}
\caption{Integrated $\elltwo$-error}
\label{fig:cav_results3a}
\end{subfigure}
\begin{subfigure}[t]{0.45\textwidth}
\includegraphics[trim={0cm 0cm 0cm 0cm},clip,width=1.\linewidth]{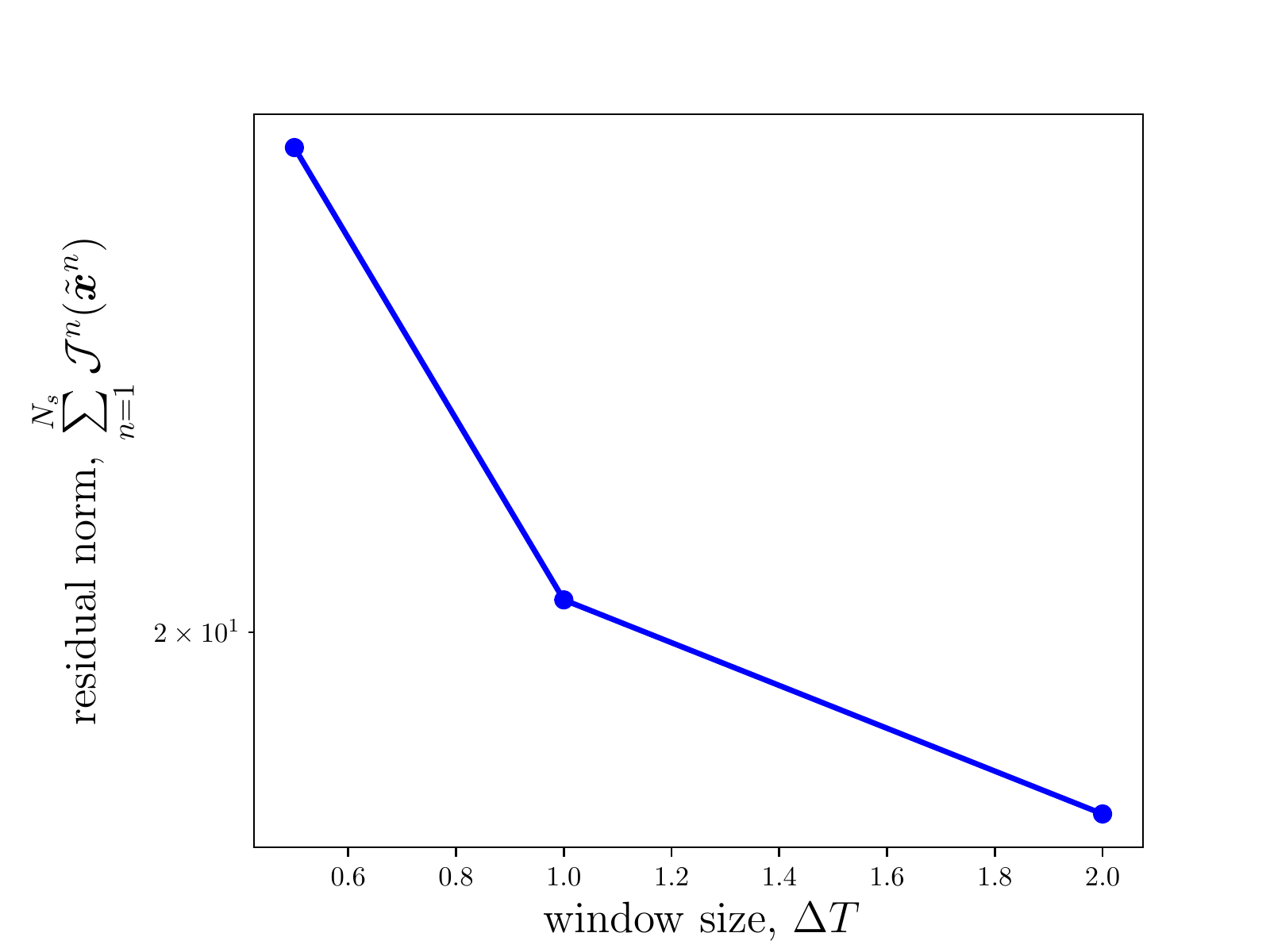}
\caption{Objective function} 
\label{fig:cav_results3b}
\end{subfigure}
\end{center}
\caption{Integrated error (left) and objective function (right) as a function of window size.}
\label{fig:cav_results3}
\end{figure}

Finally, Figure~\ref{fig:cav_wallclock} shows the wall-clock times for $t \in [0,4]$ of the \methodAcronymROMs\ as compared to the LSPG ROMs for basis \#2. Increasing the window size again leads to an increase in computational cost. Minimizing the residual over a window comprising 20 time instances yields a 3x increase in cost over LSPG. 
\begin{figure}
\begin{center}
\includegraphics[trim={0cm 0cm 0cm 0cm},clip,width=0.49\linewidth]{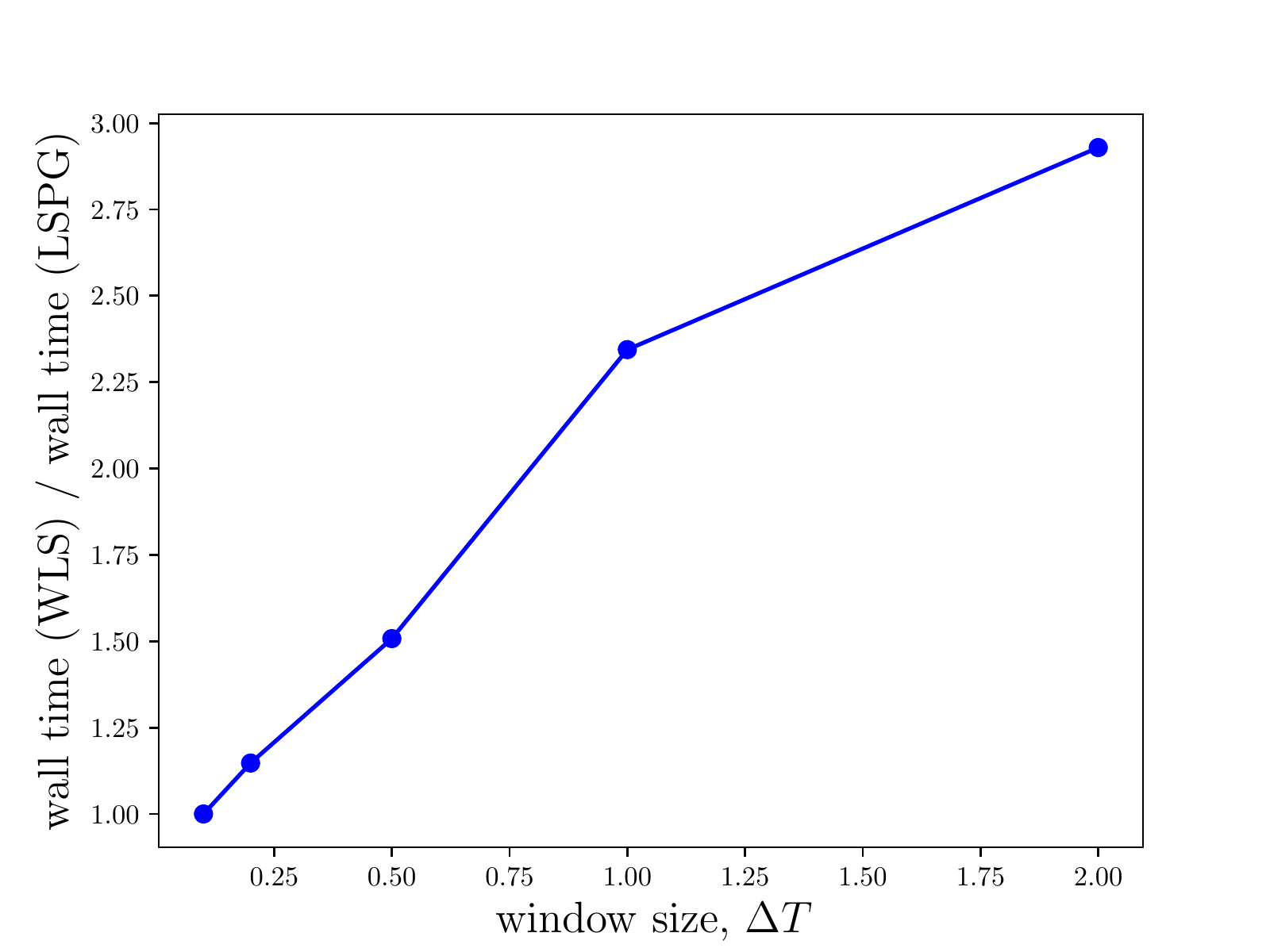}
\caption{Wall-clock times of \methodAcronymROMs\ with respect to the LSPG ROM.}
\label{fig:cav_wallclock}
\end{center}
\end{figure}
